\documentclass{article}
\usepackage{fullpage}
\usepackage{amsmath, amssymb}
\usepackage{mathtools}
\usepackage{amsthm}
\usepackage{enumerate}
\usepackage{url}
\usepackage{makeidx}
\usepackage{multirow}
\usepackage{caption}
\usepackage{subcaption}

\usepackage{tikz}
\usetikzlibrary{decorations.markings}
\pgfdeclarelayer{background}
\pgfdeclarelayer{nodelayer}
\pgfdeclarelayer{edgelayer}
\pgfsetlayers{background,edgelayer,nodelayer,main}

\tikzstyle{new}=[circle,  minimum width=4pt,inner sep=0pt, fill=black,draw=black]
\tikzstyle{none}=[circle,fill=white,draw=black]
\tikzstyle{n}=[shape=rectangle,minimum width=1pt,inner sep=0pt, fill=none,draw=none]
\tikzstyle{emph}=[circle,  minimum width=4pt,inner sep=0pt, fill=magenta,draw=magenta]

\tikzset{directed/.style={decoration={
  markings,
  mark=at position .66 with {\arrow{>}}},postaction={decorate}}}
\tikzset{directed2/.style={decoration={
  markings,
  mark=at position .6875 with {\arrow{>}},
  mark=at position .7125 with {\arrow{>}}},postaction={decorate}}}
\tikzset{directed3/.style={decoration={
  markings,
  mark=at position .7 with {\arrow{>}}},postaction={decorate}}}

\usepackage{array}
\usepackage{float}
\usepackage{wrapfig}

\usepackage{graphicx}
\usepackage{rotating}
\usepackage{nicefrac}
\usepackage{longtable}
\usepackage{xcolor}

\usepackage[overload]{textcase} 
\makeatletter
\def\cleardoublepage{
  \clearpage
  \if@twoside\ifodd\c@page\else
  \hbox{}
  \thispagestyle{empty}
  \newpage
  \if@twocolumn\hbox{}\newpage\fi
  \fi\fi
  }
\makeatother
\usepackage{hyperref}

\newtheoremstyle{plainsl}%
	{\topsep}
	{\topsep}
	{\slshape} 
	{}
	{\normalfont\bfseries}
	{.}
	{ }
	{}

\swapnumbers

{\theoremstyle{plainsl}
\newtheorem{theorem}{Theorem}[section]
\newtheorem{lemma}[theorem]{Lemma}
\newtheorem{proposition}[theorem]{Proposition}

\newtheorem{openprob}[theorem]{Open Problem}
}
{\theoremstyle{remark}
\newtheorem{remark}[theorem]{Remark}
}

\newcommand\cref[1]{Corollary~\ref{cor:#1}}

\numberwithin{equation}{section}

\newcommand\comp[1]{{\mkern2mu\overline{\mkern-2mu#1}}}

\DeclareMathOperator\ev{ev}

\DeclareMathOperator\tr{tr}

\DeclareMathOperator{\boxprod}{\mathbin{\Box}}

\DeclareMathOperator{\lcm}{lcm}

\newcommand\cx{{\mathbb C}}

\newcommand\ints{{\mathbb Z}}
\newcommand\re{{\mathbb R}}
\newcommand\rats{{\mathbb Q}}
\newcommand\cA{{\mathcal A}}

\newcommand\cG{{\mathcal G}}

\newcommand\Ze{{\mathbf e}}

\newcommand\Zv{{\mathbf v}}

\newcommand\ones{{\mathbf 1}}

\newcommand\revarc[1]{{\mkern5mu\overleftarrow{\mkern1mu#1}}}

\newcommand{\slimmod}[1]{\ (\mathrm{mod}\ #1)}

\definecolor{vcolour}{RGB}{230,97,0}
\definecolor{kcolour}{RGB}{93,58,155}

\newcommand\krystalsays[1]{}
\newcommand\vincentsays[1]{}

\newcommand{\arxiv}[1]{\href{https://arxiv.org/abs/#1}{\texttt{arXiv:#1}}}

\title{State transfer in discrete-time quantum walks \\ via projected transition matrices}
\author{Krystal Guo \and Vincent Schmeits }
\date{January 13, 2025}

\begin{document}

\maketitle

\begingroup
    \renewcommand\thefootnote{}
    \footnotetext[0]{Korteweg-de Vries Institute for Mathematics, University of Amsterdam, Amsterdam, The Netherlands.  \href{https://qusoft.org/}{QuSoft} (Research center for quantum software \& technology), Amsterdam, The Netherlands.  (\texttt{\string{k.guo,v.f.schmeits\string}@uva.nl}). }
\endgroup
\begin{abstract}

In this paper, we analyze state transfer in quantum walks by using combinatorial methods. We generalize perfect state transfer in two-reflection discrete-time quantum walks to a notion that we call \textsl{peak state transfer};  we define peak state transfer as the highest state transfer that can be achieved between an initial and a target state under unitary evolution, even when perfect state transfer is unattainable. We give a spectral characterization of peak state transfer that allows us to fully characterize peak state transfer in the arc-reversal (Grover) walk on various families of graphs, including strongly regular graphs and incidence graphs of block designs (assuming that the walk starts at a point of the design). In addition,  we  provide many examples of peak state transfer, including an infinite family where the amount of peak state transfer tends to $1$ as the number of vertices grows. We further demonstrate that peak state transfer properties extend to infinite families of graphs generated by vertex blow-ups, and we characterize periodicity in the vertex-face walk on toroidal grids. In our analysis, we make extensive use of the spectral decomposition of a matrix that is obtained by projecting the transition matrix down onto a subspace. Though we are motivated by a problem in quantum computing, we identify several open problems that are purely combinatorial, arising from the spectral conditions required for peak state transfer in discrete-time quantum walks.


\end{abstract}

\tableofcontents

\section{Introduction}\label{sec:intro}

In this paper, we analyze discrete-time quantum walks using combinatorial methods. Although this problem was originally motivated by quantum algorithms, it has since evolved into a recognized topic in algebraic graph theory; for example, a recent survey on eigenvalues of Cayley graphs \cite{LiuZho2022} and the dynamic survey on distance-regular graphs \cite{vDamKooTan2016} both dedicate sections on perfect state transfer in quantum walks.

A \textsl{discrete-time quantum walk} is represented by a unitary operator that is used to evolve some initial unit-length vector, or \textsl{(quantum) state}, in discrete time steps. Colloquially speaking, given such an initial state, we are interested in how close this unitary evolution can bring it to another given state. We defer the formal definitions until Section \ref{sec:2-reflection-examples} and will first illustrate our motivation with Figure \ref{fig:peak_ST_illustration}, which contains an example of the evolution of a discrete-time quantum walk with some initial state $\phi_1$.

\begin{figure}[htb]
\centering
\begin{subfigure}{0.3\textwidth}
    \centering
    \includegraphics[width=\textwidth]{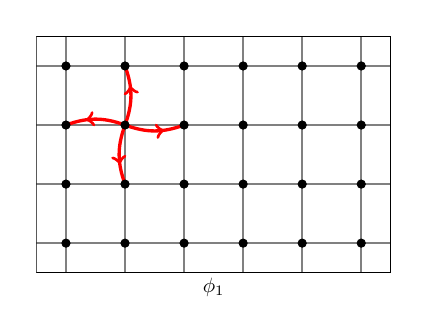}
\end{subfigure}
\begin{subfigure}{0.3\textwidth}
    \centering
    \includegraphics[width=\textwidth]{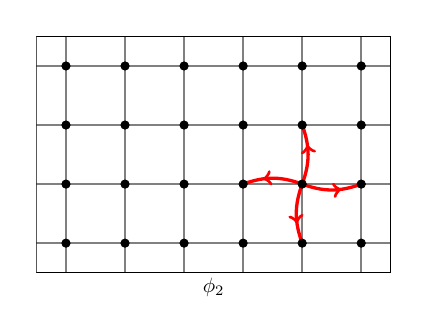}
\end{subfigure}

\begin{subfigure}{0.26\textwidth}
    \centering
    \includegraphics[width=\textwidth]{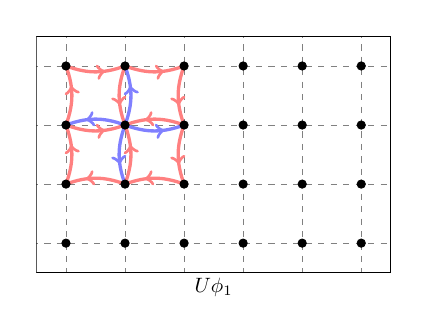}
\end{subfigure}
\kern-1em
\begin{subfigure}{0.26\textwidth}
    \centering
    \includegraphics[width=\textwidth]{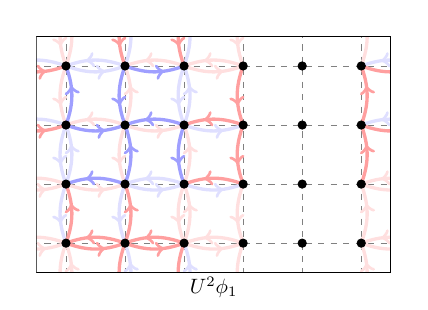}
\end{subfigure}
\kern-1em
\begin{subfigure}{0.26\textwidth}
    \centering
    \includegraphics[width=\textwidth]{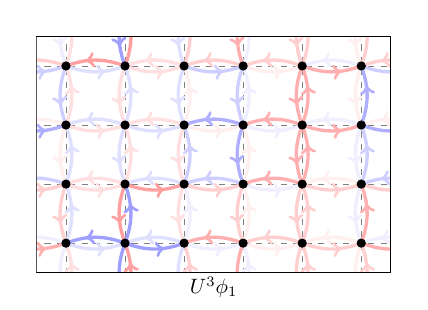}
\end{subfigure}
\kern-1em
\begin{subfigure}{0.26\textwidth}
    \centering
    \includegraphics[width=\textwidth]{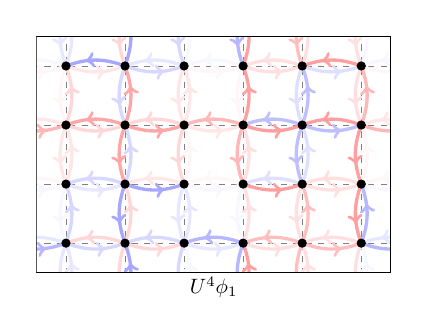}
\end{subfigure}
\caption{The pictures in the top row show two quantum states; each state $\phi_i$ represents a vector indexed by a set of $96$ arcs. Its value is $0$ on every arc, except for the $4$ red-colored arcs, on which it has value $0.5$. The second row depicts the first four steps of a discrete-time quantum walk, where red (resp.\ blue) arcs represent positive (resp.\ negative) entries of the corresponding states. The level of opaqueness of each arc reflects the magnitude of the corresponding entry, making the arc invisible whenever the entry is equal to zero.\label{fig:peak_ST_illustration}}
\end{figure}

Loosely speaking, the main question we wish to ask is, given a starting state $\phi_1$ and a target state $\phi_2$, how close can $U^{t}\phi_1$ come to $\phi_2$? If there is a time $\tau$ such that $U^{\tau}\phi_1$ is equal to $\phi_2$ (or some other scalar multiple thereof), we say that there is \textsl{perfect state transfer} from $\phi_1$ to $\phi_2$. This is a well-studied phenomenon in quantum physics, see for example \cite{KurWoj2011, YalGed2015, SteSko2016}. However, it is understood that this is a rare phenomenon. Thus, even if perfect state transfer does not occur, it makes sense to instead  ask how close the walk can come to the state $\phi_2$. In the example depicted in  Figure \ref{fig:peak_ST_illustration}, the state $U^3\phi_1$ is as close to the state $\phi_2$ as the walk can get, in the sense that $t = 3$ and $\gamma = 1$ minimize the quantity
\begin{equation}\label{eq:intro-peak-quantity}
  m(t,\gamma) \coloneq \|U^t\phi_1 - \gamma\phi_2\| .
\end{equation}
This is an instance of what we call \textsl{peak state transfer}. 
Broadly speaking, $U^{t}\phi_1$ is prevented from reaching $\phi_2$ because $ m(t,\gamma)$ is lower-bounded in terms of spectral properties of $U$; peak state transfer occurs when this lower bound is met with equality.  In Section \ref{sec:peakst}, we will make this definition precise and we will see that perfect state transfer is a special case of peak state transfer, where the lower bound for $m(t,\gamma)$ is equal to $0$.

The graph-theoretic properties behind perfect state transfer in continuous-time quantum walks\footnote{For a \textsl{continuous-time quantum walk}, the transition matrix $U(t)$ (which replaces $U^t$ in the discrete case) depends continuously on $t$: typically, $U(t) = \exp(itA)$ where $A$ is the adjacency matrix of a graph.} have been well-studied, see e.g.\ \cite{CheGod2011,CouGodGuo2015,GodGuoKem2020}.  Perfect state transfer has also been studied in discrete-time walks. There are multiple models of discrete-time quantum walks; here we focus on walks where the unitary operator is a product of two reflections, which we will call \textsl{two-reflection walks}; we defer the formal definition until Section \ref{sec:2reflectionwalks} and note for now that two examples of two-reflection walks are the \textsl{arc-reversal walk} and the \textsl{vertex-face walk}. 
Perfect state transfer was characterized for the arc-reversal walk by Zhan \cite{Zha2019} for regular graphs and  for graphs by Kubota and Yoshino \cite{KubYos2024}. In a previous paper \cite{GuoSch2024}, the authors built the mathematical foundations for studying  perfect state transfer in vertex-face walks. 

If perfect state transfer does not occur, but instead for every $\epsilon > 0$ there exist $t$ and $\gamma$ such that $m(t,\gamma) < \epsilon$, we say that there is \textsl{pretty good state transfer} from $\phi_1$ to $\phi_2$. For discrete-time walks, this notion has been studied by Chan \& Zhan \cite{ChaZha2023}, though it is more well-studied in continuous-time walks, in particular on path graphs \cite{GodKirSev2012,VinZhe2012,BanCouGod2017,CouGuoBom2017,Bom2019,ChaSin2023}. Both pretty good state transfer and peak state transfer are generalizations of perfect state transfer. Colloquially speaking, with pretty good state transfer, the walk comes arbitrarily close to some $\gamma \phi_2$, but there is no good handle on the time that it takes for the walk to reach $\gamma \phi_2$ within a given distance; the characterizations of pretty good state transfer (for both discrete-time and continuous-time) rely on Kronecker's theorem, the proof of which is non-constructive. This in contrast to perfect state transfer, where the times at which it occurs can be deduced from the characterization given by Kubota and Yoshino \cite{KubYos2024}. As is the case with perfect state transfer, if one can show that there is peak state transfer between a pair of vertices, it is easy to determine the times at which it occurs. One could say that peak state transfer is a relaxation of perfect state transfer on the \textsl{amount} of state transfer (in the sense that we do not require $m(t,\gamma)$ to attain the value $0$), whereas pretty good state transfer is a relaxation on the \textsl{time} of state transfer (in the sense that we cannot (easily) deduce the times at which the amount of state transfer is high).

In this paper, we define and give a characterization of peak state transfer (Theorem \ref{thm:PeakSTchar}). We do this by using ideas from the continuous-time case to study a matrix that is obtained from $U$ by projecting it down onto the state space. This matrix, that we call the \textsl{projected transition matrix}, was studied before in the context of perfect state transfer in \cite{KubSeg2022} for arc-reversal walks and separately in \cite{GuoSch2024} for vertex-face walks (the definitions of which are given in Section \ref{sec:prelim-projected}). The characterization of peak state transfer includes a characterization of perfect state transfer and differs from the result in \cite{KubYos2024} only in that a condition called `strong cospectrality' need not be met. It is also formulated more generally for all two-reflection walks and the formulation is time-independent: we characterize whether or not peak state transfer occurs between distinct states, and in the case that it does, we provide expressions for the times at which it occurs. We also characterize when there is `peak state transfer' from a state to itself, which turns out to be equivalent to the walk being periodic at that state. Like in the case of peak state transfer, this characterization is time-independent.

In the remaining sections of this paper, we will use these characterizations as tools in order to classify peak state transfer in the arc-reversal walk on complete graphs, cycles, strongly regular graphs and incidence graphs of block designs (for walks starting in a point of the design). We also show that every instance of peak state transfer in the arc-reversal walk can be extended to an infinite family of graphs that admit peak state transfer, by blowing up the vertices of the graph. Additionally, we provide a family of graphs on which the arc-reversal walk has the property that the minimum of $m(t,\gamma)$ tends to zero as the number of vertices grows, without actually reaching it.

Moreover, we classify periodicity in the vertex-face walks on grids embedded in the torus (answering a question that was left open in \cite{GuoSch2024}) and we give an infinite family of toroidal grids that admit peak state transfer (but no perfect state transfer).

The organization of this paper is as follows. In Section \ref{sec:prelim}, we set up definitions and notation for two-reflection walks, and give the definitions of two commonly  studied two-reflection walks: the arc-reversal walk and the vertex-face walk. In Section \ref{sec:peakst}, we give the characterizations of peak state transfer (Theorem \ref{thm:PeakSTchar}) and periodicity (Theorem \ref{thm:periodicity_characterisation}). In Section \ref{sec:lex_products_arc_reversal}, we classify peak state transfer in the arc-reversal walks on complete graphs and cycles (Lemmas \ref{lem:cycles_peakST} and \ref{lem:complete_graphs_peakST}), and we show how peak state transfer is preserved when we blow up the vertices of a graph (Lemma \ref{lem:lex_product_peakST}). In Section \ref{sec:examples}, we discuss more examples of peak state transfer in arc-reversal walks: in particular, we classify peak state transfer in strongly regular graphs (Theorem \ref{thm:SRGs_peakST}) and in the point-block incidence graphs of block designs (Theorem \ref{thm:block_designs_peakST}). In Section \ref{sec:grids}, we classify periodicity in vertex-face walks on toroidal grids (Theorem \ref{thm:toroidal_periodicity_characterisation}), and we give an infinite family of grids for which the vertex-face walk admits peak state transfer (Theorem \ref{thm:4n_grids_peakST}). In Section \ref{sec:openproblems}, we discuss future directions of research and we describe several purely combinatorial open problems that have arisen from our study of quantum walks.

\section{Preliminaries}\label{sec:prelim} 
In Section \ref{sec:2reflectionwalks}, we will give preliminary definitions regarding two-reflection walks. In a setting which is common to the vertex-face walk and the arc-reversal walk, studying state transfer is equivalent to studying the projected transition matrix, which we define in Section \ref{sec:prelim-projected}. We will give the definitions for the arc-reversal and vertex-face walks in Section \ref{sec:2-reflection-examples}, together with their corresponding projected transition matrices. In Section \ref{sec:prelim-spectral_decomposition} we discuss properties and notation regarding an important mathematical tool in the study of state transfer: the spectral decomposition of a matrix. Finally in Section \ref{sec:prelim-sums_of_cosines} we provide a useful lemma that is a simplified version of a result of Conway \& Jones \cite{ConJon1976}, which we will need in order to work with equations involving cosines of rational multiples of $\pi$.

\subsection{Two-reflection walks}\label{sec:2reflectionwalks}

Let $\cA$ be some finite set. The transition matrix $U \in \cx^{\cA \times \cA}$ of a \textsl{two-reflection walk} is a product of two reflections: it has the form
\[
U = (2P - I)(2Q - I),
\]
where $P$ and $Q$ are orthogonal projections onto some subspaces of $\cx^{\cA}$. In the corresponding \textsl{discrete-time quantum walk}, if $\Zv$ is the initial state, we let the state evolve by repeatedly multiplying it with $U$, so that $U^t\Zv$ is the state of the walk at time $t \in \ints_{\geq 0}$. In this general setting, implementing the transition matrix $U$ is as easy as it is to implement measurements in the subspaces that $P$ and $Q$ project onto (see \cite[Claim 3.16]{JefZur2023}), which makes the two-reflection walk an important model to study. One of the early works in which a two-reflection walk is studied is that of Szegedy \cite{Sze2004}. In their work, a specific two-reflection walk (which Szegedy calls a `bipartite walk') is constructed as follows. Let $X$ and $Y$ be finite sets and suppose that for each $y \in Y$ we have a probability vector $p_y \in [0,1]^X$ and for each $x \in X$ a probability vector $q_x \in [0,1]^Y$. Then define for each $y \in Y$ and $x \in X$ the vectors
\[
\psi_y = \sum_{x\in X} \sqrt{p_y(x)}\, \Ze_{x,y} \quad \text{and} \quad \phi_x = \sum_{y\in Y} \sqrt{\vphantom{p_y}q_x(y)}\, \Ze_{x,y},
\]
where $\Ze_{x,y}$ denotes the standard basis vector corresponding to the element $(x,y) \in X \times Y$. The corresponding walk is now defined by taking $\cA = X \times Y$ and letting $P$ and $Q$ be the orthogonal projectors onto the subspaces of $\cx^{\cA}$ spanned by the sets $\{\psi_y\}_{y\in Y}$ and $\{\phi_x\}_{x\in X}$ respectively. Note that each of these is a set of pairwise orthonormal vectors. If $X = Y$ and $p_x(y) = q_x(y)$ for all $x,y \in X$, then the matrix $\{p_y(x)\}_{x,y \in X}$ is the transition matrix of a discrete-time Markov chain. In other words, this process can be used to turn a Markov chain into a quantum walk \cite{Sze2004}.

In the most general setting of two-reflection walks, we do not construct the sets $\{\psi_y\}_{y\in Y}$ and $\{\phi_x\}_{x\in X}$ from probability vectors, but we let each of them be an arbitrary set of pairwise orthonormal vectors, so that $P$ and $Q$ are arbitrary orthogonal projections in $\cx^{\cA \times \cA}$. Throughout this paper, we let $M \in \cx^{\cA \times Y}$ denote the matrix that has the states $\psi_y$ as its columns, that is, $M\Ze_y = \psi_y$ for all $y \in Y$. Then we can write $P = MM^*$, where $M^*$ denotes the Hermitian transpose of $M$. Similarly, define $N \in \cx^{\cA \times X}$ by $N\Ze_x = \phi_x$ for all $x \in X$, so that $Q = NN^*$. Note also that $M^*M = I_Y$ and $N^*N = I_X$.

We are particularly interested in the scenario where the initial state is $\phi_x$ for some $x\in X$ and $U^t\phi_x$ is close to a scalar multiple of $\phi_{x'}$ for $x' \in X$. Loosely speaking, this can be thought of as transfer of the state from $\phi_x$ to $\phi_{x'}$. In the following section, we will introduce the main tool for studying this process.

\subsection{Projected transition matrices and state transfer}\label{sec:prelim-projected}

We will follow \cite{GuoSch2024} and define the titular matrices of this section.
We  project each power $U^t$ of $U$ down onto the column space of $N$ as follows: for $t \in \ints_{\geq 0}$, define
\[
B_t = N^*U^tN.
\]
The matrix $B_t$ is Hermitian: this follows from the fact that $(2Q - I)N = N$, which can be applied on both sides to find:
\[
B_t = N^* U^t N = N^* (2Q - I)U^{t-1}(2P - I) N = N^* (U^t)^* N = B_t^*.
\]
Specifically, we will call $B := B_1 = N^* UN$ the \textsl{projected transition matrix} of the two-reflection walk. We remark that $B$ can be written in terms of the \textsl{discriminant matrix} $D \coloneq N^*M$ of $U$ (following the nomenclature of \cite{Sze2004}) as follows:
\[
B = N^* (2P - I)N = 2N^*MM^*N - N^*N = 2DD^* - I.
\]
In particular, the eigenvalues of $B$ can be obtained from the singular values of $D$.
In \cite{GuoSch2024} (as well as in \cite{KubSeg2022} for the specific case of the arc-reversal walk) it is shown that, for every $t\geq 1$, the matrix $B_t$ can be written as a polynomial of $B$; more precisely, 
\[
B_t = T_t(B),
\]
where $T_t$ is the $t$-th Chebyshev polynomial of the first kind.

For $u, v \in X$, consider the quantity
\[
m(t,\gamma) := \| U^t \phi_u - \gamma \phi_v \| 
\]
for each time $t \geq 0$ and each unimodular $\gamma \in \cx$. If $u \neq v$ and there exist $t$ and $\gamma$ such that $m(t,\gamma)=0$, then $U^t \phi_u = \gamma \phi_v$ and we say that there is \textsl{perfect state transfer} from $\phi_u$ to $\phi_v$ at time $t$. Equivalently, there is perfect state transfer from $\phi_u$ to $\phi_v$ at time $t$ if $B_t(v,u) = \gamma$ for some $t$ and $\gamma$. (See \cite{GuoSch2024}; Note that $\phi_v^* U^t \phi_u = B_t(v,u)$.)

Even if such $t$ and $\gamma$ do not exist, we can consider the value
\[
m = \inf_{t \in \ints_{> 0}} \inf_{\gamma: |\gamma| = 1} m(t,\gamma)
\]
If $m = 0$ we say that there is \textsl{pretty good state transfer} from $\phi_u$ to $\phi_v$. (Equivalently, since the unit circle is compact, there is pretty good state transfer from $\phi_u$ to $\phi_v$ if there exists a unimodular $\gamma \in \cx$ such that $
\inf_{t \in \ints_{> 0 }} m(t,\gamma) = 0$.) We can also give a lower bound on the value of $m$ and determine whether there exists a time $t$ for which this lower bound can be attained. Using that $\phi_u = N\Ze_u$, we see that
\[
\begin{split}
m(t,\gamma)^2 &= (U^t N \Ze_u - \gamma N \Ze_v)^*(U^t N \Ze_u - \gamma N \Ze_v) \\
&= 2 - \gamma\Ze_u^*N^*(U^t)^*N\Ze_v - \comp{\gamma}\Ze_v^*N^*U^tN\Ze_u \\
&= 2 - 2\Re(\gamma B_t(u,v)),
\end{split}
\]
where the second equality holds because both $U^tN\Ze_u$ and $\gamma N\Ze_v$ have norm $1$ and the third equality holds because $B_t = N^*U^t N$ is Hermitian. This means that lower bounding $m(t,\gamma)$ is equivalent to upper bounding $\Re(\gamma B_t(u,v))$ which is in turn equivalent to upper bounding $|B_t(u,v)|$, because $\Re(\gamma B_t(u,v))$ is maximal precisely when $\gamma B_t(u,v) = |B_t(u,v)|$. In Lemma 3.1, we give an upper bound for $\sup_t|B_t(u,v)|$ and we characterize when it is tight. This will lead us to the definition of peak state transfer.

If $U^t\phi_u = \gamma\phi_u$ for some $t$ and $\gamma$ (or equivalently, if $B_t(u,u) = \gamma$) then the walk is \textsl{periodic at $\phi_u$} and the \textsl{period} is the smallest $t$ for which this occurs. Alternatively, we say that the walk is \textsl{$t$-periodic at $\phi_u$}. If the periodicity happens with the same period $t$ at $\phi_u$ for every $u \in X$, we say that the walk itself is \textsl{$t$-periodic}. In this case, $B_t$ is a diagonal matrix. For example, for the graph in Figure \ref{fig:high_peakST}, the arc-reversal walk (to be defined in Section \ref{sec:2-reflection-examples}) is periodic at vertex $u$, but not at the other vertices. 

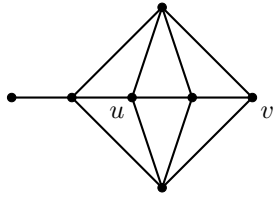
\begin{figure}[htb]
\centering
\begin{tikzpicture}[scale=.8]
	\begin{pgfonlayer}{nodelayer}
		\node [label={[shift={(-.2,-.2)}]center:{$u$}}] (0) at (-0.5, 0) {};
		\node (1) at (0.5, 0) {};
		\node (2) at (-1.5, 0) {};
		\node [label={[shift={(.2,-.2)}]center:{$v$}}] (3) at (1.5, 0) {};
		\node (4) at (-2.5, 0) {};
		\node (5) at (0, 1.5) {};
		\node (6) at (0, -1.5) {};
	\end{pgfonlayer}
	\begin{pgfonlayer}{edgelayer}
		\draw [thick] (4.center) to (2.center);
		\draw [thick] (2.center) to (0.center);
		\draw [thick] (0.center) to (1.center);
		\draw [thick] (1.center) to (3.center);
		\draw [thick] (2.center) to (5.center);
		\draw [thick] (0.center) to (5.center);
		\draw [thick] (1.center) to (5.center);
		\draw [thick] (3.center) to (5.center);
		\draw [thick] (2.center) to (6.center);
		\draw [thick] (0.center) to (6.center);
		\draw [thick] (1.center) to (6.center);
		\draw [thick] (3.center) to (6.center);
	\end{pgfonlayer}
    \filldraw [black]
    (0) circle (2pt)
    (1) circle (2pt)
    (2) circle (2pt)
    (3) circle (2pt)
    (4) circle (2pt)
    (5) circle (2pt)
    (6) circle (2pt);
\end{tikzpicture}
\caption{The arc-reversal walk on this $7$-vertex graph is $12$-periodic at the vertex $u$ an not periodic at any other vertex. For the vertices $u$ and $v$, the value of $m(t,\gamma)$ is lower bounded by $2 - \sqrt{3} \approx 0.267949$, a bound that is reached for $t = 6$ and $\gamma = 1$. It turns out that this is an instance of peak state transfer, the definition of which is given in Section \ref{sec:peakst}.
\label{fig:high_peakST}}
\end{figure}

\begin{remark}
\label{rem:all_examples_have_same_N}
These definitions are set up in such a way that it is natural to consider the projected transition matrix $B$; we are only interested in state transfer between those states that form a basis for the subspace that $Q$ projects onto, and the $(v,u)$-entry of $B_t$ represents the inner product of $\phi_v$ with $U^t\phi_u$. In all of the examples in this paper, the set $\cA$ is the set of arcs of a graph and $X$ is the set its set of vertices. The matrix $N$ is always the so-called `normalized arc-vertex incidence matrix' of that graph, as we define in Section \ref{sec:2-reflection-examples}. For this reason, we will also usually speak of `state transfer from $u$ to $v$' rather than `state transfer from $\phi_{u}$ to $\phi_{v}$', even in the general case of a two-reflection walk where there is no underlying graph.
\end{remark}

Many of the two-reflection walks that are discussed in the literature---including Szegedy's walk, the arc-reversal walk, and vertex-face walk---involve only real-valued matrices. Thus, for simplicity and readability, throughout this text, we will assume that all states $\phi_x$ and $\psi_y$ are real, so that all matrices involved are real as well. We will make some comments about the complex analogues of our statements in Section \ref{sec:openproblems}.

\subsection{Examples of two-reflection walks }\label{sec:2-reflection-examples}

In addition to Szegedy's walk, which was discussed in Section \ref{sec:2reflectionwalks}, various examples of two-reflection walks can be found in \cite{GodZha2019} and \cite{GodZha2023}, two of which we will highlight in this section: the arc-reversal walk and the vertex-face walk. These two walks will constitute essentially all of the examples given in this paper. For both of these walks, the state space is indexed by the arcs of a graph.

The \textsl{arc-reversal walk} (also called \textsl{Grover walk}) is defined as follows. Let $G = (V,E)$ be a graph and let $\cA$ be the set of arcs of $G$; i.e.\ the set of all ordered pairs of vertices $(u,v)$ for which $uv$ is an edge of $E$. If $a \coloneq (u,v)$ is an arc of $G$, we say that $u$ is the \textsl{tail} of $a$ and $v$ is the \textsl{head} of $a$. Now let $N \in \cx^{\cA \times V}$ be the \textsl{normalized arc-vertex incidence matrix} of $G$, that is
\[
N(a,v) = \begin{cases}
\frac{1}{\sqrt{d(v)}} &\text{if $v$ is the tail of $a$}; \\
0 &\text{otherwise,}
\end{cases}
\]
where $d(v)$ denotes the degree of the vertex $v$. The transition matrix for the arc-reversal walk is given by
\[
U = R(2NN^T - I),
\]
where $R$ is the \textsl{arc-reversal matrix} of $G$: if $\revarc{a} = (v,u)$ denotes the arc opposite to arc $a = (u,v)$, then
\[
R(a,b) = \begin{cases}
1 &\text{if $b = \revarc{a}$}; \\
0 &\text{otherwise.}
\end{cases}
\]
Clearly, $R$ is a reflection and we can write $R = (2MM^T - I)$, where $M$ is the \textsl{normalized arc-edge incidence matrix}:
\[
M(a,e) = \begin{cases}
\frac{1}{\sqrt{2}} &\text{if $a = (u,v)$ such that $uv \in E$};\\
0 &\text{otherwise.}
\end{cases}
\]
The projected transition matrix of the arc-reversal walk is the \textsl{normalized adjacency matrix} of $G$ and it is given by
\[
B = N^T R N = D_G^{-1/2} A D_G^{-1/2},
\]
where $A$ and $D_G$ denote the adjacency matrix and degree matrix of $G$ respectively.  In particular, if $G$ is $d$-regular, we have that $B = \frac{1}{d}A$. The normalized adjacency matrix appears in other contexts, such as random walks \cite{Spi2019}. If we subtract it from the identity matrix we obtain the \textsl{normalized Laplacian matrix} of the graph \cite{Chu1997}.

The \textsl{vertex-face walk} is a two-reflection walk takes place on a graph embedding. It was defined in \cite{Zha2021} and a detailed description and further properties are given in \cite{GuoSch2024}. We will give an intuitive description here. In a topological sense, a \textsl{map} is a $2$-cell embedding of a graph in an orientable surface, the latter of which can be viewed as a sphere with a certain number of handles attached to it (this number is the \textsl{genus} of the surface). The embedding being \textsl{$2$-cell} means that the \textsl{faces} of the embedding, which are the disjoint regions obtained upon removal of the embedding from the surface, are homeomorphic to open disks. We can denote the map as a triple $(V,E,F)$ of its vertices, edges and faces respectively. We can assign to every edge of the map a pair of arcs that lie on opposite sides of that edge and that point in opposite directions. Since the embedding is orientable, this can be done in such a way that, after picking an orientation, the arcs lying inside each face form a clockwise walk around the edges of that face. See Figure \ref{fig:K4_torus_facial_walks}.
\begin{figure}
\centering
\vspace{-30pt}
\begin{tikzpicture}[scale=1.5]
\begin{pgfonlayer}{nodelayer}
    \node (0) [label={[shift={(-.3,.1)}]center:{$v$}}]   at (0, 0.87868) {};
    \node (1) at (-0.87868, 0) {};
    \node (2) at (0, -0.87868) {};
    \node (3) at (0.87868, 0) {};
    \node (a) at (-1.5, 1.5) {};
    \node (b) [label={[shift={(-.75,-.75)}]center:{$f_2$}}] at (1.5, 1.5) {};
    \node (c) at (-1.5, -1.5) {};
    \node (d) at (1.5, -1.5) {};
    \node (8) at (0, 2.12132) {};
    \node (9) at (2.12132, 0) {};
    \node (10) at (0, -2.12132) {};
    \node (11) at (-2.12132, 0) {};
    \node (12) at (0, 2.12132) {};
    \node (13) [label=center:{$f_1$}] at (0,0) {};
\end{pgfonlayer}
\begin{pgfonlayer}{edgelayer}
\clip (a) rectangle (d);
    \draw [gray, dashed] (1.center) to (0.center);
    \draw [gray, dashed] (0.center) to (3.center);
    \draw [gray, dashed] (3.center) to (2.center);
    \draw [gray, dashed] (2.center) to (1.center);
    \draw [gray, dashed] (1.center) to (11.center);
    \draw [gray, dashed] (0.center) to (12.center);
    \draw [gray, dashed] (3.center) to (9.center);
    \draw [gray, dashed] (2.center) to (10.center);
    \draw [bend right=22, very thick, style=directed, color=lightgray] (3.center) to (0.center);
    \draw [bend right=22, very thick, style=directed, color=lightgray] (0.center) to (1.center);
    \draw [bend right=22, very thick, style=directed, color=lightgray] (1.center) to (2.center);
    \draw [bend right=22, very thick, style=directed, color=lightgray] (2.center) to (3.center);
    \draw [bend right=22, very thick, style=directed, color=black] (1.center) to (0.center);
    \draw [bend right=22, very thick, style=directed, color=black] (0.center) to (3.center);
    \draw [bend right=22, very thick, style=directed, color=black] (3.center) to (2.center);
    \draw [bend right=22, very thick, style=directed, color=black] (2.center) to (1.center);
    \draw [bend right=22, very thick, style=directed, color=lightgray] (0.center) to (12.center);
    \draw [bend right=22, very thick, style=directed, color=lightgray] (3.center) to (9.center);
    \draw [bend right=22, very thick, style=directed, color=lightgray] (2.center) to (10.center);
    \draw [bend right=22, very thick, style=directed, color=lightgray] (1.center) to (11.center);
    \draw [bend right=22, very thick, style=directed, color=lightgray] (12.center) to (0.center);
    \draw [bend right=22, very thick, style=directed, color=lightgray] (9.center) to (3.center);
    \draw [bend right=22, very thick, style=directed, color=lightgray] (10.center) to (2.center);
    \draw [bend right=22, very thick, style=directed, color=lightgray] (11.center) to (1.center);
    \draw [thin] (a.center) to (b.center);
    \draw [thin] (b.center) to (d.center);
    \draw [thin] (d.center) to (c.center);
    \draw [thin] (c.center) to (a.center);
\end{pgfonlayer}
\filldraw [black]
(0) circle (2pt)
(1) circle (2pt)
(2) circle (2pt)
(3) circle (2pt)
;
\end{tikzpicture}
\vspace{-30pt}
\caption{This embedding of the complete graph $K_4$ in the torus (genus $1$) has two faces $f_1$ and $f_2$. (The pairs of opposite sides of the square are identified to form the torus.) The corresponding clockwise walks in black and gray have lengths $4$ and $8$ respectively so that $d(f_1) = 4$ and $d(f_2) = 8$. Since $v$ appears on the clockwise walk of $f_2$ twice, so that the $(v,f_1)$-entry of $N^TM$ is $1/\sqrt{12}$ and the $(v,f_2)$-entry is $1/\sqrt{6}$.
\label{fig:K4_torus_facial_walks}}
\end{figure}
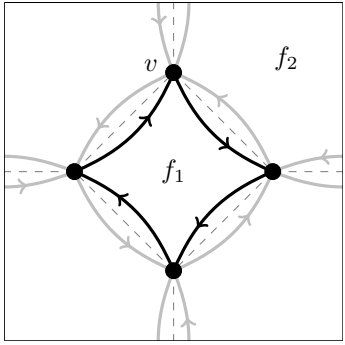
As will be the case for all examples in this paper, the matrix $N$ corresponding to the first reflection of the two-reflection walk on the map is again the normalized arc-vertex incidence matrix. For the matrix $M$ corresponding to the second reflection, we take the \textsl{normalized arc-face incidence matrix}:
\[
M(a,f) = \begin{cases}
\frac{1}{\sqrt{d(f)}} &\text{if $a$ lies inside $f$}; \\
0 &\text{otherwise,}
\end{cases}
\]
where $d(f)$ is the \textsl{face degree} of $f$, which is equal to the number of arcs that lie inside $f$. The discriminant matrix for the walk is the \textsl{normalized vertex-face incidence matrix} $N^TM$: its $(v,f)$-entry is equal to $\frac{\alpha}{\sqrt{d(v)d(f)}}$, where $\alpha$ denotes the number of times that $v$ appears on the clockwise walk corresponding to $f$. The value of this $\alpha$ can be greater than $1$, see the example in Figure \ref{fig:K4_torus_facial_walks}. As mentioned in Section \ref{sec:prelim-projected}, the projected transition matrix of any two-reflection walk can be written in terms of the discriminant matrix, so we have $B = 2N^TMM^TN - I$.

The first row of pictures in Figure \ref{fig:K4_2reflectionwalks} depicts the arc-reversal walk starting in a vertex of the complete graph $K_4$. The second row shows a similar evolution for the vertex-face walk in the unique planar embedding of $K_4$. In the third row, we see the vertex-face walk on an embedding of $K_4$ in the torus. The respective projected transition matrices are given by
\[
\frac{1}{3}
\left[
\begin{array}{rrrr}
0 & 1 & 1 & 1 \\
1 & 0 & 1 & 1 \\
1 & 1 & 0 & 1 \\
1 & 1 & 1 & 0
\end{array}\right],\quad
\frac{1}{9}
\left[
\begin{array}{rrrr}
-3 & 4 & 4 & 4 \\
4 & -3 & 1 & 4 \\
4 & 4 & -3 & 4 \\
4 & 4 & 4 & -3
\end{array}\right] \quad \text{and} \quad
\frac{1}{2}
\left[
\begin{array}{rrrr}
-1 & 1 & 1 & 1 \\
1 & -1 & 1 & 1 \\
1 & 1 & -1 & 1 \\
1 & 1 & 1 & -1
\end{array}\right].
\]
The third walk is $2$-periodic at every vertex.

Another example of a vertex-face walk was depicted in Figure \ref{fig:peak_ST_illustration}. The map in that figure is an embedding of the Cartesian product of the cycle graphs $C_4$ and $C_6$. It is called the `toroidal $(4,6)$-grid'; toroidal grids are discussed in detail in Section \ref{sec:grids}.

\begin{figure}
\centering
\begin{subfigure}{0.19\textwidth}
    \centering
    \includegraphics[scale=.7]{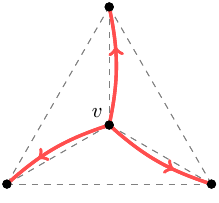}
\end{subfigure}
\begin{subfigure}{0.19\textwidth}
    \centering
    \includegraphics[scale=.7]{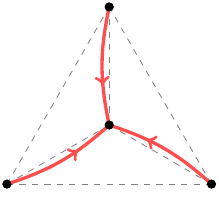}
\end{subfigure}
\begin{subfigure}{0.19\textwidth}
    \centering
    \includegraphics[scale=.7]{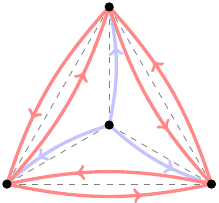}
\end{subfigure}
\begin{subfigure}{0.19\textwidth}
    \centering
    \includegraphics[scale=.7]{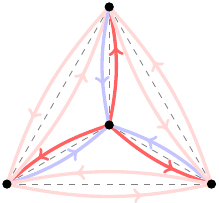}
\end{subfigure}
\begin{subfigure}{0.19\textwidth}
    \centering
    \includegraphics[scale=.7]{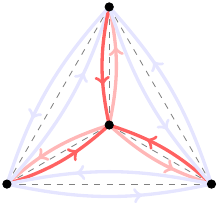}
\end{subfigure}

\vspace{5pt}

\begin{subfigure}{0.19\textwidth}
    \centering
    \includegraphics[scale=.7]{figures/K4ar0.pdf}
\end{subfigure}
\begin{subfigure}{0.19\textwidth}
    \centering
    \includegraphics[scale=.7]{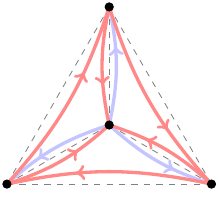}
\end{subfigure}
\begin{subfigure}{0.19\textwidth}
    \centering
    \includegraphics[scale=.7]{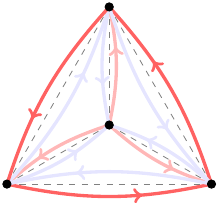}
\end{subfigure}
\begin{subfigure}{0.19\textwidth}
    \centering
    \includegraphics[scale=.7]{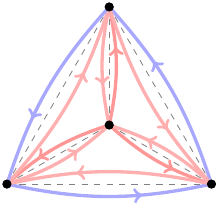}
\end{subfigure}
\begin{subfigure}{0.19\textwidth}
    \centering
    \includegraphics[scale=.7]{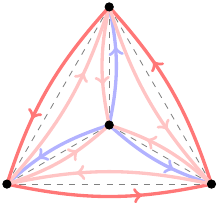}
\end{subfigure}

\vspace{5pt}

\begin{subfigure}{0.19\textwidth}
    \centering
    \includegraphics[scale=.7]{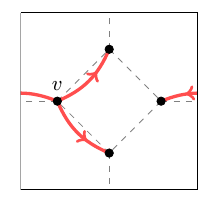}
\end{subfigure}
\begin{subfigure}{0.19\textwidth}
    \centering
    \includegraphics[scale=.7]{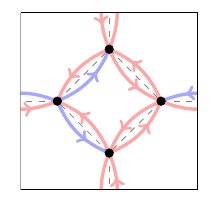}
\end{subfigure}
\begin{subfigure}{0.19\textwidth}
    \centering
    \includegraphics[scale=.7]{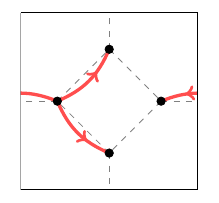}
\end{subfigure}
\begin{subfigure}{0.19\textwidth}
    \centering
    \includegraphics[scale=.7]{figures/K4torus1.pdf}
\end{subfigure}
\begin{subfigure}{0.19\textwidth}
    \centering
    \includegraphics[scale=.7]{figures/K4torus0.pdf}
\end{subfigure}

\caption{Pictured are the first few steps of three different two-reflection walks on the arcs of $K_4$; the arc-reversal walk \emph{(in row 1)}, the vertex-face walk on $K_4$ embedded in the plane \emph{(in row 2)} and the vertex-face walk on $K_4$ embedded in the torus with a faces of size $4$ and $8$ \emph{(in row 3)}. For each walk, we initialize the system at the out-going arcs at vertex $v$ and the picture depict the amplitudes at times $0,1,\ldots, 4$. Like in Figure \ref{fig:grid_vxfwalk_peakST}, red (resp.\ blue) arcs represent positive (resp.\ negative) entries of the corresponding states and the magnitudes are reflected by the opaqueness.\label{fig:K4_2reflectionwalks}}
\end{figure}

The definitions for arc-reversal walks and vertex-face walks can be extended to (embeddings of) multigraphs. Conventionally, the degree of a loop is considered to be $2$, in order to keep the property that the vertex degrees sum to twice the number of edges. Thus, every loop on a vertex $v$ contributes two arcs with $v$ as their tail vertex.

\subsection{Spectral decomposition}
\label{sec:prelim-spectral_decomposition}

A tool that is used extensively in this paper is the spectral decomposition of a matrix (in our case usually the projected transition matrix). The spectral decomposition has many applications in the area of algebraic combinatorics, see e.g.\ \cite{GodRoy2001}. In this subsection, we discuss the notation and nomenclature with respect to the spectral decomposition that we will use throughout this text.

Given a square matrix $B$, we let $\ev(B)$ denote its set of eigenvalues. Its \textsl{spectrum} $\sigma(B)$ is the multiset that contains each eigenvalue with its corresponding algebraic multiplicity. If $B$ is Hermitian, then $\ev(B) \subset \re$ and its eigenspaces are pairwise orthogonal. This means that $B$ can be written in the form
\begin{equation}
\label{eq:spectral_decomposition}
B = \sum_{\theta \in \ev(B)} \theta E_\theta,
\end{equation}
where all eigenvalues $\theta$ are real-valued and where $E_\theta$ denotes the orthogonal projection onto the $\theta$-eigenspace; that is, $E_\theta^* = E_\theta$ and
\[
E_\theta E_\rho = \begin{cases}
E_\theta &\text{if $\theta = \rho$}; \\
0 & \text{otherwise.}
\end{cases}
\]
for all $\theta,\rho \in \ev(B)$. This implies that, indeed, $BE_\theta = \theta E_\theta$. Finally, we have that
\[
\sum_{\theta \in \ev(B)} E_\theta = I.
\]
The expression in \eqref{eq:spectral_decomposition} is called the \textsl{spectral decomposition} of $B$, and we will call the $E_\theta$ the \textsl{spectral idempotents} of $B$. As an example, we give the spectral decomposition of the adjacency matrix $A_{C_n}$ of the cycle graph $C_n$. We denote the rows and columns of $A_{C_n}$ by the elements of $\{0,\ldots,n-1\}$, so that we can define
\[
A_{C_n}(i,j) = 
\begin{cases}
    1 &\text{if $i - j = \pm 1 \slimmod n$}; \\
    0 &\text{otherwise.}
\end{cases}
\]

We will state its spectral decomposition below, without proof, since it is standard and easily derivable from the eigenvectors; see for instance \cite{Dav1979}.
\begin{lemma}
\label{lem:cycle_eigenvectors}
The spectral decomposition of $A_{C_n}$ is given by
\[
A_{C_n} = \sum_{k=0}^{\lfloor n/2 \rfloor} 2\cos\left(\frac{2\pi k}{n}\right) F_{n,k},
\]
where each spectral idempotent $F_{n,k}$ is defined by
\begin{equation}
\label{eq:cycle_idempotent_entries}
F_{n,k}(j,\ell) = \begin{cases}
\frac{1}{n} &\text{if $k = 0$;} \\
\frac{1}{n} \cdot (-1)^{j + \ell} &\text{if $k = \frac{n}{2}$;} \\
\frac{2}{n}\cos\left(\frac{2\pi k(j-\ell)}{n}\right) &\text{otherwise.}
\end{cases}
\end{equation}
The rank of $F_{n,k}$ is $1$ for $k \in \{0,\frac{n}{2}\}$ and $2$ for all other $k$. \qedhere 
\end{lemma}
\noindent We will make use of this lemma in Section \ref{sec:lex_products_arc_reversal} and in Section \ref{sec:grids}.

Suppose that the rows and columns of $B$ are indexed by some set $X$. We define the \textsl{mutual eigenvalue support} of a pair of elements $x,y \in X$ by $\Lambda(x,y) = \Lambda^1(x,y) \cup \Lambda^{-1}(x,y)$, where
\[
\Lambda^1(x,y) = \{\theta \in \ev(B) \mid E_\theta(x,y) > 0\} \quad \text{and} \quad \Lambda^{-1}(x,y) = \{\theta \in \ev(B) \mid E_\theta(x,y) < 0\}
\]
are respectively called the \textsl{positive mutual eigenvalue support} and the \textsl{negative mutual eigenvalue support} of $x$ and $y$. By the positive semi-definiteness of the $E_\theta$, we have that $\Lambda(x,x) = \Lambda^1(x,x)$. The set $\Lambda(x,x)$ is simply called the \textsl{eigenvalue support} of $x$, and we will often denote it by $\Lambda(x)$ instead. For $x,y\in X$, we say that $x$ is \textsl{strongly cospectral} to $y$ if and only if for all $\theta \in \ev(B)$,
\[
E_\theta \Ze_x = \pm E_\theta\Ze_y.
\]
The notions of eigenvalue support and strong cospectrality are closely related to perfect state transfer and periodicity; see for instance \cite{God2011,GodSmi2024,Zha2019}.

We have summarized the notation regarding spectral properties that we will use throughout this paper in the table below.

\begin{center}
\begin{tabular}{|l|l|}
\hline
$\ev(B)$ & set of distinct eigenvalues of $B$ \\
$\sigma(B)$ & spectrum of $B$ \\
$E_\theta$ & spectral idempotent for eigenvalue $\theta \in \ev(B)$ \\ 
$\Lambda(x,y)$ & mutual eigenvalue support of $x$ and $y$ \\
$\Lambda^{\pm 1}(x,y)$ & positive/negative mutual eigenvalue support of $x$ and $y$ \\
$\Lambda(x)$ & eigenvalue support of $x$ \\
\hline
\end{tabular}
\end{center}

\subsection{Sums of cosines}
\label{sec:prelim-sums_of_cosines}

In Section \ref{sec:peakst}, we will see that for peak state transfer to occur between a pair of states, the eigenvalues in their mutual eigenvalue support are required to be cosines of rational multiples of $\pi$. In various examples in Sections \ref{sec:lex_products_arc_reversal}, \ref{sec:examples} and \ref{sec:grids}, we need an answer to the following question: given a set of numbers that can be written as cosines of rational multiples of $\pi$, when is a rational linear combination of them (i.e.\ a linear combination with rational coefficients) itself rational? The following lemma is a special case of a theorem from \cite{ConJon1976} that characterizes all rational solutions to the equation
\[
a\cos(2\pi \alpha) + b\cos(2\pi \beta) + c\cos(2\pi \gamma) + d\cos(2\pi \delta) = e.
\]
For the scope of this paper, we only need to know what happens when $c = d = 0$, i.e., when there are at most two cosines. It turns out that the solutions are very restricted: 
\begin{lemma}[\cite{ConJon1976}, Theorem 7]
\label{lem:conway_jones}
Let $\theta, \phi \in [0,\pi/2]$ be rational multiples of $\pi$. The following are true:
\begin{enumerate}[(i)]
\item $\cos(\theta)$ is rational if and only if $\theta \in \{0, \frac{\pi}{3},\frac{\pi}{2}\}$;
\item if $\cos(\theta)$ and $\cos(\phi)$ are irrational and there exist rational numbers $p$ and $q$ such that the linear combination
\[
p\cos(\theta) + q\cos(\phi) = r
\]
is rational, then this rational linear combination is proportional to
\begin{equation}
\label{eq:conway_jones_pi5}
\cos\frac{\pi}{5} - \cos\frac{2\pi}{5} = \frac{1}{2}.
\end{equation}
\end{enumerate}
\qed
\end{lemma}
\noindent Note that in Theorem 7 of \cite{ConJon1976}, the angles involved are required to be distinct. However, since the statement above includes only two angles at most, this requirement is not necessary. Indeed, we cannot have $\theta = \phi$ in statement (ii), since it implies that $cos(\theta) = \frac{r}{p+q}$ is rational.

\section{Peak state transfer} \label{sec:peakst}

A foundational tool for working with continuous-time quantum walks, which is summarized in \cite[Section 7.1]{CouGod2023book}, is a sequence of inequalities that has led to necessary conditions for perfect state transfer and that has motivated the definition of strongly cospectral vertices, in the study of continuous-time quantum walks.
In this section, we bring this powerful tool to the discrete-time setting.

A continuous-time quantum walk on a graph $G$ has the transition matrix
$U(t) = e^{itA}$,
where $A$ denotes the adjacency matrix of $G$ and where $t$ can take on nonnegative real value. For vertices $u,v$ of $G$,  the value $|U(t)_{u,v}|$ can be bounded by the quantity $1$ by following a sequence of three inequalities. By itself, this bound is trivial, since $U(t)$ is unitary, but each of the three inequalities is an equality if and only if a specific condition holds. If all inequalities are equalities then $|U(t)_{u,v}| = 1$ and there is perfect state transfer from $u$ to $v$ in the continuous-time walk. A similar sequence of inequalities appears in \cite{ChaZha2023} for discrete-time walks in de context of pretty good state transfer.

We will now follow a similar argument for the discrete-time two-reflection walks,  using the projected transition matrix $B$ instead. In the following lemma, $T_n$ denotes the $n$-th Chebyshev polynomial of the first kind. These polynomials are recursively defined by setting $T_0(x) = 1$ and $T_1(x) = x$, and then 
\[
T_{n+1}(x) = 2xT_n(x) - T_{n-1}(x)
\]
for all $x \in \re$. Alternatively, $T_n$ is the unique function that satisfies the expression
\begin{equation}
\label{eq:chebyshev_expression}
T_n(\cos\phi) = \cos(n\phi)
\end{equation}
for all $\phi \in [0,\pi]$.

\begin{lemma}
\label{lem:technical_lemma}
    Let $B \in \re^{X \times X}$ be the projected transition matrix of a two-reflection walk and let the spectral decomposition of $B$ be given by $B = \sum_{\theta} \theta E_{\theta}$. Let $u,v \in X$ and let $\Lambda := \Lambda(u,v)$ denote the mutual eigenvalue support of $u$ and $v$. We have that, for any time $t$,
    \[
    |B_t(u,v)| \leq \sum_{\theta} |E_{\theta}(u,v)|.
    \]
   Equality holds if and only if 
   \[
   T_t(\theta) = \pm 1 \quad \forall \theta \in \Lambda
   \]
   and the sign of $T_t(\theta) E_{\theta}(u,v)$ is the same for every $\theta \in \Lambda$. Further,
   \[ \sum_{\theta} |E_{\theta}(u,v)| \leq 1\]
   and equality holds if and only if $u$ and $v$ are strongly cospectral with respect to $B$.
\end{lemma}

\begin{proof}
For any time $t$, we have
\begin{align}
    \nonumber |B_t(u,v)| &= \left| \sum_{\theta} T_t(\theta) E_{\theta}(u,v) \right| \\
    \label{eq:Bt_inequality1} &\leq \sum_{\theta} |T_t(\theta) E_{\theta}(u,v) | \\
    \nonumber &=  \sum_{\theta}| T_t(\theta)| \cdot |E_{\theta}(u,v)| \\
     \label{eq:Bt_inequality2} &\leq  \sum_{\theta} |E_{\theta}(u,v)|
\end{align}
where the inequality \eqref{eq:Bt_inequality1} is due to  the triangle inequality and \eqref{eq:Bt_inequality2} holds because the extrema of $T_t(x)$ are $\{-1,1\}$ for $x \in [-1,1]$. (Indeed, all eigenvalues of $B$ lie inside $[-1,1]$: for any unit-length vector $\Zv \in \re^{X \times X}$, we have
\[
\|B\Zv\| = \|N^TUN\Zv\| \leq \|U N\Zv\| = 1,
\]
since $N$ has orthonormal columns.) Moreover, $|B_t(u,v)| = \sum_{\theta} |E_{\theta} (u,v)|$ if and only if both \eqref{eq:Bt_inequality1} and \eqref{eq:Bt_inequality2} are equalities; the latter holds precisely when $T_{t}(\theta) = \pm 1$ for all $\theta \in \Lambda$, the former is true when $T_t(\theta)E_\theta(u,v)$ has the same sign for every $\theta \in \Lambda$.

Finally, by applying Cauchy-Schwarz twice, we see that
\begin{align}
\nonumber\sum_\theta |E_{\theta}(u,v)| &= \sum_{\theta}|(E_\theta\Ze_u)^T (E_\theta\Ze_v)| \\
\label{eq:idempotent_inequality1}&\leq \sum_{\theta} \|E_\theta\Ze_u\|\|E_\theta\Ze_v\| \\
\nonumber &\leq \sqrt{ \sum_{\theta} \|E_\theta\Ze_u\|^2} \sqrt{\sum_{\theta} \|E_\theta\Ze_v\|^2} \\
\nonumber&= \sqrt{\|\Ze_u\|^2}\sqrt{\|\Ze_v\|^2} \\
\nonumber&= 1,
\end{align}
where the second to last equality follows from  Pythagoras' theorem.
Moreover, if $\sum_{\theta}E_{\theta}(u,v) = 1$, then both of the inequalities above are equalities. In particular, considering  the inequality \eqref{eq:idempotent_inequality1}, this implies that $E_\theta\Ze_u$ and $E_\theta\Ze_v$ are linearly dependent for every eigenvalue $\theta$, i.e., there exist $\alpha_\theta \in \re$ such that
\[
E_\theta\Ze_v = \alpha_\theta E_\theta\Ze_u.
\]
Hence in this case,
\[
1 = \sum_{\theta} \|E_\theta\Ze_u\|\|E_\theta\Ze_v\| = \sum_{\theta} |\alpha_\theta|\cdot \|E_\theta\Ze_u\|^2 \leq \sum_{\theta} \|E_\theta\Ze_u\|^2 = 1,
\]
which implies for each $\theta$ that either $|\alpha_\theta| = 1$ or $E_\theta\Ze_u = E_\theta\Ze_u = 0$. In other words, $u$ and $v$ are strongly cospectral with respect to $B$. Conversely, if $u$ and $v$ are strongly cospectral with respect to $B$, then we can substitute $E_\theta\Ze_v = \pm E_\theta\Ze_u$ and see that
\[
\sum_{\theta} |E_\theta(u,v)| = \sum_{\theta}|(E_\theta\Ze_u)^T (E_\theta\Ze_u)| = \sum_{\theta} \|E_\theta\Ze_u\|^2 = 1. \qedhere
\]
\end{proof}

In Lemma \ref{lem:technical_lemma} above, we see that the maximal absolute value of the $(u,v)$-entry of $B_t$ never exceeds the sum of the absolute values of the $(u,v)$-entries of the spectral idempotents $E_\theta$. If this value is attained for some integer $\tau > 0$, i.e.\ if 
\[
|B_\tau(u,v)| = \sum_{\theta} |E_\theta(u,v)|
\]
for distinct indices $u $ and $ v$, we say that there is \textsl{peak state transfer} from $u$ to $v$ at time $\tau$. We note that if we allow $u=v$, then $|B_\tau(u,u)| = \sum_\theta |E_\theta(u,u)|$ occurs if and only if $|B_\tau(u,u)| =1$ since each $E_\theta(u,u)$ is nonnegative and they sum to $1$. Hence `peak state transfer' from vertex to itself only occurs in the form of periodicity.

\begin{remark}
\label{rem:lambda_empty}
    There is a trivial case of peak state transfer that occurs when $E_\theta(u,v) = 0$ for all $\theta$, i.e.,\ when the mutual eigenvalue support of $u$ and $v$ is empty. We will call this \textsl{zero state transfer}. When this happens, Lemma \ref{lem:technical_lemma} implies that $B_t(u,v) = 0$ for all $t$. In fact, the reverse implication also holds, which we prove in Lemma \ref{lem:zero_state_transfer}. As an example, zero state transfer occurs in the arc-reversal walk precisely between vertices in distinct connected components of the underlying graph. Although we are usually interested in cases with high state transfer, we will expand on this a bit further in Section \ref{subsec:no-st}. Otherwise, this is a case of peak state transfer that we would like to disregard. As such, we will not consider zero state transfer as an instance of peak state transfer throughout this paper, unless mentioned otherwise.
\end{remark}

In the theorem below, we will characterize peak state transfer in two-reflection walks, as well as perfect state transfer, which is a special case of peak state transfer. For this, we also need a result about the extrema of $T_t$, which are, for instance, given in \cite{Riv1974}:
\begin{lemma}
\label{lem:Chebyshev_extrema}
Let $x \in [-1,1]$ and $n > 0$ be some positive integer. The following are true:
\begin{enumerate}[(i)]
\item $T_n(x) = 1$ if and only if $x  = \cos  \frac{2k\pi}{n}$ for  some integer $0 \leq k \leq \frac{n}{2}$;
\item $T_n(x) = -1$ if and only if $x  = \cos \frac{(2k + 1)\pi }{n} $ for some integer $0 \leq k \leq \frac{n-1}{2}$.
\end{enumerate}
\end{lemma}

\begin{remark}
Given the expression for $T_n$ as given in equation \eqref{eq:chebyshev_expression}, the properties in Lemma \ref{lem:Chebyshev_extrema} are trivial. From an alternative viewpoint, note that, as is pointed out in \cite{Zha2021}, each eigenvalue of $B = 2DD^* - I$ can be written in the form $\theta = \cos(\phi)$, where $\phi$ is such that $e^{i\phi}$ is an eigenvalue of the transition matrix $U$. Hence $T_{n}(\theta) = 1$ if and only if $e^{i\phi}$ is an $n$-th root of unity, and $T_{n}(\theta) = -1$ if and only if $n$ is even and $e^{i\phi}$ is an $(n/2)$-th root of unity. In Theorem \ref{thm:PeakSTchar}, we give a characterization of peak state transfer that does not (explicitly) mention the Chebyshev polynomials.
\end{remark}
The following theorem is very similar to characterizations of perfect state transfer that were given by \cite[Theorem 5.3]{Zha2019} for the arc-reversal walk on regular graphs and \cite[Theorem 6.5]{KubYos2024}
for the arc-reversal walk in general. Our theorem is different in the sense that it is stated for two-reflection walks in general and that it also characterizes peak state transfer, which does not require strong cospectrality. Moreover, we do not characterize peak/perfect state transfer at a specific time. Rather, in our statement of the theorem, we characterize whether it occurs, and specify the times at which the transfer occurs. The theorem is in part a reformulation of Lemma \ref{lem:technical_lemma}. Its proof seems rather long; it requires careful bookkeeping, with some elementary number theory, to determine the first time of peak state transfer. Note that the assumption of $\Lambda$ being non-empty implies that zero state transfer from $u$ to $v$ does not occur.

\begin{theorem}\label{thm:PeakSTchar}
Let $B \in \re^{X \times X}$ be the projected transition matrix of a two-reflection walk and let the spectral decomposition of $B$ be given by $B = \sum_{\theta} \theta E_\theta$. Let $u,v \in X$ be distinct, and let $\Lambda = \Lambda^1 \cup \Lambda^{-1}$, where
\[
\Lambda^1 = \{\theta \mid E_\theta(u,v) > 0\} \quad \text{and} \quad \Lambda^{-1} = \{\theta \mid E_\theta(u,v) < 0\}
\]
denote the respective postive and negative mutual eigenvalue support of $u$ and $v$, and assume that $\Lambda \neq \emptyset$. There is peak state transfer from $u$ to $v$ at some time if and only if the following hold:
\begin{enumerate}[(i)]
\item Each $\theta \in \Lambda$ is a cosine of a rational multiple of $\pi$: it can be uniquely written as
\[
\theta = \cos\frac{p(\theta)\pi}{q(\theta)},
\]
where $p(\theta) \geq 0$ and $q(\theta) > 0$ are integers such that $p(\theta) \leq q(\theta)$ and $\gcd(p(\theta),q(\theta)) = 1$.
\item Let $\tau = \lcm\{q(\theta) \mid \theta \in \Lambda\}$. For some choice of $\gamma \in \{-1,1\}$, and all $\theta \in \Lambda$:
\begin{enumerate}[(a)]
    \item $\theta \in \Lambda^\gamma$ if and only if the integer $\tau \cdot \frac{p(\theta)}{q(\theta)}$ is even;
    \item $\theta \in \Lambda^{-\gamma}$ if and only if the integer $\tau \cdot \frac{p(\theta)}{q(\theta)}$ is odd.
\end{enumerate}
\end{enumerate}
If there is peak state transfer, it occurs precisely at all times that are odd multiples of $\tau$. This is perfect state transfer if and only if, in addition to (i) and (ii),  $u$ and $v$ are strongly cospectral with respect $B$.
\end{theorem}

\begin{proof}
Note first of all that $\Lambda^1$ and $\Lambda^{-1}$ are both non-empty: since $u$ and $v$ are distinct, we have that
\[
\sum_\theta E_\theta(u,v) = 0
\]
and because at least one of $\Lambda^{\pm 1}$ is non-empty, the other is non-empty as well.

By Lemma \ref{lem:technical_lemma} and Lemma  \ref{lem:Chebyshev_extrema}, there is peak state transfer from $u$ to $v$ at time $t$ if and only if the following hold:
\begin{enumerate}[(i')]
\item Each $\theta \in \Lambda$ is a cosine of a rational multiple of $\pi$ in the following way:
\[
\theta = \cos\frac{s(\theta)\pi}{t},
\]
where $s(\theta)$ is some integer that satisfies $0 \leq s(\theta) \leq t$.
\item For some choice of $\gamma \in \{-1,1\}$, and all $\theta \in \Lambda$:
\begin{enumerate}[(a')]
    \item $\theta \in \Lambda^\gamma$ if and only if $s(\theta)$ is even;
    \item $\theta \in \Lambda^{-\gamma}$ if and only if $s(\theta)$ is odd.
\end{enumerate}
\end{enumerate}

Now assume that there is peak state transfer from $u$ to $v$ at some time, say $t$. Then (i') and (ii') hold for some integers $s(\theta)$ and a choice of $\gamma \in \{-1,1\}$. For each $\theta \in \Lambda$, let
\[
\frac{p(\theta)}{q(\theta)} = \frac{s(\theta)}{t}
\]
be the reduced form of that nonnegative fraction, with integers $p(\theta)$ and $q(\theta)$ that satisfy the conditions as in (i) in the statement of the theorem. To prove that (ii) holds, we will write
\[
\tau \coloneq \lcm\{q(\theta) \mid \theta \in \Lambda\},
\]
and show that this value satisfies the conditions (a) and (b). We have for all $\theta \in \Lambda$ that
\begin{equation}
\label{eq:tau_divides_t}
\frac{t}{\tau}\cdot \left(\tau \cdot \frac{p(\theta)}{q(\theta)}\right) = t\cdot \frac{p(\theta)}{q(\theta)} = s(\theta).
\end{equation}
Note that $\tau$ must divide $t$ since every $q(\theta)$ divides $t$ by definition. If $\theta \in \Lambda^{-\gamma}$, then $s(\theta)$ is odd by (b'), so that the integers $\frac{t}{\tau}$ and $\tau \cdot \frac{p(\theta)}{q(\theta)}$ are both odd and (b) holds. In particular, since $\Lambda^{-\gamma}$ is always non-empty, $\frac{t}{\tau}$ is always odd. Hence for $\theta \in \Lambda^\gamma$, since $s(\theta)$ is even by (a'), we conclude from \eqref{eq:tau_divides_t} that $\tau \cdot \frac{p(\theta)}{q(\theta)}$ is even, so that (a) holds as well.

Conversely, suppose that (i) and (ii) hold for certain $p(\theta)$, $q(\theta)$ and $\gamma$ and suppose that $\tau$ is again the least common multiple of the $q(\theta)$. Then we can write for all $\theta \in \Lambda$:
\[
\theta = \cos\frac{p(\theta) \pi}{q(\theta)} = \cos\frac{s(\theta) \pi}{\tau},
\]
where
\[
s(\theta) \coloneq \tau \cdot \frac{p(\theta)}{q(\theta)} \quad \text{is }
\begin{cases}
\text{even} &\text{if $\theta \in \Lambda^{\gamma}$}; \\
\text{odd} &\text{if $\theta \in \Lambda^{-\gamma}$},
\end{cases}
\]
so that (i') and (ii') hold for $t = \tau$, and there is peak state transfer at time $\tau$. There is also peak state transfer at any odd multiple $(2k+1)\tau$ of $\tau$, since
\[
\theta = \cos\frac{(2k+1) s(\theta) \pi}{(2k+1)\tau}
\]
and since $(2k+1)s(\theta)$ has the same parity as $s(\theta)$ for all $\theta \in \Lambda$. Now assume that there is peak state transfer at a time that is an even multiple of $\tau$, say $2k\tau$, and pick some $\theta \in \Lambda^{-\gamma}$. We find by (i') that
\[
\cos\frac{s(\theta) \pi}{\tau} = \theta =  \cos\frac{r(\theta) \pi}{2k\tau},
\]
where $0 \leq r(\theta) \leq 2k\tau$ is an odd integer. However, by the restrictions on $s(\theta)$ and $r(\theta)$, we must have that $r(\theta) = 2ks(\theta)$, which contradicts $r(\theta)$ being odd. We conclude that there is no peak state transfer at any time that is an even multiple of $\tau$. If there is peak state transfer at some time $t$ that is not a multiple of $\tau$, then there is some $\theta \in \Lambda$ for which $q(\theta)$ does not divide $t$. On the other hand, by (i'), there is some integer $0 \leq r(\theta) \leq t$ such that
\[
\cos\frac{p(\theta)\pi}{q(\theta)} = \theta = \cos\frac{r(\theta)\pi}{t}.
\]
Similarly to the previous case, this implies that $r(\theta) = t\cdot \frac{p(\theta)}{q(\theta)}$, which is not an integer (as $\gcd(p(\theta),q(\theta)) = 1$), leading to another contradiction.

Finally, the peak state transfer is perfect state transfer if and only if the second inequality in Lemma \ref{lem:technical_lemma} holds as well, i.e., if and only if $u$ and $v$ are strongly cospectral with respect to $B$.
\end{proof}

\begin{remark}
\label{rem:gamma=1}
The value of $\gamma$ in Theorem \ref{thm:PeakSTchar} above determines the sign of the entry $B_{\tau}(u,v)$, since $\gamma$ is the sign of $T_\tau(\theta)E_\theta(u,v)$ for every $\theta$:
\[
B_\tau(u,v) = \sum_\theta T_\tau(\theta)E_\theta(u,v) = \sum_\theta \gamma |E_\theta(u,v)| = \gamma \sum_\theta |E_\theta(u,v)|.
\]
For arc-reversal walks on connected graphs, the projected transition matrix $B = D^{-1/2}AD^{-1/2}$ is nonnegative and irreducible. It always has $D^{1/2} \ones$ as an eigenvector for its Perron-Frobenius eigenvector $1$. Therefore, the $1$-eigenspace is one-dimensional, so that $1$ is in the positive mutual eigenvalue support of all pairs $u \neq v$. As $1 = \cos\frac{0}{1}$, we must have $\gamma = 1$ for all instances of peak state transfer in arc-reversal walks on connected graphs. This answers a question left open in \cite{KubYos2024}: the authors mention that they do not know of perfect state transfer occurring in non-regular graphs with $\gamma = -1$. With our remark we have argued that such examples do not exist: the $\gamma$ in their characterization of perfect state transfer in arc-reversal walks \cite[Theorem 6.5]{KubYos2024} can be replaced by the value $1$.
\end{remark}

The characterization of peak state transfer from $u$ to $v$ in Theorem \ref{thm:PeakSTchar} gives us that it is determined by the eigenvalues and the $(u,v)$-entries of the spectral idempotents. In particular, the status of peak state transfer will be the same for two discrete-time quantum walks whose projected transition matrices are cospectral and where the $(u,v)$-entries of the spectral idempotents are equal. We discuss a specific instance of this behavior for arc-reversal walks on distance-regular graphs in Remark \ref{rem:DRGs}.

In the next result, we give a precise characterization of periodicity in two-reflection walks, which is similar to Theorem \ref{thm:PeakSTchar}. 

\begin{theorem}
\label{thm:periodicity_characterisation}
Let $B \in \re^{X \times X}$ be the projected transition matrix of a two-reflection walk and let the spectral decomposition of $B$ be given by $ B = \sum_{\theta} \theta E_{\theta}$. Let $u \in X$ and let $\Lambda = \{\theta \mid E_{\theta}(u,u) \neq 0\}$ be the eigenvalue support of $u$. There is periodicity at $u$ if and only if each $\theta \in \Lambda$ is a cosine of a rational multiple of $\pi$:
\begin{equation}
\label{eq:periodicity_theta_cosine}
\theta = \cos\frac{p(\theta)\pi}{q(\theta)},
\end{equation}
where $p(\theta) \geq 0$ and $q(\theta) > 0$ are integers such that $p(\theta) \leq q(\theta)$ and $\gcd(p(\theta),q(\theta)) = 1$. If this is the case, let $\tau = \lcm\{q(\theta) \mid \theta \in \Lambda\}$. The period is $\tau$ if $\tau \cdot \frac{p(\theta)}{q(\theta)}$ has the same parity for all $\theta \in \Lambda$. Otherwise, the period is $2\tau$.
\end{theorem}

\begin{proof}
Note first of all that since the $E_r$ are positive semidefinite and sum to the identity matrix, we have that
\[
\sum_{\theta} |E_\theta(u,u)| = \sum_{\theta} E_\theta(u,u) = 1.
\]
In addition,
\[
\Lambda = \Lambda^1 = \{\theta \mid E_\theta(u,u) > 0\},
\]
so that for some time $t$, by Lemma \ref{lem:technical_lemma}, we have that $|B_t(u,u)| = 1$ if and only if $T_t(\theta) = 1$ for all $\theta \in \Lambda$ or $T_t(\theta) = -1$ for all $\theta \in \Lambda$. So by Lemma \ref{lem:Chebyshev_extrema}, if there is periodicity, every eigenvalue $\theta \in \Lambda$ can be written as cosine of a rational multiple of $\pi$.

Now assume that every $\theta \in \Lambda$ can be written in the form \eqref{eq:periodicity_theta_cosine} and let $\tau$ be the least common multiple of all $q(\theta)$ over $\theta \in \Lambda$. By construction, $\tau$ is the first time at which $T_\tau(\theta) = \pm 1$ for all $\theta \in \Lambda$. The sign of each $T_\tau(\theta)$ depends on the parity of $\tau \cdot \frac{p(\theta)}{q(\theta)}$. If the sign is the same for all $\theta \in \Lambda$, then there is periodicity with period $\tau$. If not, then from the properties of Chebyshev polynomials, we know that
\[
T_{2\tau}(\theta) = 2T_\tau(\theta)^2 - 1 = 1
\]
for all $\theta \in \Lambda$, so that periodicity happens with period $2\tau$.
\end{proof}

\begin{remark}
\label{rem:periodicity_char_simplified}
In Theorem \ref{thm:periodicity_characterisation}, if $\tau \cdot \frac{p(\theta)}{q(\theta)}$ is odd for every $\theta \in \Lambda$, then there is periodicity at time $\tau$ and $B_\tau(u,u) = -1$. For some models of two-reflection walks, this type of periodicity cannot occur at all: as is shown in \cite{GodZha2023} and \cite{GuoSch2024}, in both the arc-reversal walk and the vertex-face walk, periodicity can only occur with $B_\tau(u,u) = 1$. (The same argument as in Remark \ref{rem:gamma=1} can also be followed in order to see this.) Hence for such walks, the statement of Theorem \ref{thm:periodicity_characterisation} can be slightly simplified: in part (ii), the value $\tau \cdot \frac{p(\theta)}{q(\theta)}$ is always even. So we can write every $\theta \in \Lambda$ as a cosine of a rational multiple of $\pi$, in the following form:
\[
\theta = \cos\frac{2r(\theta) \pi}{s(\theta)},
\]
where $r(\theta) \geq 0$ and $s(\theta) > 0$ are integers such that $r(\theta) \leq \frac{1}{2}s(\theta)$ and $\gcd(r(\theta),s(\theta)) = 1$, and the period is always $\tau = \lcm\{s(\theta) \mid \theta \in \Lambda\}$.
\end{remark}

We note that our characterization of peak state transfer only requires as input the spectral decomposition of the projected transition matrix $B$; the time of peak state transfer is determined by the eigenvalues. 
We have reduced the problem to determining if some given numbers are equal to cosines of rational multiples of $\pi$, combined with some parity conditions on the entries of the spectral idempotents. While the former task could be difficult computationally, we will show, in Section \ref{sec:grids}, that it is possible to apply our theorems from this section, using results from algebraic number theory. 
If the characteristic polynomial of $B$ has rational coefficients, then we note that $\Lambda = \Lambda^{1} \cup \Lambda^{-1}$ is closed under taking algebraic conjugates. If $\Lambda$ contains only cosines of rational multiples of $\pi$, then we can partition $\Lambda$ into sets of algebraic conjugates. The minimal polynomials would have to divide the Chebyshev polynomial of the first kind. This is a way to detect when the rational multiples of $\pi$ occur with exact precision.

\section{Arc-reversal walks and lexicographic products}\label{sec:lex_products_arc_reversal}

In this section we consider only arc-reversal walks: for this specific type of two-reflection walk, we classify peak state transfer in the complete graphs and cycle graphs. We then show that from every graph $G$ that admits peak state transfer, we can obtain a larger graph that also admits peak state transfer by blowing up all vertices of $G$ (or equivalently, by taking the lexicographic product of $G$ with the empty graph $\overline{K_m}$).

In the following lemma, we classify peak state transfer (and periodicity) in the complete graphs $K_n$. As usual, we work with the spectral decomposition of the projected transition matrix
\[
B = \sum_{\theta \in \ev(B)} \theta E_\theta,
\]
where for the arc-reversal walk, $B$ is the normalized adjacency matrix of $G$, as described in Section \ref{sec:prelim}. As was also defined in Section \ref{sec:prelim}, we let $\Lambda(u,v)$ denote mutual eigenvalue support of vertices $u$ and $v$; recall that this set has a positive and a negative part.
\begin{lemma}
\label{lem:complete_graphs_peakST}
Let $n\geq 2$. The arc-reversal walk on the complete graph $K_n$ does not admit peak state transfer for $n > 2$ and there is no periodicity at any vertex for $n>3$. The graph $K_2$ admits perfect state transfer at time $1$ and is $2$-periodic everywhere. The graph $K_3$ is $3$-periodic everywhere.
\end{lemma}
\begin{proof}
If $A$ is the adjacency matrix of $K_n$, the eigenvalues of its projected transition matrix $B = \frac{1}{n-1}A$ are given by $1$ and $-\frac{1}{n-1}$, and the corresponding spectral idempotents are
\[
\frac{1}{n-1}J_n \quad \text{and} \quad I - \frac{1}{n-1}J_n
\]
respectively. For vertices $u,v$ of $K_n$, this means that $\Lambda(u,v) = \{1,-\frac{1}{n-1}\}$. By Lemma \ref{lem:conway_jones}, the eigenvalue $-\frac{1}{n-1}$ is only a cosine of a rational multiple of $\pi$ if $n \in \{2,3\}$. Hence for all other values of $n$, there is no perfect state transfer and no periodicity. For $n = 2$,
\[
B = \begin{bmatrix}
0 & 1 \\
1 & 0
\end{bmatrix},
\]
so there is peak state transfer at time $1$ and periodicity at both vertices at time $2$. For $n = 3$, we write the eigenvalues $1$ and $-\frac{1}{2}$ as rational multiples of cosines, in the form specified in Theorem \ref{thm:PeakSTchar} and Theorem \ref{thm:periodicity_characterisation}:
\[
1 = \cos\frac{0\cdot \pi}{1} \quad \text{and} \quad -\frac{1}{2} = \cos\frac{2\pi}{3}.
\]
We then take $\tau = \lcm\{1,3\} =3$. Since $\tau \cdot 0$ and $\tau \cdot \frac{2}{3}$ are both even, there is periodicity everywhere with period $3$ by Theorem \ref{thm:periodicity_characterisation}. However, for distinct vertices $u$ and $v$ of $K_3$, we have that $\Lambda^1(u,v) = \{1\}$ and $\Lambda^{-1}(u,v)= \{-\frac{1}{2}\}$. So condition (ii) of Theorem \ref{thm:PeakSTchar} is false and there is no peak state transfer at any time.
\end{proof}

Next, we discuss cycle graphs. Unlike the complete graphs, the cycle graphs do admit peak state transfer that is not perfect state transfer. In Lemma \ref{lem:cycle_eigenvectors}, we gave the spectral decomposition of the adjacency matrix of the cycle graph $C_n$. As is the case in that lemma, we will denote the vertices of $C_n$ by the elements of $\{0,\ldots,n-1\}$, where vertices $i$ and $j$ are adjacent if and only if $i - j = \pm 1 \slimmod n$. In the following lemma, we give a full characterization of peak state transfer in cycles; we note that periodicity and perfect state transfer in cycles has been determined in earlier work (see \cite{YalGed2015}).

\begin{lemma}
\label{lem:cycles_peakST}
Consider the arc-reversal walk on $C_n$ and let $i$ be any vertex. The walk is $n$-periodic at vertex $i$. If $n = 2m$, there is perfect state transfer from vertex $i$ to vertex $i + m \slimmod n$ at time $m$. If $n = 4\ell$, there is peak state transfer from $i$ to vertices $i \pm \ell \slimmod n$ at time $\ell$. There is no peak state transfer between any of the other pairs of vertices of $C_n$.
\end{lemma}
\begin{proof}
Throughout this proof, we may assume without loss of generality that $i=0$. If $A$ is the adjacency matrix of the $n$-cycle, then its projected transition matrix with respect to the arc-reversal walk is given by $B = A/2$. By Lemma \ref{lem:cycle_eigenvectors}, we know that the eigenvalues of $B$ are $\cos\frac{2\pi k}{n}$ for $k=0,\ldots,\lfloor n/2\rfloor$, with corresponding spectral idempotents $F_{n,0},\ldots,F_{n,\lfloor n/2\rfloor}$ as defined in \eqref{eq:cycle_idempotent_entries}. The eigenvalue support $\Lambda(u,u)$ of vertex $0$ is the whole spectrum of $B$. By Theorem \ref{thm:periodicity_characterisation}, since every eigenvalue of $B$ is a cosine of a rational multiple of $\pi$, there is periodicity at $0$; by Remark \ref{rem:periodicity_char_simplified}, because $\cos\frac{2\cdot 1 \cdot \pi}{n}$ is in the eigenvalue support, the period must be $n$.

Now assume that $n=2m$. For every $k$, we have
\[
F_{n,k}(0,m) = \begin{cases}
\frac{1}{n} &\text{if $k=0$};\\
(-1)^m\frac{1}{n} &\text{if $k=m$}; \\
\frac{1}{m}\cos k\pi &\text{otherwise.}
\end{cases}
\]
That is, $F_{n,k}(0,m) > 0$ for all even $k$ and $F_{n,k}(0,m) < 0$ for all odd $k$. In particular, the mutual eigenvalue support of $0$ and $m$ is the whole spectrum of $B$. Since $n$ is even, the eigenvalues can be written in the form
\[
\cos\frac{k\cdot\pi}{m}, \quad k\in\left\{0,\ldots,\left\lfloor \frac{n}{2}\right\rfloor\right\}
\]
Since we can take $k=1$, by Theorem \ref{thm:PeakSTchar}, if there is peak state transfer from $0$ to $m$, it must occur for the first time at time $m$. Now for all $k$:
\[
m \cdot \frac{2k}{n} = k,
\]
so that both conditions (i) and (ii) of Theorem \ref{thm:PeakSTchar}(ii) are satisfied and there is peak state transfer. This peak state transfer is perfect state transfer, because $0$ and $m$ are strongly cospectral with respect to $B$. Indeed, we have for all vertices $x$:
\[
F_{n,0}(x,m) = \frac{1}{n} = F_{n,0}(x,0) \quad \text{and} \quad F_{n,m}(x,m) = \frac{1}{n}\cdot (-1)^{x+m} = (-1)^m \cdot F_{n,m}(x,0)
\]
and for all $k \notin \{0,m\}$:
\[
F_{n,k}(x,m) = \frac{1}{m} \cos\frac{\pi k (x - m)}{m} = \frac{1}{m} \cos\left(\frac{\pi k x}{m} - \pi k \right) = (-1)^k \cdot F_{n,k}(x,0).
\]

Next, assume that $n = 4\ell$. Without loss of generality, it suffices to show that there is peak state transfer from vertex $0$ to vertex $\ell$. This time, we have for all $k$:
\begin{equation}
\label{eq:peakST_cycle_n=4l_entries}
F_{n,k}(0,\ell) = \begin{cases}
\frac{1}{n} &\text{if $k\in \{0,2\ell\}$};\\
\frac{1}{\ell}\cos \frac{k\pi}{2} &\text{otherwise.}
\end{cases}
\end{equation}
That is, the mutual eigenvalue support of $0$ and $\ell$ consists of all eigenvalues $\cos\frac{2\pi k}{n}$ for which $k$ is even. More specifically, $F_{n,k}(0,\ell) > 0$ if $k \equiv 0 \slimmod 4$ and $F_{n,k}(0,\ell) < 0$ if $k \equiv 2 \slimmod 4$. Now every eigenvalue in the mutual eigenvalue support can be written in the form
\[
\cos\frac{(k/2)\cdot \pi}{\ell},
\]
so that if peak state transfer occurs from $0$ to $\ell$, it must occur for the first time at time $m$. Finally, we find that
\[
\ell \cdot \frac{2k}{n} = \frac{k}{2},
\]
so that the parity of this integer depends on $k \slimmod 4$; once more, we can see that conditions (i) and (ii) of \ref{thm:PeakSTchar} are both satisfied.

Finally, consider for any $n > 2$ the pair of vertices $0$ and $x \in \{0,\ldots,n-1\}$ such that $x \notin \{0,n/4,n/2,3n/4\}$. Then
\[
F_{n,1}(0,x) = \frac{2}{n}\cos\frac{2\pi x}{n} \neq 0
\]
so that $\cos\frac{2\pi}{n}$ is in the mutual eigenvalue support of $0$ and $x$. Let
\[
t = \begin{cases}
n &\text{if $n$ is odd;} \\
\frac{n}{2} &\text{if $n$ is even.}
\end{cases}
\]
Since $t$ is the smallest integer that makes $t \cdot \frac{2 k}{n}$ integer for all $k$, we know by Theorem \ref{thm:PeakSTchar} that if there is peak state transfer from vertex $0$ to vertex $m$, it must occur for the first time at time $t$. However, there is periodicity at vertex $0$ at time $t = n$ if $n$ is odd and there is perfect state transfer from vertex $0$ to vertex $n/2$ at time $t = n/2$ if $n$ is even. In either case, since the columns of the matrix $B_t$ have norm at most $1$, it must be that $B_t(0,x) = 0$, whereas
\[
\sum_{k} |F_{n,k}(0,x)| \geq |F_{n,1}(0,x)| > 0 = B_t(0,x),
\]
which means that there is no peak state transfer. We conclude that there is no peak state transfer from $0$ to $x$ at any time.
\end{proof}

Using Theorem \ref{thm:PeakSTchar} to prove Lemma \ref{lem:cycles_peakST} somewhat obscures the simple structure of the arc-reversal walk on $C_n$. In fact, it is not difficult to show that the transition matrix $U$ acts as a permutation on the $2n$-arcs of the cycle, cyclically permuting each of the two sets of arcs that point in the same direction. With this perspective, it is easy to see that there is periodicity at time $n$ and perfect state transfer at time $n/2$ between pairs of opposite vertices if the cycle is even. This also implies that $B_t(0,t) = 1/2$ for all $t \notin \{0,n/2\}$. Hence, in order to show that there is peak state transfer from $0$ to $\ell = n/4$ at time $\ell$, it suffices to show that the absolute sum of all $F_{n,k}(0,\ell)$ as given in \eqref{eq:peakST_cycle_n=4l_entries} is equal to $1/2$. Note that `peak state transfer' does not mean that there is `more' state transfer from vertex $0$ to vertex $\ell$ than to other vertices; rather the upper bound $\sum_{k}|F_{k,n}(0,x)|$ is lower for $x = \ell$ than for other $x$, and that bound can be attained.

Consider the arc-reversal walk on some graph $G$ and let $m > 1$ be an integer. We can construct a new graph that is a \textsl{blow-up} of $G$: it has vertex set
\[
V(G) \times \{1,\ldots,m\}
\]
and $(u,a)\sim (v,b)$ if and only if $u\sim v$. In words, every vertex in $G$ becomes a coclique of size $m$ in the blow-up graph and there is a  copy of $K_{m,m}$ between two of those cocliques if there is an edge in $G$ between the corresponding vertices in the original graph. Alternatively, this blow-up graph is the \textsl{lexicographic product} of $G$ with the empty graph $\overline{K_m}$ on $m$ vertices. We will denote this graph by $G[\overline{K_m}]$, notation that is also used in e.g.\ \cite{GodRoy2001}. The adjacency matrix of $G[\overline{K_m}]$ is given by
\[
A_G \otimes J_m,
\]
so that the projected transition matrix (which is the normalized adjacency matrix) of $G[\overline{K_m}]$ is
\[
B' := \frac{1}{m} B \otimes J_m,
\]
with $B$ the projected transition matrix of $G$. The eigenvalues of $B'$ are the products of the eigenvalues of $B$ with $0$ and $1$. This means that $\ev(B') = \ev(B) \cup \{0\}$. The spectral decomposition of $B'$ is given by
\[
B' = 0 \cdot F_0 + \sum_{\theta \in \ev(B) \setminus \{0\}} \theta F_\theta,
\]
where
\begin{equation}
\label{eq:nonzero_idempotents_mG}
F_\theta := E_\theta \otimes \frac{1}{m}J_m
\end{equation}
for $\theta \neq 0$ and the idempotent $F_0$ for the eigenvalue $0$ is
\begin{equation}
\label{eq:zero_idempotent_mG}
F_0 :=
\begin{cases}
I \otimes (I - \frac{1}{m}J_m) + E_0 \otimes \frac{1}{m}J_m &\text{if $0 \in \ev(B)$}; \\
I \otimes (I - \frac{1}{m}J_m) &\text{otherwise.} \\
\end{cases}
\end{equation}
This leads to the following lemma.
\begin{lemma}
\label{lem:lex_product_peakST}
Consider the arc-reversal walk on some graph $G$. Let $u,v$ be distinct vertices of $G$, let $m > 1$ be an integer and let $a,b \in \{1,\ldots,m\}$. Then there is peak state transfer from $u$ to $v$ at time $t$ in $G$ if and only if there is peak state transfer from $(u,a)$ to $(v,b)$ at time $t$ in $G[\overline{K_m}]$.
\end{lemma}
\begin{proof}
It suffices to show that $B$ and $B'$ as described above have identical positive and negative eigenvalue supports. Let $\theta$ be any eigenvalue of the projected transition matrix $B$ of $G$. By equations \eqref{eq:nonzero_idempotents_mG} and \eqref{eq:zero_idempotent_mG}, we have that
\[
F_\theta((u,a),(v,b)) = \frac{1}{m}E_\theta(u,v)
\]
(also if $\theta = 0$, since $u \neq v$)
so that $\theta$ is in the positive (resp.\ negative) eigenvalue support of $(u,a)$ and $(v,b)$ with respect to $B'$ if and only if it is in the positive (resp.\ negative) eigenvalue support of $u$ and $v$ with respect to $B$.

Now assume that $0$ is not an eigenvalue of the original projected transition matrix $B$, then since $u \neq v$, we have $F_0((u,a),(v,b)) = 0$ by \eqref{eq:zero_idempotent_mG}, so that $0$ is not in the mutual eigenvalue support of $(u,a)$ and $(v,b)$. This completes the proof.
\end{proof}

In the following lemma, we show that periodicity in $G$ can also lead to peak state transfer in the blow-up $G[\overline{K_m}]$ of $G$. The proof is straightforward, but rather long, once more because of the careful bookkeeping that is necessary. Therefore, we have moved its proof into Appendix \ref{app:proofs}.

\begin{lemma}
\label{lem:lex_product_periodicity}
Assume that the arc-reversal walk on some connected graph $G$ with at least two vertices is $\tau$-periodic at some vertex $u$. Moreover, let $m > 1$ be an integer and let $a,b \in \{1,\ldots,m\}$ be distinct. Write $\tau' = \lcm\{\tau,4\}$. The following are true:
\begin{enumerate}[(i)]
\item There is periodicity at $(u,a)$ with period $\tau'$ in the arc-reversal walk on $G[\overline{K_m}]$.
\item There is peak state transfer from $(u,a)$ to $(u,b)$ at time $\tau'/2$ in the arc-reversal walk on $G[\overline{K_m}]$ if and only if one of the following holds:
\begin{enumerate}[(a)]
\item $\tau$ is not divisible by $4$;
\item $\tau \equiv 4 \mod 8$ and the integer $\tau \cdot \frac{r(\theta)}{s(\theta)}$ is even for all eigenvalues $\theta \in \Lambda(u,u) \setminus \{0\}$ of the projected transition matrix of $B$, where $r(\theta)$ and $s(\theta)$ are as in Remark \ref{rem:periodicity_char_simplified}.
\end{enumerate}
\end{enumerate}
Conversely, if there is periodicity at $(u,a)$ or peak state transfer from $(u,a)$ to $(u,b)$ in $G[\overline{K_m}]$, there must be periodicity at $u$ in $G$.
\end{lemma}

A consequence of Lemmas \ref{lem:lex_product_peakST} and \ref{lem:lex_product_periodicity} is that if there is no peak state transfer or periodicity in $G$, then there is no peak state transfer or periodicity in the blow-up graph $G[\overline{K_m}]$ either. Since the only complete graphs for which the arc-reversal walk admits perfect state transfer or periodicity are $K_2$ and $K_3$, we can conclude that the only complete multipartite graphs with $r$ parts of size $m$ that admit periodicity and peak state transfer are the bipartite graphs $K_{m,m}$ and the tripartite graphs $K_{m,m,m}$. We will make use of this when we classify peak state transfer in strongly regular graphs in Section \ref{subsec:SRGs}.

\section{More examples of peak state transfer in arc-reversal walks}\label{sec:examples}
We will begin this section by expanding on the notion of zero state transfer that was introduced in Section \ref{sec:peakst}. We then classify peak state transfer in arc-reversal walks on strongly regular graphs (Section \ref{subsec:SRGs}) and partly in block designs (Section \ref{subsec:blockdesigns}). We also give a family of graphs for which the amount of peak state transfer (that is, the value of $\sum_{\theta}|E_\theta(u,v)|$) tends to $1$ (Section \ref{subsec:peakST_arc_reversal_approaches_1}). Finally, we briefly discuss peak state transfer in other examples that arise from certain incidence structures (Section \ref{subsec:other_examples}). 

For most of our examples, we do not compute the value of $\sum_{\theta} |E_\theta(u,v)|$. The symmetric nature of our examples makes it so that this amount is usually quite low: since the columns of $B_t$ have norm at most $1$, if there is peak state transfer from a vertex $u$ to $m > 0$ other vertices at time $t$ with equal amounts, then this amount is at most $1/\sqrt{m}$. That does not mean that the amount of peak state transfer is always low, as we will see in the infinite family where this quantity tends to $1$. The graph that we saw in Figure \ref{fig:high_peakST}  provides another small example: its  arc-reversal walk admits peak state transfer from $u$ to $v$ with $B_6(u,v) = \frac{1}{2}\sqrt{3} \approx 0.866025$.  

\subsection{Zero state transfer}\label{subsec:no-st}

As defined in Section \ref{sec:peakst}, for the transition matrix $U = (2MM^T - I)(2NN^T - I)$ of a two-reflection walk, let $B_t = N^T U^t N$ for all $t \in \ints_{\geq 0}$ and write 
\[
B \coloneq B_1 = \sum_{\theta \in \ev(B)} \theta E_\theta
\]
for the spectral decomposition of the projected transition matrix. As mentioned earlier in Remark \ref{rem:lambda_empty}, if $E_\theta(u,v) = 0$ for all $\theta \in \Lambda(u,v)$, which we called `zero state transfer', then $B_t(u,v) = 0$ for all $t$. Here we give a short proof of the equivalence of these two conditions.
\begin{lemma}
\label{lem:zero_state_transfer}
There is zero state transfer from $u$ to $v$ if and only if $B_t(u,v) = 0$ for all $t$.
\end{lemma}
\begin{proof}
It suffices to prove the reverse implication. Every idempotent $E_\theta$ can be written as a polynomial in $B$: we can write $E_\theta = p_\theta(B)$, where
\[
p_\theta(x) = \prod_{\rho \neq \theta}\frac{x - \rho}{\theta - \rho}.
\]
Moreover, every $p_\theta$ can in turn be written as a linear combination of Chebyshev polynomials, so that every $E_\theta$ is a linear combination of the $B_t = T_t(B)$. Hence if $B_t(u,v) = 0$ for all $t$, then $E_\theta(u,v) = 0$ for all $\theta$ as well.
\end{proof}

In the arc-reversal walk, zero state transfer occurs between all pairs of vertices in different components of the underlying graph. We will now argue that it cannot occur between vertices that are in the same connected component of the graph, by proving a slightly more general statement for a certain class of two-reflection walks.

An $n \times n$ matrix $M$ over any field is called \textsl{reducible} if there exists a non-trivial partition $S \cup T$ of its index set such that
\begin{equation}
\label{eq:reducible_partition}
M(i,j) = 0 \quad \text{for all $(i,j) \in S \times T$}.
\end{equation}
(Equivalently, $M$ is reducible if there is a permutation matrix $P$ such that $PMP^T$ is block upper-triangular (with more than one block), in particular block diagonal if $M$ is symmetric.) If no such partition exists, we say that $M$ is \textsl{irreducible}. If $M$ is irreducible and nonnegative (i.e.\ all its entries are nonnegative), it satisfies the necessary conditions for applying the Perron-Frobenius theorem (see \cite{BroHae2012}).

Now we claim the following. Assume that there exist pairwise distinct real values $a_\theta$ such that
\[
A \coloneq \sum_{\theta \in \ev(B)} a_\theta E_\theta
\]
is a nonnegative matrix, where the $E_\theta$ are the spectral idempotents of $B$. Then zero state transfer  occurs in the two-reflection walk if and only if $B$ is reducible. The projected transition matrix of the arc-reversal walk is itself nonnegative, which makes it a two-reflection walk to which our claim applies. Similarly, for the vertex-face walk we have that
\[
DD^T = \frac{1}{2}(B + I) = \sum_{\theta} \frac{1}{2}(\theta + 1)E_\theta
\]
is nonnegative, so the claim also applies there. Each entry of $DD^T$ is denoted by a pair of vertices and since $D$ is the normalized vertex-face incidence matrix, that entry is non-zero if and only if the corresponding vertices share a face of the embedding. The connected nature of the underlying surface makes it so that for every pair of vertices $u$ and $v$, there is a sequence $u = v_0,v_1,\ldots,v_{k} = v$ such that the $(v_{i-1},v_{i})$-entry of $DD^T$ is non-zero for every $i=1,\ldots,k$. In other words, $DD^T$ (as well as $B$) is also irreducible and zero state transfer does not occur in any vertex-face walk.

To prove the claim, assume that such a nonnegative matrix $A$ exists. Note that, as mentioned in the proof of Lemma \ref{lem:zero_state_transfer}, every $E_\theta$ is a polynomial in $B$, so that if $B$ is reducible with $X = S\cup T$ a nontrivial partition of the index set as in \eqref{eq:reducible_partition}, then $E_\theta(u,v) = 0$ for all $\theta$ and all $(u,v) \in S \times T$, i.e., there is zero state transfer between all such $u$ and $v$.
Conversely if $B$ is irreducible, then $A$ is also irreducible, since it has the same spectral idempotents. Because $A$ is nonnegative, the Perron-Frobenius theorem applies: the largest eigenvalue of $A$ (in absolute value) has multiplicity $1$, and the corresponding rank-1 spectral idempotent is a positive matrix. Every entry of this idempotent is non-zero, so that there is no zero state transfer between any pair of indices.

It is possible for zero state transfer to occur even if $B$ is irreducible, but in such case, no linear combination $\sum_\theta a_\theta E_\theta$ (with distinct $a_\theta$) can be nonnegative. As an example, consider the following two-reflection walk, which can be considered as a `signed' arc-reversal walk on the $4$-cycle. Let
\[
U = (2MM^T - I)(2NN^T - I),
\]
where
\[
N = \frac{1}{\sqrt{2}}\begin{bmatrix}
1 & 0 & 0 & 0 \\
1 & 0 & 0 & 0 \\
0 & 1 & 0 & 0 \\
0 & 1 & 0 & 0 \\
0 & 0 & 1 & 0 \\
0 & 0 & 1 & 0 \\
0 & 0 & 0 & 1 \\
0 & 0 & 0 & 1
\end{bmatrix}
\quad \text{and} \quad
M =
\frac{1}{\sqrt{2}}
\begin{bmatrix}
1 & 0 & 0 & 0 \\
0 & 0 & 0 & 1 \\
0 & 1 & 0 & 0 \\
-1 & 0 & 0 & 0 \\
0 & 0 & 1 & 0 \\
0 & 1 & 0 & 0 \\
0 & 0 & 0 & 1 \\
0 & 0 & 1 & 0.
\end{bmatrix}
\]
This is similar to the arc-reversal walk on $C_4$, except that one of the arcs has been given a negative `sign'. See Figure \ref{fig:signed_arc_reversal_C4}. The associated projected transition matrix is given by
\[
B = N^T (2LL^T - I)N = \frac{1}{2}\begin{bmatrix}
0 & -1 & 0 & 1 \\
-1 & 0 & 1 & 0 \\
0 & 1 & 0 & 1 \\
1 & 0 & 1 & 0
\end{bmatrix}
\]
Its spectral idempotents are
\[
E_{\theta} = \frac{1}{4}\begin{bmatrix}
2 & \sqrt{2} & 0 & -\sqrt{2} \\
\sqrt{2} & 2 & -\sqrt{2} & 0 \\
0 & -\sqrt{2} & 2 & -\sqrt{2} \\
-\sqrt{2} & 0 & -\sqrt{2} & 2
\end{bmatrix} \quad \text{and} \quad E_\rho = \frac{1}{4}\begin{bmatrix}
2 & -\sqrt{2} & 0 & \sqrt{2} \\
-\sqrt{2} & 2 & \sqrt{2} & 0 \\
0 & \sqrt{2} & 2 & \sqrt{2} \\
\sqrt{2} & 0 & \sqrt{2} & 2
\end{bmatrix},
\]
corresponding to eigenvalues $\theta = -\frac{1}{2}\sqrt{2}$ and $\rho = \frac{1}{2}\sqrt{2}$. We can see that there is zero state transfer between the pairs $v_1,v_3$ and $v_2,v_4$, even though $B$ is irreducible.

\begin{figure}
    \centering
    \begin{tikzpicture}[scale=.8]
    	\begin{pgfonlayer}{nodelayer}
    		\node [label=above:$v_1$] (0) at (0, 2) {};
    		\node [label=right:$v_2$] (1) at (2, 0) {};
    		\node [label=below:$v_3$] (2) at (0, -2) {};
    		\node [label=left:$v_4$](3) at (-2, 0) {};
    	\end{pgfonlayer}
    	\begin{pgfonlayer}{edgelayer}
    		\draw [style=directed, thick, color=black, bend right=15] (0.center) to node [pos = 0.6, left=2pt] {$-1$} (1.center);
    		\draw [style=directed, thick, color=black, bend right=15] (1.center) to node [pos = 0.6, above=2pt] {$1$} (2.center);
    		\draw [style=directed, thick, color=black, bend right=15] (2.center) to node [pos = 0.6, right=2pt] {$1$} (3.center);
    		\draw [style=directed, thick, color=black, bend right=15] (3.center) to node [pos = 0.6, below=2pt] {$1$} (0.center);
    		\draw [style=directed, thick, color=black, bend right=15] (0.center) to node [pos = 0.6, above=2pt] {$1$} (3.center);
    		\draw [style=directed, thick, color=black, bend right=15] (3.center) to node [pos = 0.6, left=2pt] {$1$} (2.center);
    		\draw [style=directed, thick, color=black, bend right=15] (2.center) to node [pos = 0.6, below=2pt] {$1$} (1.center);
    		\draw [style=directed, thick, color=black, bend right=15] (1.center) to node [pos = 0.6, right=2pt] {$1$} (0.center);
    	\end{pgfonlayer}
            \filldraw [black]
                (0) circle (2.5pt)
                (1) circle (2.5pt)
                (2) circle (2.5pt)
                (3) circle (2.5pt);
    \end{tikzpicture}
    \caption{The `signed' arc-reversal walk has zero state transfer between the antipodal pairs of vertices, even though the underlying graph is connected.\label{fig:signed_arc_reversal_C4}}
\end{figure}
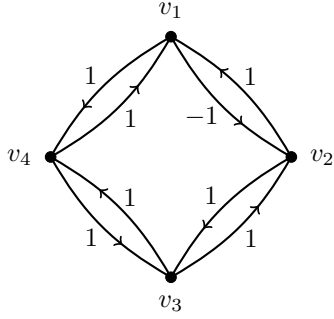

\subsection{Strongly regular graphs}
\label{subsec:SRGs}

A strongly regular graph $X$ with parameter set $(n,k,a,c)$ is a $k$-regular graph on $n$ vertices where every pair of adjacent vertices has $a$ common neighbours and every pair of non-adjacent vertices has $c$ common neighbours. If $X$ is connected, then $X$ has diameter $2$. We refer to \cite{GodRoy2001} for the basic of strongly regular graphs, some of which we will repeat here, in order to establish notation. 

For a strongly regular graph $X$, the adjacency matrix $A(X)$ has three distinct eigenvalues, $k\geq  \theta >\tau$, where  the first inequality is strict if and only if $X$ is connected. (Note: in this section on strongly regular graphs, we let $\theta$ denote the specific `middle' eigenvalue of a strongly regular graph and we use $\lambda$ to denote an arbitrary eigenvalue.)

We will make use of the  following lemma that characterizes disconnected strongly regular graphs.

\begin{lemma}[\cite{GodRoy2001}, Lemma 10.1.1]
\label{lem:godsil_royle_imprimitive}
Let $X$ be a strongly regular graph. The following conditions are equivalent. 
\begin{enumerate}[(i)]
\item $X$ is disconnected;
\item $c = 0$;
\item $a = k-1$;
\item $X$ is isomorphic to the disjoint union of $r$ copies of $K_{k+1}$ for some $r > 1$.
\end{enumerate}
\end{lemma}

The complement $\overline{X}$ of a strongly regular graph $X$ is also strongly regular, with parameter set $(n,\overline{k},\overline{a},\overline{c})$, where
\begin{equation}
\label{eq:srg_complement_parameters}
\overline{k} = n-k-1, \quad \overline{a} = n-2-2k+c \quad \text{and} \quad \overline{c} = n-2k+a
\end{equation}
and with eigenvalues
\begin{equation}
\label{eq:srg_complement_eigenvalues}
n-k-1 \geq -1 -\tau > -1 - \theta.
\end{equation}
If $X$ and its complement $\overline{X}$ are both connected, we say that the strongly regular graph is \textsl{primitive}, otherwise it is \textsl{imprimitive}. If $\overline{X}$ is disconnected, then Lemma \ref{lem:godsil_royle_imprimitive} applies and $X$ is isomorphic to the complete multipartite graph with $\frac{n}{n-k}$ parts of size $n-k$. A \textsl{conference graph} is a strongly regular graph for which the multiplicities of $\theta$ and $\tau$ are equal; these are the only strongly regular graphs of which the eigenvalues are potentially non-integral (see for instance \cite[Lemma 10.3.3]{GodRoy2001}). In the following theorem, we give a classification of peak state transfer in strongly regular graphs.

\begin{theorem}
\label{thm:SRGs_peakST}
 Suppose that $X$ is a strongly regular graph on $n$ vertices. Peak state transfer occurs in $X$ if and only if $X$ is isomorphic to one of the following:
 \begin{enumerate}[(a)]
     \item the disjoint union of $\frac{n}{2}$ copies of $K_2$;
    \item the  complete bipartite graph $K_{\frac{n}{2},\frac{n}{2}}$; or 
    \item the complete tripartite graph $K_{\frac{n}{3},\frac{n}{3},\frac{n}{3}}$. 
 \end{enumerate}
\end{theorem}
\begin{proof}
Assume first that $X$ is imprimitive. If $X$ is disconnected, it is isomorphic to the disjoint union of $r$ copies of $K_{k+1}$ for some $r > 1$ by Lemma \ref{lem:godsil_royle_imprimitive}. There is zero state transfer between distinct connected components of $X$, and within each component, there is peak state transfer if and only if the complete graph $K_{k+1}$ admits peak state transfer. By Lemma \ref{lem:complete_graphs_peakST}, the only complete graph that admits peak state transfer is $K_2$, so that this is case (a).

If $\overline{X}$ is disconnected, then $X$ is isomorphic to the complete multipartite graph with $r := \frac{n}{n-k}$ parts of size $n-k$. As discussed at the end of the previous section, by Lemmas \ref{lem:lex_product_peakST} and \ref{lem:lex_product_periodicity}, such a multipartite graph admits peak state transfer if and only if it is bipartite or tripartite, so that gives either case (b) or case (c).

It remains is to show that if $X$ is primitive, it does not admit peak state transfer. Since $X$ is primitive, its eigenvalues satisfy
\begin{equation}
\label{eq:srg_eigenvalues_primitive}
k > \theta > 0 > \tau > k-n,
\end{equation}
by standard argumentation that can be found in \cite{GodRoy2001}. We also have that $\theta$ and $\tau$ satisfy the following equations:
\begin{equation}
\label{eq:srg_theta_tau_equations}
\theta + \tau = a - c \quad \text{and} \quad \theta\tau = c - k
\end{equation}

We denote the multiplicities of the eigenvalues $k$, $\theta$ and $\tau$ by $m_k = 1$, $m_\theta$ and $m_\tau$ respectively. The distance matrices $A_0,A_1,A_2$ of $X$ (where $A_i(u,v) = 1$ if the distance from $u$ to $v$ is $i$ and $0$ otherwise) can be written as 
\[
A_0 = I, \quad A_1 = A(X) \quad \text{and} \quad A_2 = J - I - A(X)
\]
where $A(X)$ is the adjacency matrix. The spectral idempotents can be written as linear combinations of the $A_j$:
\[
E_\lambda = \frac{1}{n}\sum_{j=0}^2 q_\lambda(j) A_j,
\]
where the $q_\lambda \in \re^{\{0,1,2\}}$ are given by
\[
q_k = \begin{bmatrix}
1 \\ 1 \\ 1
\end{bmatrix}, \quad q_\theta = \begin{bmatrix}
m_\theta \\ \frac{n-k+\tau}{\theta-\tau} \\ \frac{\tau-k}{\theta-\tau}
\end{bmatrix}
\text{ and} \quad q_{\tau} = 
\begin{bmatrix}
m_\tau \\ \frac{k-n-\theta}{\theta-\tau} \\ \frac{k-\theta}{\theta-\tau}
\end{bmatrix}.
\]
Since $X$ is $k$-regular, the projected transition matrix of the arc-reversal walk on $X$ is given by $B = \frac{1}{k}A(X)$ and it has eigenvalues $1 > \frac{\theta}{k} > \frac{\tau}{k}$, with the same respective idempotents $E_k$, $E_\theta$ and $E_\tau$.

By \eqref{eq:srg_eigenvalues_primitive}, and since $\theta \geq 0$, all entries of the $q_\lambda$ are non-zero, which then implies that all entries of each idempotent $E_\lambda$ are non-zero. We conclude that the mutual eigenvalue support
$\Lambda = \{\lambda : E_\lambda(u,v) \neq 0\}$ is equal to $\{1,\frac{\theta}{k},\frac{\tau}{k}\}$ for all pairs of vertices $u,v$. In particular if there is peak state transfer in the primitive strongly regular graph $X$, then $\theta$ and $\tau$ must be cosines of rational multiples of $\pi$.

We proceed by assuming first that $X$ is not a conference graph. For a non-conference-graph, the eigenvalues $\theta$ and $\tau$ are integer, so that $\frac{\theta}{k}$ and $\frac{\tau}{k}$ are rational. In light of Lemma \ref{lem:conway_jones}, as well as the restrictions on $\theta$ and $\tau$ that we have already mentioned, if peak state transfer occurs, $\theta$ must be equal to $\frac{k}{2}$ and $\tau$ is either $-\frac{k}{2}$ or $-k$. If $\tau = -\frac{k}{2}$, it must be that $a = c$ and $c-k = -\frac{k^2}{4}$ (from \eqref{eq:srg_theta_tau_equations}). This can only hold if $k = 2 \pm 2\sqrt{1-c}$, i.e., if $c = 1$ and $k = 2$, or $c=0$ and $k=4$. In the former case, $a = c = 1 = k-1$, so that in both cases, $X$ would be disconnected (Lemma \ref{lem:godsil_royle_imprimitive}), but we assumed that it is not. On the other hand, if $\tau = -k$, then $k - c = \frac{k^2}{2}$, so that $k = 1 \pm \sqrt{1 - 2c}$, i.e., $c = 0$ and $k=2$, which again contradicts the imprimitivity of $X$.

Finally, assume that $X$ is a conference graph. By \cite{GodRoy2001}, $n$ must be congruent to $1 \mod 4$, and $k = \frac{n-1}{2}$ and $c = \frac{n-1}{4}$. Since $0 < \theta < k$ and $-k < \tau < 0$, let $\phi,\theta \in (0,\frac{\pi}{2})$ be such that $\cos(\phi)= \frac{\theta}{k}$ and $\cos(\psi) = -\frac{\tau}{k}$. 
Assume that there is peak state transfer: then both $\phi$ and $\psi$ are rational multiples of $\pi$. It follows that
\[
\cos(\phi)\cos(\psi) = \frac{k-c}{k^2} = \frac{1}{n-1},
\]
or equivalently
\[
\cos(\phi - \psi) + \cos(\phi + \psi) = \frac{2}{n-1}.
\]
Note that $\phi - \psi \in (-\frac{\pi}{2},\frac{\pi}{2})$ and $\phi + \psi \in (0,\pi)$, so there exist rational multiples $\alpha \in [0,\frac{\pi}{2})$ and $\beta \in (0,\frac{\pi}{2}]$ of $\pi$ such that
\begin{equation}
\label{eq:cos_equation_srgs}
\cos(\alpha) + \gamma\cos(\beta) = \frac{2}{n-1}
\end{equation}
for some $\gamma = \pm 1$. Either $\cos(\alpha)$ and $\cos(\beta)$ are both rational, or they are both irrational. In the latter case, by Lemma \ref{lem:conway_jones}(ii), $n$ must be equal to $5$. The only primitive strongly regular graph on $5$ vertices is the $5$-cycle, which does not admit peak state transfer by Lemma \ref{lem:cycles_peakST}. If $\cos(\alpha)$ and $\cos(\beta)$ are both rational on the other hand, we must have $\cos(\alpha) \in \{\frac{1}{2},1\}$ and $\cos(\beta) \in \{0,\frac{1}{2}\}$. Plugging in all possible values in \eqref{eq:cos_equation_srgs}
\[
\frac{2}{n-1} \in \left\{0,\tfrac{1}{2},1,\tfrac{3}{2}\right\}
\]
so that the only possible integer values for $n$ are $3$ and $5$. Once more, that leaves the only the $5$-cycle as a primitive strongly regular graph. We conclude that peak state transfer does not occur in conference graphs.
\end{proof}

\begin{remark}\label{rem:DRGs}
We have used that the spectral idempotents of a strongly regular graph can be expressed as linear combinations of $I, A$ and $J-A-I$. Since the eigenvalues and the coefficients in these expressions for the spectral idempotents only depend on the parameter set $(n,k,a,c)$, the existence of peak state transfer is also determined by $(n,k,a,c)$. In other words, if $X$ is a strongly regular graph with parameter set $(n,k,a,c)$ that admits peak state transfer between some pair of vertices, then any other strongly regular graph with parameter set $(n,k,a,c)$ also admits peak state transfer between some pair of vertices. 

More generally, the same holds for distance-regular graphs (for definitions, see \cite{BroCohNeu1989}). For a distance-regular graph of diameter $d$, the eigenvalues and entries in the spectral idempotents are determined by the intersection array $\{b_0,\ldots,b_{d-1}:c_1,\ldots, c_d\}$. Thus, the existence of peak state transfer is the same for all distance-regular graphs having the same intersection array. 
\end{remark}

\subsection{Block designs}
\label{subsec:blockdesigns}

A \textsl{block design} (also known as a \textsl{$2$-design} or \textsl{balanced incomplete block design}) with parameters $v,b,r,k$ and $\lambda$ is an incidence structure on a set $\mathcal{P}$, consisting of $v$ elements (the \textsl{points}), and a collection $\mathcal{B}$, consisting of $b$ subsets of $\mathcal{P}$ of size $k$ (the \textsl{blocks}), such that each point is in $r$ blocks and each pair of distinct points is in exactly $\lambda$ blocks. See e.g.\ \cite{GodRoy2001} or \cite{ColDin2007} for more information on block designs. Clearly $vr = bk$. Moreover, we will assume that $1 < k < v$ to omit some trivialities, and then one can show that
\begin{equation}
\label{eq:2design_parameter_relations}
\frac{v(v-1)}{k(k-1)} = \frac{b}{\lambda} \quad \text{and} \quad \frac{v-1}{k-1} = \frac{r}{\lambda}
\end{equation}
so that $b$ and $r$ are completely determined by $v$, $k$ and $\lambda$. Therefore, we will call a block design with specific parameters a \textsl{$(v,k,\lambda)$-design}. Given such a design, let $N \in \{0,1\}^{v \times b}$ be the point-block incidence matrix, that is,
\[
N(x,y) = \begin{cases}
1 &\text{if point $x$ is on block $y$}; \\
0 &\text{otherwise.}
\end{cases}
\]
We will write $J_{c \times d}$ for the $c \times d$ all-ones matrix, or simply $J$ if it is clear from the context what its  dimensions are. We have
\begin{align}
\label{eq:NNT_expression_JI} NN^T &= \lambda J + (r - \lambda)I  \\
\label{eq:NNT_spectral_decomposition} &= rk \cdot \frac{1}{v}J + (r-\lambda)\cdot \left(I - \frac{1}{v}J\right),
\end{align}
where the last expression is the decomposition of $NN^T$. Since $r \neq \lambda$ by our assumptions, the matrix $N$ has full row rank, which implies that $v \leq b$. Note further that for any positive integer $c$, we have
\begin{equation}
\label{eq:N_times_J}
NJ_{b \times c} = rJ_{v \times c} \quad \text{and} \quad N^T J_{v \times c} = kJ_{b\times c}.
\end{equation}

Consider now the matrix
\[
A = \begin{bmatrix}
0 & N \\
N^T & 0
\end{bmatrix},
\]
which is the adjacency matrix of the bipartite graph that represents the block design, with the points in the first part and the blocks in the other. The bipartite graph is biregular, so that the projected transition matrix with respect to the arc-reversal walk on this graph is given by
\[
B = \frac{1}{\sqrt{rk}} A.
\]
In order to determine whether peak state transfer can occur in the arc-reversal walk on this bipartite graph, we need the spectral decomposition of $A$, which is given in the following lemma. Although this result is well-known and deducible in a straightforward manner, it is not commonly state in this form, so we have provided a proof in Appendix \ref{app:proofs}.

\begin{lemma}
\label{lem:2design_idempotents}
Let $X$ be the point-block incidence graph of a block design with parameters $v,b,r,k$ and $\lambda$. If $v < b$, the eigenvalues of its adjacency matrix $A$ are $0, \pm \sqrt{r-\lambda}$ and $\pm\sqrt{rk}$ and their respective spectral idempotents are
\begin{align*}
E_{0} &= 
\begin{bmatrix}
0 & 0 \\
0 & I - \frac{1}{r-\lambda}\left(N^TN - \frac{\lambda v}{b}J\right)
\end{bmatrix}; \\
E_{\pm \sqrt{r-\lambda}} &= 
\frac{1}{2}
\begin{bmatrix}
I - \frac{1}{v}J & \pm\frac{1}{\sqrt{r-\lambda}} \left(N - \frac{k}{v}J\right)\\
\pm \frac{1}{\sqrt{r-\lambda}} \left(N - \frac{k}{v}J\right) & \frac{1}{r-\lambda}\left(N^TN - \frac{rk}{b}J\right)
\end{bmatrix}
; \quad \text{and}\\
E_{\pm \sqrt{rk}} &= \frac{1}{2}
\begin{bmatrix}
\frac{1}{v}J & \pm \frac{k}{v\sqrt{rk}} J \\
\pm \frac{k}{v\sqrt{rk}} J & \frac{1}{b}J
\end{bmatrix}.
\end{align*}
If $v = b$, then its eigenvalues are $\pm \sqrt{r-\lambda}$ and $\pm \sqrt{rk}$, with the same spectral idempotents as above.
\end{lemma}

The restrictions on the parameters $v$, $k$ and $\lambda$ make it so that $E_{\pm \sqrt{r-\lambda}}(u,v) \neq 0$
whenever $u$ represents a point (and $v$ can correspond to any point or block) in the incidence graph. This means that peak state transfer from $u$ to another vertex can only occur if the eigenvalue $\sqrt{\frac{r-\lambda}{rk}}$ can be written as a cosine of a rational multiple of $\pi$. As for peak state transfer between blocks, the lower-right matrix
\[
\frac{1}{r-\lambda}\left(N^TN - \frac{rk}{b}J\right)
\]
in the expressions for $E_{\pm \sqrt{r-\lambda}}$ in the lemma above can have entries that are equal to zero: an off-diagonal entry of $N^TN$ represents the number of points that are in both corresponding distinct blocks, an it is possible for this number to equal $rk/b$. However, the matrix $N^T N$ depends on more than just the parameters of the block design: it is possible for one realization of a $(v,k,\lambda)$-design to admit peak state transfer between a pair of blocks whereas another does not. For this reason, our main result of this section is restricted to peak state transfer involving a point of the design.

\begin{theorem}
\label{thm:block_designs_peakST}
Consider the arc-reversal walk on the point-block incidence graph $X$ of a $(v,k,\lambda)$-design. There is peak state transfer starting at a vertex corresponding to a point if and only if $v = 3$ and $k = 2$ or $v = 9$ and $k = 3$.
\end{theorem}
\begin{proof}
Assume that there is peak state transfer from vertex $x$ to vertex $y$ at some time, where $x$ corresponds to a point of the design. To determine the mutual eigenvalue support of $x$ and $y$, we consider the spectral idempotents $E_\mu$ of the adjacency matrix $A$ of $X$, as given in Lemma \ref{lem:2design_idempotents}. From $E_0$, we can see that the eigenvalue $0$ is not in their mutual eigenvalue support. The top row of blocks of the idempotents $E_{\pm \sqrt{rk}}$ clearly do not contain a zero, and the same is true for the top row of blocks of the idempotents $E_{\pm \sqrt{r - \lambda}}$, since we are assuming that $1 < k < v$. As $B = \frac{1}{\sqrt{rk}}A $ has the same spectral idempotents as $A$, this means that its eigenvalues $\pm \sqrt{\frac{r-\lambda}{rk}}$ and $\pm 1$ are in the mutual eigenvalue support of $x$ and $y$. Hence, by Theorem \ref{thm:PeakSTchar}, they must be cosines of rational multiple of $\pi$. This trivially holds for the eigenvalues $\pm 1$. This condition is very restrictive for the eigenvalue $\sqrt{\frac{r-\lambda}{rk}}$, however: if there exists $q \in \rats$ such that
\[
\sqrt{\frac{r-\lambda}{rk}} = \cos(q\pi),
\]
then 
\[
\frac{r-\lambda}{rk} = \cos^2(q\pi) = \frac{1}{2}\cos(2q \pi) + \frac{1}{2}.
\]
By Lemma \ref{lem:conway_jones}, since the expression on the right is rational, this can only happen if
\[
\frac{r-\lambda}{rk} \in \left\{0,\tfrac{1}{4}, \tfrac{1}{2}, \tfrac{3}{4}, 1\right\}.
\]
We are assuming that $\lambda < r$, so this eigenvalue cannot be 0. Moreover,
\[
\frac{r-\lambda}{rk} = \frac{1}{k} - \frac{\lambda}{rk} < \frac{1}{k} \leq \frac{1}{2},
\]
since we also assume that $k > 1$ This implies that $\frac{r - \lambda}{rk} = \frac{1}{4}$ is the only solution, and that the eigenvalues of $B$ are $-1,-\frac{1}{2},0,\frac{1}{2}$ and $1$. Using \eqref{eq:2design_parameter_relations}, we can rewrite
\[
\begin{split}
4(r-\lambda) = rk \,\, &\iff \,\, 4\lambda\left(\frac{v-1}{k-1}-1\right) = \frac{v-1}{k-1}\cdot \lambda k \\ 
&\iff \,\, 4(v-1) - 4(k-1) = (v-1)k \\
&\iff \,\, v = \frac{3k}{4-k}
\end{split}
\]
Then $k<4$ and we already ruled out $k=1$, so there are two cases: $k=2$, which implies $v=3$, and $k=3$, which implies $v=9$.

Conversely, assume that either of the cases in the statement of the theorem holds, so we have a $(3,2,\lambda)$-design or a $(9,3,\lambda)$-design. The following eigenvalues are in the mutual eigenvalue support of $x$ and $y$:
\[
-1 = \cos \frac{1 \cdot \pi}{1},\quad
-\frac{1}{2} = \cos\frac{2\cdot \pi}{3}, \quad
\frac{1}{2} = \cos\frac{1 \cdot \pi}{3}, \quad
1 = \cos\frac{0\cdot \pi}{1}
\]
By Theorem \ref{thm:PeakSTchar}, if there is peak state transfer between $x$ and $y$, it must occur for the first time at time $3$. Since the value of $\gamma$ in Theorem \ref{thm:PeakSTchar}(ii) is $1$, their is peak state transfer if and only if $-\frac{1}{2}$ and $1$ are in the positive mutual eigenvalue support of $x$ and $y$ and $-1$ and $\frac{1}{2}$ are in the negative eigenvalue support. Indeed, if $y$ is a block that $x$ is not incident to, we can see from Lemma \ref{lem:2design_idempotents} that $1$ is in their positive mutual eigenvalue support $-1$ is in their negative mutual eigenvalue support. Further, since $N(x,y) = 0$ we have that $E_{\sqrt{r-\lambda}}(x,y) < 0$ and $E_{-\sqrt{r-\lambda}}(x,y) > 0$, so that the eigenvalues $\frac{1}{2}$ and $-\frac{1}{2}$ are respectively in the positive and negative mutual eigenvalue supports of $x$ and $y$. This completes the proof.
\end{proof}

Since the number of blocks of a $(3,2,1)$-design is $3$, the only possible design is, up to permutations of points and blocks, given by the incidence matrix
\[
N = \begin{bmatrix}
0 & 1 & 1 \\
1 & 0 & 1 \\
1 & 1 & 0
\end{bmatrix}.
\]
Note that the corresponding point-block incidence matrix corresponds to the $6$-cycle, for which the occurrence of peak state transfer was alternatively shown in Lemma \ref{lem:cycles_peakST}. The unique realisation of a $(9,3,1)$-design is provided in \cite{ColDin2007} and we give a visualization of it in Figure \ref{fig:affine_plane_peakST}.
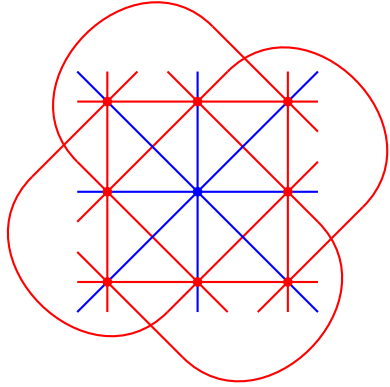
\begin{figure}
\centering
\begin{tikzpicture}[scale=.8]
	\begin{pgfonlayer}{nodelayer}
		\node (0) at (-1.5, 1.5) {};
		\node (1) at (0, 1.5) {};
		\node (2) at (1.5, 1.5) {};
		\node (3) at (-1.5, 0) {};
		\node (4) at (0, 0) {};
		\node (5) at (1.5, 0) {};
		\node (6) at (-1.5, -1.5) {};
		\node (7) at (0, -1.5) {};
		\node (8) at (1.5, -1.5) {};
		\node (9) at (-2, 2) {};
		\node (10) at (2, 2) {};
		\node (11) at (1.5, 2) {};
		\node (12) at (2, 1.5) {};
		\node (13) at (-2, 1.5) {};
		\node (14) at (-1.5, 2) {};
		\node (15) at (-2, -2) {};
		\node (16) at (-2, -1.5) {};
		\node (17) at (-1.5, -2) {};
		\node (18) at (1.5, -2) {};
		\node (19) at (2, -1.5) {};
		\node (20) at (2, -2) {};
		\node (21) at (-2, 0) {};
		\node (22) at (2, 0) {};
		\node (23) at (0, 2) {};
		\node (24) at (0, -2) {};
		\node (25) at (-2, -0.5) {};
		\node (26) at (0.5, 2) {};
		\node (27) at (1, -2) {};
		\node (28) at (2.75, -0.25) {};
		\node (29) at (0.5, -2) {};
		\node (30) at (-0.5, 2) {};
		\node (31) at (2, 0.5) {};
		\node (32) at (2, -0.5) {};
		\node (33) at (-2, 0.5) {};
		\node (34) at (-0.5, -2) {};
		\node (35) at (0.25, 2.75) {};
		\node (36) at (2, 1) {};
		\node (37) at (-2, -1) {};
		\node (38) at (-1, 2) {};
		\node (39) at (-2.75, 0.25) {};
		\node (40) at (-0.25, -2.75) {};
	\end{pgfonlayer}
	\begin{pgfonlayer}{edgelayer}
		\draw [thick, red] (13.center) to (12.center);
		\draw [thick, red] (16.center) to (19.center);
		\draw [thick, red] (14.center) to (17.center);
		\draw [thick, red] (11.center) to (18.center);
		\draw [thick, red] (25.center) to (26.center);
		\draw [thick, red, bend left=90, looseness=1.50] (26.center) to (28.center);
		\draw [thick, red] (27.center) to (28.center);
		\draw [thick, red] (29.center) to (33.center);
		\draw [thick, red, bend left=270, looseness=1.50] (35.center) to (33.center);
		\draw [thick, red] (35.center) to (36.center);
		\draw [thick, red] (31.center) to (34.center);
		\draw [thick, red, bend right=270, looseness=1.50] (34.center) to (39.center);
		\draw [thick, red] (39.center) to (38.center);
		\draw [thick, red] (30.center) to (32.center);
		\draw [thick, red, bend left=90, looseness=1.50] (32.center) to (40.center);
		\draw [thick, red] (40.center) to (37.center);
        \draw [thick, blue] (23.center) to (24.center);
        \draw [thick, blue] (21.center) to (22.center);
        \draw [thick, blue] (9.center) to (20.center);
        \draw [thick, blue] (15.center) to (10.center);
	\end{pgfonlayer}
    \filldraw [red]
        (0) circle (2pt)
        (1) circle (2pt)
        (2) circle (2pt)
        (3) circle (2pt)
        (5) circle (2pt)
        (6) circle (2pt)
        (7) circle (2pt)
        (8) circle (2pt);
    \filldraw [blue]
        (4) circle (2pt);
\end{tikzpicture}
\vspace{-20pt}
\caption{A depiction of the affine plane of order $3$, the unique realisation of a $(9,3,1)$-design. Each block is indicated by a line segment that passes through its $3$ incident points. In the corresponding point-block incidence graph, there is peak state transfer from the blue point to the red blocks at time $3$, i.e.\ to all blocks that the blue point is not incident to. \label{fig:affine_plane_peakST}}
\end{figure}

Other $(3,2,\lambda)$- and $(9,3,\lambda)$-designs can be obtained from these by taking multiple copies of all blocks. However, not every such design arises in this way: for instance, there are $36$ non-isomorphic $(9,3,2)$-designs \cite{ColDin2007}. As an example, consider the $(9,3,2)$-design formed by the following $3$-subsets of $\{1,\ldots,9\}$ as blocks, taken from \cite{ColDin2007}:
\[
\begin{split}
&\{1,2,3\},
\{1,2,4\},
\{1,3,4\},
\{1,5,6\},
\{1,5,7\},
\{1,6,8\},
\{1,7,9\},
\{1,8,9\},
\{2,3,4\},
\{2,5,6\},
\{2,5,7\},
\{2,6,8\}, \\
&\{2,7,9\},
\{2,8,9\},
\{3,5,8\},
\{3,5,9\},
\{3,6,7\},
\{3,6,9\},
\{3,7,8\},
\{4,5,8\},
\{4,5,9\},
\{4,6,7\},
\{4,6,9\},
\{4,7,8\}.
\end{split}
\]
Indeed, we can see that $b = 24$ (the amount of blocks) and $r = 8$ (the amount of times each number appears in a block) agree with the equations in \eqref{eq:2design_parameter_relations}. Moreover, every pair of numbers appears in two blocks, so it is a valid $(9,3,2)$-design.

\subsection{Infinite family of graphs with peak state transfer tending to 
\texorpdfstring{$1$}{1}}
\label{subsec:peakST_arc_reversal_approaches_1}

We will now give an example of an infinite family of graphs for which the amount of peak state transfer tends to $1$, meaning that the quantity in \eqref{eq:intro-peak-quantity} tends to $0$ as the number of vertices grows.

Consider the following family of graphs, obtained by blowing up the vertices $x$ and $y$ of the path graph $P_5$ as depicted in Figure \ref{fig:infinite_family_peakST}.
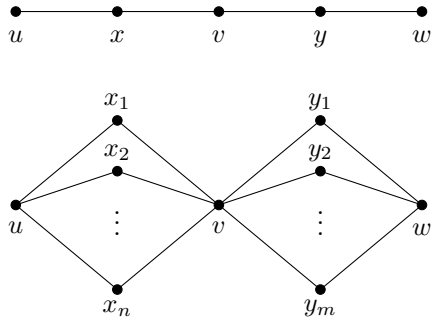
\begin{figure}
\centering
\begin{tikzpicture}[scale=.9]
	\begin{pgfonlayer}{nodelayer}
		\node [label=below:$v$] (0) at (0, 0) {};
		\node [label=below:$x$] (1) at (-1.5, 0) {};
		\node [label=below:$u$] (2) at (-3, 0) {};
		\node [label=below:$y$] (3) at (1.5, 0) {};
		\node [label=below:$w$] (4) at (3, 0) {};
	\end{pgfonlayer}
	\begin{pgfonlayer}{edgelayer}
		\draw (2.center) to (4.center);
	\end{pgfonlayer}
    \filldraw[black]
        (0) circle (2pt)
        (1) circle (2pt)
        (2) circle (2pt)
        (3) circle (2pt)
        (4) circle (2pt);
\end{tikzpicture}

\vspace{10pt}

\begin{tikzpicture}[scale=.9]
	\begin{pgfonlayer}{nodelayer}
		\node [label={[shift={(0,-.3)}]center:{$v$}}] (0) at (0, 0) {};
		\node [label={[shift={(0,.25)}]center:{$x_1$}}] (1) at (-1.5, 1.25) {};
		\node [label={[shift={(0,-.3)}]center:{$u$}}] (2) at (-3, 0) {};
		\node [label={[shift={(0,.25)}]center:{$y_1$}}] (3) at (1.5, 1.25) {};
		\node [label=below:$w$] (4) at (3, 0) {};
		\node [label={[shift={(0,.25)}]center:{$x_2$}}] (5) at (-1.5, 0.5) {};
		\node [label={[shift={(0,.25)}]center:{$y_2$}}] (6) at (1.5, 0.5) {};
		\node [label={[shift={(0,-.25)}]center:{$x_n$}}] (7) at (-1.5, -1.25) {};
		\node [label={[shift={(0,-.25)}]center:{$y_m$}}] (8) at (1.5, -1.25) {};
		\node (9) at (-1.5, -0.15) {$\vdots$};
		\node (10) at (1.5, -0.15) {$\vdots$};
	\end{pgfonlayer}
	\begin{pgfonlayer}{edgelayer}
		\draw (2.center) to (1.center);
		\draw (1.center) to (0.center);
		\draw (0.center) to (3.center);
		\draw (3.center) to (4.center);
		\draw (2.center) to (5.center);
		\draw (5.center) to (0.center);
		\draw (0.center) to (6.center);
		\draw (6.center) to (4.center);
		\draw (2.center) to (7.center);
		\draw (7.center) to (0.center);
		\draw (0.center) to (8.center);
		\draw (8.center) to (4.center);
	\end{pgfonlayer}
    \filldraw [black]
        (0) circle (2pt)
        (1) circle (2pt)
        (2) circle (2pt)
        (3) circle (2pt)
        (4) circle (2pt)
        (5) circle (2pt)
        (6) circle (2pt)
        (7) circle (2pt)
        (8) circle (2pt);
\end{tikzpicture}
\caption{The path graph $P_5$ and its blown-up counterpart $G_{n,m}$. There is peak state transfer from $u$ to $w$ at time $4$ in the arc-reversal walk on $G_{n,m}$. If $n = m$ the peak state transfer is perfect state transfer. For the graphs $G_{n,n+1}$, the amount of state transfer converges to $1$ as $n \to \infty$.\label{fig:infinite_family_peakST}}
\end{figure}
There are $n$ copies of $x$ and $m$ copies of $y$, resulting in a graph on $n + m + 3$ vertices, which we will denote by $G_{n,m}$. Its adjacency matrix is given by
\[
A = \left[\begin{array}{c|c|c|c|c}
& \ones_n^T & & & \\ \hline
\ones_n & & \ones_n & & \\ \hline
& \ones_n^T & & \ones_m^T & \\ \hline
& & \ones_m & & \ones_m \\ \hline
& & & \ones_m^T & 
\end{array}\right],
\]
the first, middle and last columns corresponding to vertices $u$, $v$ and $w$ respectively and the blocks of size $n$ and $m$ lying in between corresponding to the $x_i$ and $y_j$ respectively. We will show that in the arc-reversal walk on $G_{n,m}$, there is peak state transfer from $u$ to $w$ at time $4$ for all choices of $n$ and $m$. The projected transition matrix is the normalized adjacency matrix $B = D^{-1/2} A D^{-1/2}$, where $D$ is the degree matrix of the graph. By considering its structure, one can see that $B$ has rank $4$. As mentioned in Remark \ref{rem:gamma=1}, it has the eigenvector
\[
D^{1/2}\ones = \begin{bmatrix}\sqrt{n} & \sqrt{2}\cdot\ones_n^T & \sqrt{n+m} & \sqrt{2}\cdot\ones_m^T & \sqrt{m}\end{bmatrix}^T
\]
corresponding to its Perron-Frobenius eigenvalue $1$. The vector above has norm $2\sqrt{n+m}$, so that the $(u,w)$-entry of the corresponding spectral idempotent $E_1$ is equal to
\[
E_1(u,w) = \frac{\sqrt{nm}}{4(n + m)}.
\]
Since $G_{n,m}$ is a bipartite graph, the eigenvalues of $B$ are symmetric around $0$; there is also the eigenvector
\[
\begin{bmatrix}\sqrt{n} & -\sqrt{2}\cdot\ones_n^T & \sqrt{n+m} & -\sqrt{2}\cdot\ones_m^T & \sqrt{m}\end{bmatrix}^T
\]
for eigenvalue $-1$, so that the $(u,w)$-entry of the corresponding spectral idempotent $E_{-1}$ is equal to that of $E_1$. The remaining of non-zero eigenvalues is given by $\pm \sqrt{2}/2$, with corresponding eigenvectors
\[
\begin{bmatrix}\sqrt{m} & \pm\sqrt{m/n}\cdot\ones_n^T & 0 & \mp\sqrt{n/m}\cdot\ones_m^T & -\sqrt{n}\end{bmatrix}^T,
\]
(each having norm $\sqrt{2(n+m)}$), yielding
\[
E_{\pm \sqrt{2}/2}(u,w) = -\frac{\sqrt{nm}}{2(n+m)}.
\]
Finally, since $\sum_{\theta \in \ev(B)} E_{\theta}(u,w) = 0$, this means that the idempotent for the remaining eigenvalue $0$ has as its $(u,w)$-entry:
\[
E_0(u,w) = \frac{\sqrt{nm}}{2(n+m)}.
\]
We conclude that the positive and negative mutual eigenvalue supports of $u$ and $v$ are respectively $\Lambda^1 = \{-1,0,1\}$ and $\Lambda^{-1} = \{-\sqrt{2}/2,\sqrt{2}/2\}$. Since the eigenvalues can alternatively be expressed as $\cos\frac{k\pi}{4}$ for $k = 0,\ldots,4$, the odd $k$ corresponding precisely to the eigenvalues in the negative mutual eigenvalue support, there is peak state transfer from $u$ to $w$ at time $4$ by Theorem \ref{thm:PeakSTchar}. In particular, we have
\[
|B_4(u,w)| = \sum_{\theta \in \ev(B)} |E_\theta(u,w)| = \frac{2\sqrt{nm}}{n+m}.
\]
If we take $m = n+1$. If $n = m$, this value is equal $1$ and there is perfect state transfer. If we take for instance $m = n+1$, we obtain the sequence of graphs $G_{n,n+1}$ that has peak state transfer with
\[
B_4(u,w) = \sqrt{1 - \frac{1}{(2n+1)^2}}.
\]
Thus, $B_4(u,w)$ converges to $1$ as $n$ goes to $\infty$.

\subsection{Other sporadic examples}
\label{subsec:other_examples}

Since for a $k$-regular graph with adjacency matrix $A$, the transition matrix for its arc-reversal walk is given by $B = \frac{1}{k}A$, we have that $k$-regular graphs with eigenvalues in the set $\{-k,-k/2,0,k/2,k\}$ are straightforward candidates for peak state transfer to occur, as $-1,-1/2,0,1/2$, and $1$ are all cosines of rational multiples of $\pi$ (and they are in fact the only rational numbers with this property, by Lemma \ref{lem:conway_jones}(i)). Considering Theorem \ref{thm:PeakSTchar}, peak state transfer in such graphs occurs for the first time at time at most $6$, as $0 = \cos\frac{1 \cdot \pi}{2}$ and $\frac{1}{2} = \cos\frac{1 \cdot \pi}{3}$. Examples of regular graphs with few eigenvalues are distance-regular graphs (see \cite{BroCohNeu1989}). We have already considered the distance-regular graphs of diameter at most $2$; they are the strongly regular graphs that we discussed in Section \ref{subsec:SRGs}. A distance-regular graph of diameter $d$ has $d+1$ distinct eigenvalues, so searching the tables for distance-regular graphs of diameter $3$ and $4$ in Chapter 14 of \cite{BroCohNeu1989} may provide us some instances of peak state transfer.

Indeed, an example is the Hamming graph $H(3,3)$ (also defined in \cite{BroCohNeu1989}), with vertex set $\{0,1,2\}^3$, where two vertices are adjacent if they differ in exactly one coordinate. It has diameter $3$, with spectrum $\{-3,0,3,6\}$. Using SageMath \cite{sagemath} we determined that this graph of order 27 has peak state transfer from any vertex $v$ to all vertices at distance $2$ from $v$ at time $6$, and the walk is periodic with period $12$.

Interestingly, this graph has a cospectral mate that we found on House of Graphs \cite{CooDhoGoe2023} (its HoG Id is 1118). In the literature, it  is known as the dual Menger graph of the Gray configuration; this graph has 27 vertices, one for each  line of the configuration, and  two vertices being adjacent if the corresponding lines are incident to a common point. See \cite{MarPisTom2005} for a more detailed description of the configuration. There is peak state transfer from any vertex $v$ to $12$ of the vertices at distance $2$ at time $3$, and to the $4$ remaining vertices at distance $2$ at time $6$.

Another example can be found in \cite{BroCohNeu1989}: the folded $8$-cube has $2^7 = 128$ vertices and can be obtained from the $7$-dimensional hypercube by adding an edge between all antipodal pairs of vertices. It has diameter $4$ and its spectrum is $\{-8,-4,0,4,8\}$. With Sagemath, we determined that there is peak state transfer from any vertex $v$ to all vertices at distance $3$ from $v$ at time $3$ and the walk is periodic with period $6$. 

Yet another pair of examples is given by two instances of partial geometries with parameters $(5,5,2)$. (See e.g.\ \cite{ColDin2007} for information on this type of incidence structure.) The corresponding graph is the bipartite point-line incidence graph of the partial geometry. The specific instances can be found in \cite{LinSch1981} and \cite{Ved2021}, and have as point set the elements of the finite field $\mathbb{F}_{81}$ and the line set consists of certain $6$-subsets of its elements, where the incidence relation is given by inclusion. For both instances, when starting at any point (line) in the point-line incidence graph, there is peak state transfer at time $3$ to all lines (points) at distance $3$ in that graph. There is also peak state transfer to $75$ points (lines) at distance $2$ and at time $6$ to $30$ other points, with the walk being periodic with period $12$.

\section{State transfer in vertex-face walks in grids}\label{sec:grids}

In this section, we discuss state transfer in vertex-face walks on toroidal grids. For $n,m \geq 3$, the \textsl{(toroidal) $(n,m)$-grid} is an embedding of $C_n \boxprod C_m$ in the torus. We saw an example of the $(4,6)$-grid in Figure \ref{fig:peak_ST_illustration} in Section \ref{sec:intro}.  The definition can be extended to where $n$ and $m$ can also take on the values $1$ and $2$, in which case the embeddings are those of multigraphs; if $n=2$ or $m=2$, the corresponding multigraph has doubled edges and if $n=1$ or $m=1$ it has loops. A precise description in terms of rotation systems is given, for instance, in \cite{GuoSch2024}. We denote the vertices of the $(n,m)$-grid by the elements of the group $V:= \ints_n \times \ints_m$. (Alternatively, the $(n,m)$-grid is an embedding of the Cayley graph of $\ints_n \times \ints_m$ with generating set $\{\pm (1,0),\pm (0,1)\}$.)

In Section \ref{subsec:grids_periodicity_classification}, we classify periodicity as well as perfect state transfer of the vertex-face walk on toroidal grids. In Section \ref{subsec:peakST_grids_inf_fam}, we show that every $(4,n)$-grid admits peak state transfer for every $n \geq 1$.

\subsection{Classification of periodicity}
\label{subsec:grids_periodicity_classification}

The authors have shown in a previous paper \cite{GuoSch2024} that the $(1,m)$- and $(2,m)$-grids admit perfect state transfer and periodicity. Here, we will complete the classification of toroidal grids on which the vertex-face walks admits periodicity. The main result of this section is the following theorem:
\begin{theorem}
\label{thm:toroidal_periodicity_characterisation}
The vertex-face walk on the toroidal $(n,m)$-grid is periodic if and only if $\min\{n,m\} \leq 2$ or $n = m = 4$.
\end{theorem}
Since the automorphism group of the $(n,m)$-grid acts transitively on its vertex set (i.e.\ ``every vertex looks the same''), periodicity at one vertex implies periodicity at every vertex. Moreover, by \cite[Theorem 5.2]{GuoSch2024}, since the number of faces of the $(n,m)$-grid is equal to the number of vertices, periodicity at every vertex at time $\tau$ is equivalent to the transition matrix $U$ having the property that $U^\tau = I$. We will simply say that the vertex-face walk on a toroidal grid is `periodic' if these (equivalent) properties occur. Note also that perfect state transfer from vertex $u$ to vertex $v$ at time $\tau$ implies periodicity at $u$ and $v$ at time $2\tau$, since perfect state transfer is symmetric. This means that if periodicity does not occur, perfect state transfer does not occur either.

We will use the remainder of this subsection to prove the theorem above. In Figure \ref{fig:periodicity_proof_diagram}, we have provided a diagram that shows how the results in this section are related to each other, culminating in the proof of Theorem \ref{thm:toroidal_periodicity_characterisation}. In the results on the left side of the diagram, we use the characteristic polynomial of the transition matrix $U$ in order to rule out periodicity from occuring in $(n,m)$-grids whenever $nm$ is not divisible by $8$. In the results on the right side of the figure, we then argue that if some grid admits periodicity, and if $n$ or $m$ is divisible by $4$, then it must be the case that $n = m = 4$. This part of the argument relies on our characterization of periodicity in two-reflection walks (Theorem \ref{thm:periodicity_characterisation}). We also make use of Lemma \ref{lem:conway_jones} in order to solve equations involving cosines of rational multiples of $\pi$.

\begin{figure}
\centering
\begin{tikzpicture}
		\node[draw, thick] (0) at (0, 0) {\textbf{Thm \ref{thm:toroidal_periodicity_characterisation}}};
		\node[draw] (1) at (-2.5, 0) {Lem \ref{lem:periodic_implies_div8}};
		\node[draw] (2) at (-5, 1) {Lem \ref{lem:trace_U^2}};
		\node[draw] (3) at (-5, -1) {Rem \ref{rem:U_periodic_integer_coeffs}};
		\node[draw] (4) at (2.5, 1) {Prop \ref{prop:4_4grid_periodic}};
		\node[draw] (5) at (2.5, -1) {Lem \ref{lem:periodic_implies_div8}};
		\node[draw] (6) at (5, 1) {Thm \ref{thm:periodicity_characterisation}};
		\node[draw] (7) at (5, -1) {Lem \ref{lem:cycle_eigenvectors}};
		\node[draw] (8) at (2.5, -2) {Lem \ref{lem:conway_jones}};
        \node[draw] (9) at (2.5, 2) {Thm \ref{thm:PeakSTchar}};
        
        \draw[->] (1.east) to (0.west);
        \draw[->] (2.south east) to (1.north west);
        \draw[->] (3.north east) to (1.south west);
        \draw[->] (4.south west) to (0.north east);
        \draw[->] (5.north west) to (0.south east);
        \draw[->] (6.west) to (4.east);
        \draw[->] (7.west) to (5.east);
        \draw[->] (8.north) to (5.south);
        \draw[->] (7.north west) to (4.south east);
        \draw[->] (6.south west) to (5.north east);
        \draw[->] (9.south) to (4.north);
\end{tikzpicture}
\caption{Organization of the building blocks for the proof of Theorem \ref{thm:toroidal_periodicity_characterisation}.   \label{fig:periodicity_proof_diagram}}
\end{figure}
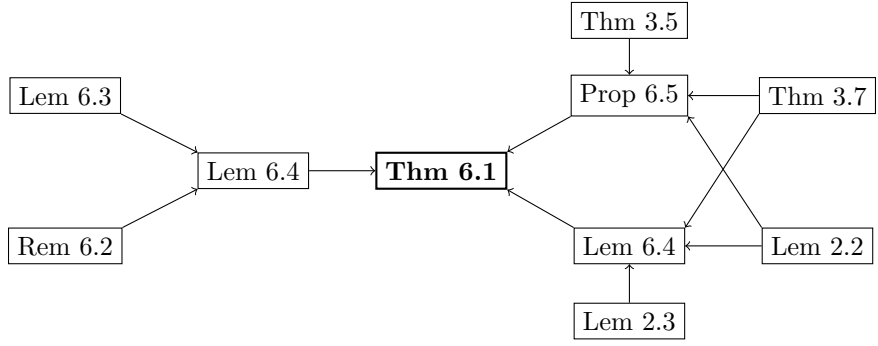

For $n \geq 1$, let $\zeta$ be a primitive $n$-th root of unity. The \textsl{$n$-th cyclotomic polynomial} $\Phi_n$, defined by
\[
\Phi_n(x) = \prod_{k \in [n]:\, \gcd(k,n) = 1} \left(x - \zeta^k \right),
\]
is irreducible over $\rats$. As such, $\Phi_n$ is the minimal polynomial of $\zeta$ over $\rats$. Moreover, all coefficients of $\Phi_n$ are integer. See \cite{DumFoo2004}, in particular Section 13.6. We observe the following, where $\phi(A,x)$ denotes the characteristic polynomial of a matrix $A$.

\begin{remark}
\label{rem:U_periodic_integer_coeffs}
If $A$ is an $n \times n$ matrix, with rational entries, such that $A^\tau = I$ for some integer $\tau \geq 0$, then all eigenvalues of $A$ are roots of unity, and the characteristic polynomial $\phi(A,x)$ of $A$ factors into cyclotomic polynomials over $\rats$. This means in particular that all coefficients of $\phi(A,x)$ are integer.
\end{remark}

Since the transition matrix $U$ of the vertex-face walk on the $(n,m)$-grid has rational entries, we can use Remark \ref{rem:U_periodic_integer_coeffs} to check if periodicity can occur: if $\phi(U,x)$ has a coefficient that is non-integer, then the grid is not periodic. We will use this approach to rule out most of the toroidal $(n,m)$-grids from being periodic.

For $U$, which is a $4nm \times 4nm$ matrix, we can write
\[
\phi(U,x) = \sum_{i=0}^{4nm} c_i x^i = \prod_{j=1}^{4nm} (x - \lambda_j),
\]
i.e.\ the $c_i$ are the coefficients of $\phi(U,x)$, and the $\lambda_j$ are the eigenvalues of $U$ (with multiplicity). It is clear that $c_{4nm} = 1$ and $c_{4nm-1} = -\sum_{j}\lambda_j = -\tr(U)$. Moreover,
\begin{equation}
\label{eq:char_3rd_coeff}
c_{4nm-2} = \sum_{j < k}\lambda_j\lambda_k = \frac{1}{2}\left(\sum_{j,k=1}^{4nm} \lambda_j\lambda_k - \sum_{j=1}^{4nm}{\lambda_j^2}\right) = \frac{\tr(U)^2 - \tr(U^2)}{2}.
\end{equation}

To compute the coefficients $c_{4nm-1}$ and $c_{4nm-2}$, we need to determine the traces of $U$ and $U^2$. For $n,m \geq 3$, the facial boundary walks of the $(n,m)$-grid are cycles in the graph (we say that the embedding is \textsl{circular}; note that this is not the case if $n=1$ or $2$). With the expression for the trace of the transition matrix $U$ of the vertex-face walk, as given in \cite[Lemma 2.3]{Zha2021}, we find that
\begin{equation}
\label{eq:trace_U}
\tr(U) = \frac{2|V||F|}{|E|} - 2|V| - 2|F| + |\cA| = nm - 4nm + 4nm = nm.
\end{equation}
This means that the coefficient $c_{4nm-1}$ is integer for all $n,m$. We will now compute the value for $c_{4nm-2}$, which turns out to be non-integer if $nm$ is not divisible by $8$.

For $n,m \geq 3$, let $C \in \{0,1\}^{V \times F}$ be the \textsl{vertex-face incidence matrix} of the vertex-face walk on the $(n,m)$-grid, where
\begin{equation}
\label{eq:C_definition}
C(v,f) = \begin{cases}
$1$ &\text{if $v$ is on the boundary of $f$;} \\
$0$ &\text{otherwise.}
\end{cases}
\end{equation}
Equivalently, $C = 4D$, where $D = N^TM$ is the discriminant matrix of the vertex-face walk.  We make use of the vertex-face incidence matrix to determine the trace of $U^2$:
\begin{lemma}
\label{lem:trace_U^2}
For $n,m \geq 3$, if $U$ is the transition matrix of the toroidal $(n,m)$-grid, then
\[
\tr(U^2) = \frac{9nm}{4}.
\]
\end{lemma}
\begin{proof}

By writing $U^2$ in terms of the orthogonal projections $P$ and $Q$ we find
\[
\begin{split}
U^2 &= (2P - I)(2Q - I)(2P - I)(2Q - I) \\
&= 16PQPQ - 8PQP - 8QPQ - 4PQ + 4QP + I,
\end{split}
\]
where the identities $P^2 = P$ and $Q^2 = Q$ were used to reduce the expression. By the cyclic property of the trace,
\[
\tr(PQP) = \tr(PQ) = \tr(QP) = \tr(QPQ),
\]
so that
\begin{equation}
\label{eq:trace_U^2_PQ}
\tr(U^2) = 16\tr(PQPQ) - 16\tr(PQ) + 4nm,
\end{equation}
as $\tr(I) = 4nm$. We then plug in $P = MM^T$ and $Q = NN^T$ and we apply the cyclic trace property once more to find
\[
\tr(PQ) = \tr(MM^TNN^T) = \tr(N^TMM^TN) = \tr(DD^T)
\]
and similarly
\[
\tr(PQPQ) = \tr(DD^TDD^T)
\]
We can put these values into \eqref{eq:trace_U^2_PQ} to obtain
\begin{align}
\nonumber \tr(U^2) &= 16\tr(DD^TDD^T) - 16\tr(DD^T) + 4nm \\
\label{eq:trace_U^2_CCT} &= \frac{1}{16}\tr(CC^TCC^T) - \tr(CC^T) + 4nm,
\end{align}
where for the second equality, we used the fact that $D = \frac{1}{4}C$, where $C$ is the vertex-face incidence matrix of the toroidal $(n,m)$-grid. What remains is to determine the traces of $CC^T$ and $CC^TCC^T$.

Since $C$ is a binary matrix, we can see from \eqref{eq:C_definition} that
\[
(CC^T)(u,v) = \sum_{f \in F} C(u,f) C(v,f) = \#\{f \in F : \text{both $u$ and $v$ are incident to $f$}\}.
\]
From Figure \ref{fig:u_nbhood_grid} we deduce that for all $u,v \in V$:
\[
(CC^T)(u,v) = \begin{cases}
4 & \text{if $u = v$;} \\
2 & \text{if $uv \in E$;} \\
1 & \text{if $u$ and $v$ lie on diagonally opposite corners of a face;} \\
0 & \text{otherwise.}
\end{cases}
\]
In particular, each row of $CC^T$ contains exactly one $4$ (which is on the diagonal), four $2$'s and four $1$'s. Since the $(n,m)$-grid has $nm$ vertices, we find that $\tr(CC^T) = 4nm$. Finally, every diagonal $(u,u)$-entry of $CC^TCC^T$ is equal to
\[
\sum_{v \in V} ((CC^T)(u,v))^2 = 1 \cdot 4^2 + 4 \cdot 2^2 + 4 \cdot 1^2 = 36,
\]
so that $\tr(CC^TCC^2) = 36nm$. When we substitute these values for the traces of $CC^T$ and $CC^TCC^T$ in \eqref{eq:trace_U^2_CCT}, we finally obtain
\[
\tr(U^2) = \frac{9nm}{4} - 4nm + 4nm = \frac{9nm}{4}.
\]
\begin{figure}
\centering
\begin{tikzpicture}[scale=.8]
	\begin{pgfonlayer}{nodelayer}
		\node (0) [label=below left:$u$] at (0, 0) {};
		\node (1) at (-2, 0) {};
		\node (2) at (2, 0) {};
		\node (3) at (0, 2) {};
		\node (4) at (-2, 2) {};
		\node (5) at (2, 2) {};
		\node (6) at (0, -2) {};
		\node (7) at (-2, -2) {};
		\node (8) at (2, -2) {};
		\node (9) at (-2.5, 2) {};
		\node (10) at (-2.5, 0) {};
		\node (11) at (-2.5, -2) {};
		\node (12) at (-2, 2.5) {};
		\node (13) at (0, 2.5) {};
		\node (14) at (2, 2.5) {};
		\node (15) at (2.5, 2) {};
		\node (16) at (2.5, 0) {};
		\node (17) at (2.5, -2) {};
		\node (18) at (2, -2.5) {};
		\node (19) at (0, -2.5) {};
		\node (20) at (-2, -2.5) {};
	\end{pgfonlayer}
	\begin{pgfonlayer}{edgelayer}
		\draw (9.center) to (15.center);
		\draw (10.center) to (16.center);
		\draw (11.center) to (17.center);
		\draw (12.center) to (20.center);
		\draw (13.center) to (19.center);
		\draw (14.center) to (18.center);
	\end{pgfonlayer}
    \filldraw [black]
    	(0) circle (3pt);
    \filldraw [red]
        (1) circle (3pt)
        (2) circle (3pt)
        (3) circle (3pt)
        (6) circle (3pt);
    \filldraw [blue]
        (4) circle (3pt)
        (5) circle (3pt)
        (7) circle (3pt)
        (8) circle (3pt);
\end{tikzpicture}
\caption{In the toroidal $(n,m)$-grid, if $n,m \geq 3$, every vertex $u$ is incident to four disctinct faces, and has four neighbours (in red) that each share exactly two faces with $u$, and 4 vertices at distance $2$ (in blue) that each share $1$ face with $u$.
\label{fig:u_nbhood_grid}}
\end{figure}
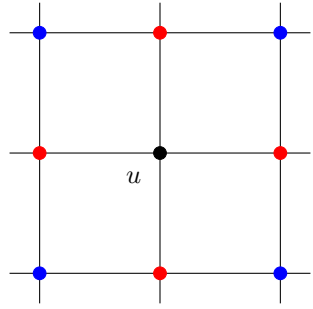
\end{proof}

This completes the first part of our classification of periodicity:

\begin{lemma}
\label{lem:periodic_implies_div8}
For $n,m \geq 3$, if the vertex-face walk on the toroidal $(n,m)$-grid is periodic, then $nm$ is divisible by $8$.
\end{lemma}
\begin{proof}
By \eqref{eq:char_3rd_coeff}, \eqref{eq:trace_U} and Lemma \ref{lem:trace_U^2}, the coefficient for $x^{4nm - 2}$ in the characteristic polynomial of $U$ is equal to
\[
\frac{\tr(U)^2 - \tr(U^2)}{2} = \frac{1}{2}\left((nm)^2 - \frac{9nm}{4}\right) = \frac{1}{8}nm(4nm - 9).
\]
If the $(n,m)$-grid is periodic, this coefficient must be integer by Remark \ref{rem:U_periodic_integer_coeffs}. Since $4nm - 9$ is odd, this happens if and only if $nm$ is divisible by $8$.
\end{proof}
\noindent In other words, if $nm$ is not divisible by $8$, then the $(n,m)$-grid is not periodic.

Now, using the sequence of matrices
\[
B_t = N^T U^t N, \quad t \in \ints_{\geq 0},
\]
we will prove that if the $(n,m)$-grid is periodic, $n,m \geq 3$ and $n$ or $m$ is divisible by $4$, that then $n = m = 4$. In order to do so, we will apply Corollary \ref{thm:periodicity_characterisation}. We first need to determine the eigenvalues of the projected transition matrix $B = B_1$. Recall from Section \ref{sec:prelim} that $B$ can be written in terms of the discriminant matrix $D$, and hence in terms of the vertex-face incidence matrix $C$:
\begin{equation}
\label{eq:B_CCT}
B = 2DD^T - I = \frac{1}{8}CC^T - I.
\end{equation}
By \cite[Lemma 6.2]{GuoSch2024}, the vertices and faces of the toroidal $(n,m)$-grid can be arranged in such a way that we can write
\[
C = (S_n + I_n) \otimes (S_m + I_m),
\]
where $S_k$ denotes the $k \times k$ permutation matrix that sends the standard basis vector $\Ze_{i} \in \cx^{k}$ to $\Ze_{i}$ to $\Ze_{i-1}$ (with the index modulo $k$). We then find that
\begin{align}
\nonumber CC^T &= (S_n + S_n^T + 2I_n) \otimes (S_m + S_m^T + 2I_m) \\
\label{eq:CCT_expression} &= (A_{C_n} + 2I_n) \otimes (A_{C_m} + 2I_m),
\end{align}
where $A_{C_n}$ denotes the adjacency matrix of the $n$-cycle graph. We provided the spectral decomposition of $A_{C_n}$, including the definition of the spectral idempotents $F_{n,k}$, in Lemma \ref{lem:cycle_eigenvectors}, which we will retain here. 

From \eqref{eq:B_CCT}, \eqref{eq:CCT_expression} and Lemma \ref{lem:cycle_eigenvectors}, it follows that the spectrum of $B$ is given by
\[
\sigma(B) = \{\lambda_{k,\ell} \mid 0 \leq k \leq n-1 \text{ and } 0 \leq \ell \leq m-1\},
\]
where
\begin{equation}
\label{eq:toroidal_B_eigs}
\lambda_{k,\ell} \coloneq \frac{1}{2}\left(\cos \frac{2 \pi k}{n} + 1\right)\left(\cos \frac{2 \pi \ell}{m} + 1\right) -1.
\end{equation}
Note that $\sigma(B)$ is the multiset in which each eigenvalue of $B$ occurs with its corresponding multiplicity. The spectral decomposition of $B$ will have the form $\sum_{\theta \in \ev(B)}\theta E_\theta$, where, considering \eqref{eq:CCT_expression}, each $E_\theta$ is written in terms of $A_{C_n}$ and $A_{C_m}$:
\[
E_\theta = \sum_{\substack{k,\ell \\\, \lambda_{k,\ell} = \theta}} F_{n,k} \otimes F_{m,\ell}.
\]
Here, the $F_{n,k}$ and $F_{m,k}$ are as defined in \eqref{eq:cycle_idempotent_entries}. Note that the diagonal elements of these matrices are always strictly positive so that the diagonal entries of the $E_\theta$ are also strictly positive. Hence the eigenvalue support of each vertex $(i,j)$ of the grid is the set of all eigenvalues $\ev(B)$. For periodicity to occur in the toroidal grid, every eigenvalue of $B$ must be a cosine of some rational multiple of $\pi$: considering Remark \ref{rem:periodicity_char_simplified}
the toroidal $(n,m)$-grid is periodic if and only if for all $k$ and $\ell$ there exists some $s_{k,\ell} \in \rats$ such that
\begin{equation}
\label{eq:vxf_grid_cos_equation}
\cos(2\pi s_{k,\ell}) = \frac{1}{2}\left(\cos \frac{2 \pi k}{n} + 1\right)\left(\cos \frac{2 \pi \ell}{m} + 1\right) -1
\end{equation}
We can use this to argue directly that the $(4,4)$-grid is periodic:
\begin{proposition}
\label{prop:4_4grid_periodic}
The vertex-face walk on the toroidal $(4,4)$-grid is $12$-periodic and it does not admit perfect state transfer at any time.
\end{proposition}
Though both the periodicity and the absence of perfect state transfer in the proposition above can be verified with a straightforward computer proof, we have provided a written proof in Appendix \ref{app:proofs} as it is enlightening for how Theorem \ref{thm:PeakSTchar} and Theorem \ref{thm:periodicity_characterisation} (or rather Remark \ref{rem:periodicity_char_simplified}) can be applied. Table \ref{tab:4_4grid_periodic} also summarizes the proof for the periodicity at time 12.
\begin{table}
\[
\begin{array}{c|c|c|c}
     (k,\ell) &  \frac{1}{2}\left(\cos \frac{\pi k}{2} + 1\right)\left(\cos \frac{\pi \ell}{2} + 1\right) -1 & s_{k,\ell} & (p_{k,\ell},q_{k,\ell}) \\ \hline
     (0,0) & \phantom{-}1 & 0 & (0,1) \\ 
     (0,1) & \phantom{-}0 & \frac{1}{4} & (1,2) \\
     (0,2) & -1 & \frac{1}{2} & (1,1)\\
     (1,1) & -\frac{1}{2} & \frac{1}{3} & (2,3) \\
\end{array}
\]
\caption{For some pairs $k,\ell \in \{0,1,2,3\}$, we give the values of the right-hand side of \eqref{eq:vxf_grid_cos_equation} and the corresponding solutions $s_{k,\ell}$ in the case of $n = m = 4$. These four pairs provide all distinct eigenvalues of the projected transition matrix of the $(4,4)$-grid. Every solution $s_{k,\ell}$ is rational, so there is periodicity. The values for $p_{k,\ell}$ and $q_{k,\ell}$ given in the last column can be used to show that condition (ii) of Theorem \ref{thm:PeakSTchar} does not hold.\label{tab:4_4grid_periodic}}
\end{table}

Next, we will show that if $n$ or $m$ is divisible by $4$ (and $n,m \geq 3$), the $(4,4)$-grid is the only toroidal $(n,m)$-grid that admits periodicity.

\begin{lemma}
\label{lem:periodic_implies_4_4grid}
For $n,m \geq 3$, if the vertex-face walk on the toroidal $(n,m)$-grid is periodic and either $n$ or $m$ is divisible by $4$, then $n = m = 4$.
\end{lemma}
\begin{proof}
Assume without loss of generality that $n$ is divisible by $4$, and pick $k = n/4$ and $\ell = 1$. Then $\cos\frac{2\pi k}{n} = 0 $ and \eqref{eq:vxf_grid_cos_equation} can be reduced to the following equation:
\begin{equation}
\label{eq:cos_equation_div_4}
\cos\frac{2\pi}{m} - 2\cos(2\pi s) = 1 
\end{equation}
with $s := s_{k,\ell}$. Since we assume that the toroidal $(n,m)$-grid is periodic, this equation has a rational solution $s$ by Theorem \ref{thm:periodicity_characterisation}. We claim that in that case, both $\cos\frac{2\pi}{m}$ and $\cos(2\pi s)$ have to be rational.

Given that the claim is true, it follows from Lemma \ref{lem:conway_jones} that $m \in \{3,4,6\}$. If $m \in \{3,6\}$, then $\cos\frac{2\pi}{m} = \pm\frac{1}{2}$ so that
\[
\cos(2\pi s) \in \{ -\tfrac{1}{4}, -\tfrac{3}{4}\}  
\]
but then $s$ is irrational, which gives a contradiction. So it must be that $m = 4$, and we can deduce that also $n = 4$ by symmetry (as $m$ is now divisible by $4$), completing the proof.

What remains is to prove the claim: assume that both cosines in \eqref{eq:cos_equation_div_4} are irrational. We can apply the reflections $\theta \mapsto -\theta$ and $\theta \mapsto \pi - \theta$ to the angles in equation \eqref{eq:cos_equation_div_4}, and add some minus signs if necessary to transform the equation into the form
\begin{equation}
\label{eq:cosines_alpha_beta}
\pm \cos\alpha \pm 2\cos\beta = 1,
\end{equation}
where $\alpha,\beta \in (0,\pi/2)$ (and also $\alpha,\beta \neq \pi/3$ since both cosines remain irrational). Now the equation \eqref{eq:cosines_alpha_beta} must be proportional \eqref{eq:conway_jones_pi5} by Lemma \ref{lem:conway_jones}, but this is clearly impossible because $\cos(\alpha)$ and $\cos(\beta)$ appear with distinct magnitudes in \eqref{eq:cosines_alpha_beta} (namely $1$ and $2$), whereas they appear with the same magnitude (namely $1$) in \eqref{eq:conway_jones_pi5}. We conclude that the cosines in \eqref{eq:cos_equation_div_4} cannot both be irrational, and if one of them is rational, they must both be rational, proving the claim.
\end{proof}

With this, we have completed our characterization of periodicity in toroidal $(n,m)$-grids:

\begin{proof}[Proof of Theorem \ref{thm:toroidal_periodicity_characterisation}]
By \cite{GuoSch2024} and Proposition \ref{prop:4_4grid_periodic}, the toroidal $(n,m)$-grid is periodic if $\min\{n,m\} \leq 2$ or $n = m = 4$. Conversely, if the toroidal $(n,m)$-grid is periodic and $n,m \geq 3$, then $nm$ is divisible by $8$ by Lemma \ref{lem:periodic_implies_div8}, so either $n$ or $m$ must be divisible by $4$, and $n = m = 4$ by Lemma \ref{lem:periodic_implies_4_4grid}.
\end{proof}

\subsection{Infinite family of examples of peak state transfer}
\label{subsec:peakST_grids_inf_fam}

In this section, we show that there is peak state transfer in the vertex-face walk on the $(4,n)$-toroidal grids. We already gave an example of such a grid in Figure \ref{fig:peak_ST_illustration}. Its behavior is rigorously treated in the following theorem. Although the proof is quite technical and requires a lot of bookkeeping, it does not not use any particularly new ideas. We have therefore moved the proof into Appendix \ref{app:proofs}.

\begin{theorem}
    \label{thm:4n_grids_peakST}
    For $n \geq 1$, the $(4,n)$-toroidal grid admits peak state transfer. Specifically, for all $(i,j) \in \ints_{4} \times \ints_n$:
    \begin{enumerate}[(a)]
        \item if $n$ is an odd integer, there is peak state transfer from vertex $(i,j)$ to vertices $(i\pm 1,j)$ at time $n$;
        \item if $n = 2m$ for some $m$, there is peak state transfer from vertex $(i,j)$ to vertices $(i\pm 1,j + m)$ at time $m$;
        \item if $n = 4k$ for some $k$, there is peak state transfer from vertex $(i,j)$ to vertices $(i\pm 1,j \pm k)$ at time $k$; 
    \end{enumerate}
\end{theorem}
Note that the times specified in each of the cases are the first time of peak state transfer occurring; they take on the role of the $\tau$ from the statement of Theorem \ref{thm:PeakSTchar}.

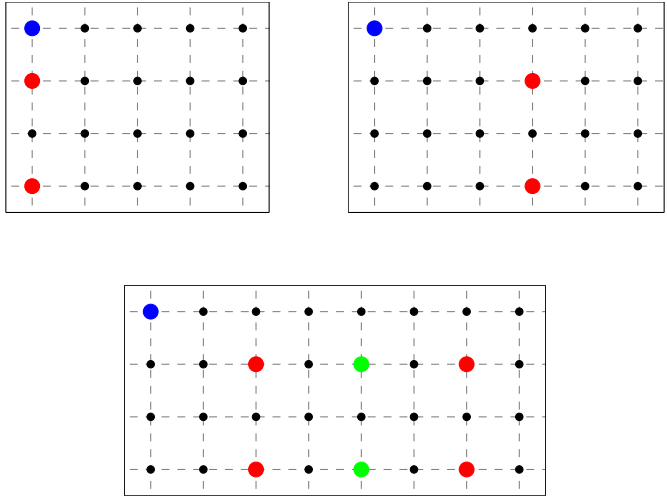
\begin{figure}
\centering
\begin{tikzpicture}[scale=.7]
\begin{pgfonlayer}{nodelayer}
   \node (0) at (-1,1) {};
   \node (1) at (0,1) {};
   \node (2) at (1,1) {};
   \node (3) at (2,1) {};
   \node (4) at (3,1) {};
   \node (5) at (4,1) {};
   \node (6) at (5,1) {};
   \node (7) at (-1,0) {};
   \node (8) at (0,0) {};
   \node (9) at (1,0) {};
   \node (10) at (2,0) {};
   \node (11) at (3,0) {};
   \node (12) at (4,0) {};
   \node (13) at (5,0) {};
   \node (14) at (-1,-1) {};
   \node (15) at (0,-1) {};
   \node (16) at (1,-1) {};
   \node (17) at (2,-1) {};
   \node (18) at (3,-1) {};
   \node (19) at (4,-1) {};
   \node (20) at (5,-1) {};
   \node (21) at (-1,-2) {};
   \node (22) at (0,-2) {};
   \node (23) at (1,-2) {};
   \node (24) at (2,-2) {};
   \node (25) at (3,-2) {};
   \node (26) at (4,-2) {};
   \node (27) at (5,-2) {};
   \node (28) at (-1,-3) {};
   \node (29) at (0,-3) {};
   \node (30) at (1,-3) {};
   \node (31) at (2,-3) {};
   \node (32) at (3,-3) {};
   \node (33) at (4,-3) {};
   \node (34) at (5,-3) {};
   \node (35) at (-1,-4) {};
   \node (36) at (0,-4) {};
   \node (37) at (1,-4) {};
   \node (38) at (2,-4) {};
   \node (39) at (3,-4) {};
   \node (40) at (4,-4) {};
   \node (41) at (5,-4) {};
   \node (a) at (-0.5,0.5) {};
   \node (b) at (4.5,0.5) {};
   \node (c) at (-0.5,-3.5) {};
   \node (d) at (4.5,-3.5) {};
\end{pgfonlayer}
\begin{pgfonlayer}{edgelayer}
\clip (a) rectangle (d);
   \draw [gray, dashed] (7.center) to (13.center);
   \draw [gray, dashed] (14.center) to (20.center);
   \draw [gray, dashed] (21.center) to (27.center);
   \draw [gray, dashed] (28.center) to (34.center);
   \draw [gray, dashed] (1.center) to (36.center);
   \draw [gray, dashed] (2.center) to (37.center);
   \draw [gray, dashed] (3.center) to (38.center);
   \draw [gray, dashed] (4.center) to (39.center);
   \draw [gray, dashed] (5.center) to (40.center);
    \draw [thin] (a.center) to (b.center);
    \draw [thin] (c.center) to (d.center);
    \draw [thin] (a.center) to (c.center);
    \draw [thin] (b.center) to (d.center);
\end{pgfonlayer}
\filldraw [black]
    (9) circle (2pt)
    (10) circle (2pt)
    (11) circle (2pt)
    (12) circle (2pt)
    (16) circle (2pt)
    (17) circle (2pt)
    (18) circle (2pt)
    (19) circle (2pt)
    (22) circle (2pt)
    (23) circle (2pt)
    (24) circle (2pt)
    (25) circle (2pt)
    (26) circle (2pt)
    (30) circle (2pt)
    (31) circle (2pt)
    (32) circle (2pt)
    (33) circle (2pt)
    ;
\filldraw [blue]
    (8) circle (4pt)
    ;
\filldraw [red]
    (15) circle (4pt)
    (29) circle (4pt)
    ;
\end{tikzpicture}
\begin{tikzpicture}[scale=.7]
\begin{pgfonlayer}{nodelayer}
   \node (0) at (-1,1) {};
   \node (1) at (0,1) {};
   \node (2) at (1,1) {};
   \node (3) at (2,1) {};
   \node (4) at (3,1) {};
   \node (5) at (4,1) {};
   \node (6) at (5,1) {};
   \node (7) at (6,1) {};
   \node (8) at (-1,0) {};
   \node (9) at (0,0) {};
   \node (10) at (1,0) {};
   \node (11) at (2,0) {};
   \node (12) at (3,0) {};
   \node (13) at (4,0) {};
   \node (14) at (5,0) {};
   \node (15) at (6,0) {};
   \node (16) at (-1,-1) {};
   \node (17) at (0,-1) {};
   \node (18) at (1,-1) {};
   \node (19) at (2,-1) {};
   \node (20) at (3,-1) {};
   \node (21) at (4,-1) {};
   \node (22) at (5,-1) {};
   \node (23) at (6,-1) {};
   \node (24) at (-1,-2) {};
   \node (25) at (0,-2) {};
   \node (26) at (1,-2) {};
   \node (27) at (2,-2) {};
   \node (28) at (3,-2) {};
   \node (29) at (4,-2) {};
   \node (30) at (5,-2) {};
   \node (31) at (6,-2) {};
   \node (32) at (-1,-3) {};
   \node (33) at (0,-3) {};
   \node (34) at (1,-3) {};
   \node (35) at (2,-3) {};
   \node (36) at (3,-3) {};
   \node (37) at (4,-3) {};
   \node (38) at (5,-3) {};
   \node (39) at (6,-3) {};
   \node (40) at (-1,-4) {};
   \node (41) at (0,-4) {};
   \node (42) at (1,-4) {};
   \node (43) at (2,-4) {};
   \node (44) at (3,-4) {};
   \node (45) at (4,-4) {};
   \node (46) at (5,-4) {};
   \node (47) at (6,-4) {};
   \node (a) at (-0.5,0.5) {};
   \node (b) at (5.5,0.5) {};
   \node (c) at (-0.5,-3.5) {};
   \node (d) at (5.5,-3.5) {};
\end{pgfonlayer}
\begin{pgfonlayer}{edgelayer}
\clip (a) rectangle (d);
   \draw [gray, dashed] (8.center) to (15.center);
   \draw [gray, dashed] (16.center) to (23.center);
   \draw [gray, dashed] (24.center) to (31.center);
   \draw [gray, dashed] (32.center) to (39.center);
   \draw [gray, dashed] (1.center) to (41.center);
   \draw [gray, dashed] (2.center) to (42.center);
   \draw [gray, dashed] (3.center) to (43.center);
   \draw [gray, dashed] (4.center) to (44.center);
   \draw [gray, dashed] (5.center) to (45.center);
   \draw [gray, dashed] (6.center) to (46.center);
    \draw [thin] (a.center) to (b.center);
    \draw [thin] (c.center) to (d.center);
    \draw [thin] (a.center) to (c.center);
    \draw [thin] (b.center) to (d.center);
\end{pgfonlayer}
\filldraw [black]
    (10) circle (2pt)
    (11) circle (2pt)
    (12) circle (2pt)
    (13) circle (2pt)
    (14) circle (2pt)
    (17) circle (2pt)
    (18) circle (2pt)
    (19) circle (2pt)
    (21) circle (2pt)
    (22) circle (2pt)
    (25) circle (2pt)
    (26) circle (2pt)
    (27) circle (2pt)
    (28) circle (2pt)
    (29) circle (2pt)
    (30) circle (2pt)
    (33) circle (2pt)
    (34) circle (2pt)
    (35) circle (2pt)
    (37) circle (2pt)
    (38) circle (2pt)
    ;
\filldraw[blue]
    (9) circle (4pt)
    ;
\filldraw[red]
    (20) circle (4pt)
    (36) circle (4pt)
    ;
\end{tikzpicture}

\begin{tikzpicture}[scale=.7]
\begin{pgfonlayer}{nodelayer}
   \node (0) at (-1,1) {};
   \node (1) at (0,1) {};
   \node (2) at (1,1) {};
   \node (3) at (2,1) {};
   \node (4) at (3,1) {};
   \node (5) at (4,1) {};
   \node (6) at (5,1) {};
   \node (7) at (6,1) {};
   \node (8) at (7,1) {};
   \node (9) at (8,1) {};
   \node (10) at (-1,0) {};
   \node (11) at (0,0) {};
   \node (12) at (1,0) {};
   \node (13) at (2,0) {};
   \node (14) at (3,0) {};
   \node (15) at (4,0) {};
   \node (16) at (5,0) {};
   \node (17) at (6,0) {};
   \node (18) at (7,0) {};
   \node (19) at (8,0) {};
   \node (20) at (-1,-1) {};
   \node (21) at (0,-1) {};
   \node (22) at (1,-1) {};
   \node (23) at (2,-1) {};
   \node (24) at (3,-1) {};
   \node (25) at (4,-1) {};
   \node (26) at (5,-1) {};
   \node (27) at (6,-1) {};
   \node (28) at (7,-1) {};
   \node (29) at (8,-1) {};
   \node (30) at (-1,-2) {};
   \node (31) at (0,-2) {};
   \node (32) at (1,-2) {};
   \node (33) at (2,-2) {};
   \node (34) at (3,-2) {};
   \node (35) at (4,-2) {};
   \node (36) at (5,-2) {};
   \node (37) at (6,-2) {};
   \node (38) at (7,-2) {};
   \node (39) at (8,-2) {};
   \node (40) at (-1,-3) {};
   \node (41) at (0,-3) {};
   \node (42) at (1,-3) {};
   \node (43) at (2,-3) {};
   \node (44) at (3,-3) {};
   \node (45) at (4,-3) {};
   \node (46) at (5,-3) {};
   \node (47) at (6,-3) {};
   \node (48) at (7,-3) {};
   \node (49) at (8,-3) {};
   \node (50) at (-1,-4) {};
   \node (51) at (0,-4) {};
   \node (52) at (1,-4) {};
   \node (53) at (2,-4) {};
   \node (54) at (3,-4) {};
   \node (55) at (4,-4) {};
   \node (56) at (5,-4) {};
   \node (57) at (6,-4) {};
   \node (58) at (7,-4) {};
   \node (59) at (8,-4) {};
   \node (a) at (-0.5,0.5) {};
   \node (b) at (7.5,0.5) {};
   \node (c) at (-0.5,-3.5) {};
   \node (d) at (7.5,-3.5) {};
\end{pgfonlayer}
\begin{pgfonlayer}{edgelayer}
\clip (a) rectangle (d);
   \draw [gray, dashed] (10.center) to (19.center);
   \draw [gray, dashed] (20.center) to (29.center);
   \draw [gray, dashed] (30.center) to (39.center);
   \draw [gray, dashed] (40.center) to (49.center);
   \draw [gray, dashed] (1.center) to (51.center);
   \draw [gray, dashed] (2.center) to (52.center);
   \draw [gray, dashed] (3.center) to (53.center);
   \draw [gray, dashed] (4.center) to (54.center);
   \draw [gray, dashed] (5.center) to (55.center);
   \draw [gray, dashed] (6.center) to (56.center);
   \draw [gray, dashed] (7.center) to (57.center);
   \draw [gray, dashed] (8.center) to (58.center);
    \draw [thin] (a.center) to (b.center);
    \draw [thin] (c.center) to (d.center);
    \draw [thin] (a.center) to (c.center);
    \draw [thin] (b.center) to (d.center);
\end{pgfonlayer}
\filldraw [black]
    (12) circle (2pt)
    (13) circle (2pt)
    (14) circle (2pt)
    (15) circle (2pt)
    (16) circle (2pt)
    (17) circle (2pt)
    (18) circle (2pt)
    (21) circle (2pt)
    (22) circle (2pt)
    (24) circle (2pt)
    (26) circle (2pt)
    (28) circle (2pt)
    (31) circle (2pt)
    (32) circle (2pt)
    (33) circle (2pt)
    (34) circle (2pt)
    (35) circle (2pt)
    (36) circle (2pt)
    (37) circle (2pt)
    (38) circle (2pt)
    (41) circle (2pt)
    (42) circle (2pt)
    (44) circle (2pt)
    (46) circle (2pt)
    (48) circle (2pt)
    ;
\filldraw[blue]
    (11) circle (4pt)
    ;
\filldraw[red]
    (23) circle (4pt)
    (27) circle (4pt)
    (43) circle (4pt)
    (47) circle (4pt)
    ;
\filldraw[green]
    (25) circle (4pt)
    (45) circle (4pt)
    ;
\end{tikzpicture}
\caption{Depicted are all the cases of peak state transfer as specified in Theorem \ref{thm:4n_grids_peakST}. The $(4,5)$-grid (top-left) admits peak state transfer from the blue vertex to the red vertices at time $5$ (case (a)). The $(4,6)$-grid (top-right) admits peak state transfer from the blue vertex to the red vertices at time $3$ (case (b)). The $(4,8)$-grid admits peak state transfer from the blue vertex to the red vertices at time $2$ and to the green vertices at time $4$ (cases (b) and (c)). \label{fig:grid_vxfwalk_peakST}}
\end{figure}

In Figure \ref{fig:grid_vxfwalk_peakST}, we have depicted all instances of peak state transfer in the $(4,n)$-grids. Note that if $n$ is divisible by $4$, both cases (b) and (c) of Theorem \ref{thm:4n_grids_peakST} apply and there is peak state transfer to two distinct pairs of vertices at distinct times. The example of the $(4,6)$-grid in Figure \ref{fig:peak_ST_illustration} falls in case (b), so that there is peak state transfer at time $3$.

\section{Conclusion} \label{sec:openproblems}

In this paper, we have defined peak state transfer and given a spectral characterization of this phenomenon. Though the characterization is time-independent, it can be highly nontrivial to verify whether the conditions hold for some infinite family of graphs. This gives rise to many open problems; some are about discrete-time quantum walks and  some are purely combinatorial related with the spectral property needed for Theorem \ref{thm:PeakSTchar}. In the remainder of this section, we will discuss further directions on state transfer in quantum walks and on combinatorial open problems that arose from this study.

\subsection{Future directions}

The main tool of this paper is to make use of the projected transition matrix and a future direction is to develop  tools, in a more general setting, including a broader definition of strong cospectrality for Hermitian matrices.

\subsection*{More general star states }
A point of consideration, also mentioned in \cite{GuoSch2024}, 
is that, in the case of the arc-reversal walk or vertex-face walk, it may be desirable to use a more general definition. In  a more general setting, perfect state transfer could be said to  occur from vertex $u$ to vertex $v$ at time $\tau$ if the vector $U^\tau N\Ze_u$ is supported only on the arcs that are incident to the vertex $v$, but not necessarily equal to the uniform superposition $N\Ze_v$. As shown in \cite{GodZha2023} and \cite{GuoSch2024}, these definitions are equivalent if $u$ and $v$ have the same degree. If $u$ and $v$ have different degrees, however, this alternative form of perfect state transfer can occur, as illustrated in Figure \ref{fig:altPST_arc_reversal_path}. While the projected transition matrix cannot be easily leveraged in this setting, an interesting direction of research would be to develop tools for the more general setting.

\begin{figure}[htb]
    \centering
    \begin{tikzpicture}[scale=.8]
    	\begin{pgfonlayer}{nodelayer}
    		\node[label=left:$u$] (0) at (0, 2) {};
    		\node[label=left:$v$] (1) at (0, 0) {};
    		\node (2) at (0, -2) {};
            \node (3) at (0,2.5) {\color{vcolour}{$N\Ze_u$}};
    	\end{pgfonlayer}
    	\begin{pgfonlayer}{edgelayer}
    		\draw [thick, style=directed, bend right=15, color=vcolour] (0.center) to node [pos = 0.6, left] {$1$} (1.center);
    		\draw [style=directed, bend right=15] (1.center) to (0.center);
    		\draw [style=directed, bend right=15] (1.center) to node [pos = 0.6, right=0pt] {$\phantom{\frac{1}{\sqrt{2}}}$} (2.center);
    		\draw [style=directed, bend right=15] (2.center) to node [pos = 0.6, right=0pt] {$\phantom{\frac{1}{\sqrt{2}}}$} (1.center);
    	\end{pgfonlayer}
        \filldraw [black]
            (0) circle (2.5pt)
            (1) circle (2.5pt)
            (2) circle (2.5pt);
    \end{tikzpicture} \hspace{20pt}
    \begin{tikzpicture}[scale=.8]
    	\begin{pgfonlayer}{nodelayer}
    		\node[label=left:$u$] (0) at (0, 2) {};
    		\node[label=left:$v$] (1) at (0, 0) {};
    		\node (2) at (0, -2) {};
            \node (3) at (0,2.5) {\color{vcolour}{$UN\Ze_u$}};
    	\end{pgfonlayer}
    	\begin{pgfonlayer}{edgelayer}
    		\draw [style=directed, bend right=15] (0.center) to (1.center);
    		\draw [thick, style=directed, bend right=15, color=vcolour] (1.center) to node [pos = 0.6, right] {$1$} (0.center);
    		\draw [style=directed, bend right=15] (1.center) to (2.center);
    		\draw [style=directed, bend right=15] (2.center) to (1.center);
    	\end{pgfonlayer}
        \filldraw [black]
            (0) circle (2.5pt)
            (1) circle (2.5pt)
            (2) circle (2.5pt);
    \end{tikzpicture} \hspace{20pt}
    \begin{tikzpicture}[scale=.8][scale=.8]
    	\begin{pgfonlayer}{nodelayer}
    		\node[label=left:$u$] (0) at (0, 2) {};
    		\node[label=left:$v$] (1) at (0, 0) {};
    		\node (2) at (0, -2) {};
            \node (3) at (0,2.5) {\color{kcolour}{$N\Ze_v$}};
    	\end{pgfonlayer}
    	\begin{pgfonlayer}{edgelayer}
    		\draw [style=directed, bend right=15] (0.center) to (1.center);
    		\draw [thick, style=directed, bend right=15, color=kcolour] (1.center) to node [pos = 0.6, right] {$\frac{1}{\sqrt{2}}$} (0.center);
    		\draw [thick, style=directed, bend right=15, color=kcolour] (1.center) to node [pos = 0.6, left] {$\frac{1}{\sqrt{2}}$}(2.center);
    		\draw [style=directed, bend right=15] (2.center) to (1.center);
    	\end{pgfonlayer}
        \filldraw [black]
            (0) circle (2.5pt)
            (1) circle (2.5pt)
            (2) circle (2.5pt);
    \end{tikzpicture}
    \caption{On the left and in the middle the arc-reversal walk on the path $P_3$ at times $0$ and $1$, starting in the leaf $u$, and on the right the state $N\Ze_v$. Even though the state at time $1$ is supported only on the arcs incident to the vertex $v$, there is no perfect state transfer at time $1$ because $UN\Ze_v$ is not parallel to $N\Ze_v$. \label{fig:altPST_arc_reversal_path}}
\end{figure}
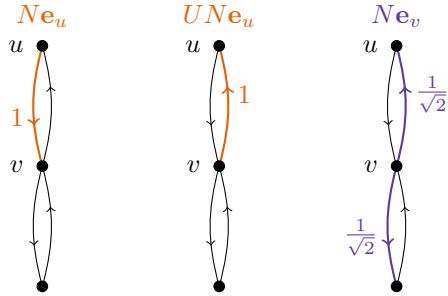

\subsection*{Complex peak state transfer and a more general notion of strong cospectrality}

In this paper, we have only discussed peak state transfer for real-valued two-reflection walks. There is a more general setting in which $U$ and $B$ have complex-valued entries. It is possible to formulate complex analogues to the results in Section \ref{sec:peakst}. If $U$ has complex-valued entries, then its projected transition matrix $B$ is Hermitian, as are all other $B_t = N^* U^t N$. As in Lemma \ref{lem:technical_lemma}, if $B$ has spectral decomposition $B = \sum_\theta \theta E_{\theta}$, then
\[
|B_t(u,v)| \leq \sum_\theta |E_\theta(u,v)| \leq 1.
\]
In the Hermitian setting, the conditions for these inequalities to be equalities are similar, but slightly different. The first inequality is an equality if and only if there exists a $\gamma \in \cx$ with $|\gamma| = 1$ such that for all $\theta$ in the mutual eigenvalue support of $u$ and $v$:
\[
T_t(\theta) E_\theta(u,v) = \gamma |E_\theta(u,v)|.
\]
In this case, $E_\theta(u,v)$ is a real multiple of $\gamma$ for every such $\theta$. Therefore, even though the entries of the spectral idempotents are not necessarily real-valued, we can still split the mutual eigenvalue support into a `positive' and a `negative' part, once we have verified that such a $\gamma$ exists:
\[
\Lambda^{\pm \gamma}(u,v) \coloneq \{\theta \in \ev(B) \mid E_\theta(u,v)/|E_\theta(u,v)| = \pm \gamma\}.
\]

Further, the second inequality is becomes an equality if and only if for all $\theta$ there exists an $\alpha_\theta \in \cx$ with $|\alpha_\theta| = 1$ such that
\[
E_\theta\Ze_v = \alpha_\theta E_\theta\Ze_u.
\]
This gives motivation to extend the notion of `strong cospectrality' to all Hermitian matrices, using this characterization as a definition. If the first and second equality both hold then $\alpha_\theta = \gamma$ for all $\theta \in \Lambda^\gamma(u,v)$ and $\alpha_\theta = -\gamma$ for all $\theta \in \Lambda^{-\gamma}(u,v)$.


Note that since periodicity at $u$ at time $t$ occurs if and only if $B_t(u,u) = \gamma$ for some $\gamma$ of unit length, we must have $\gamma = \pm 1$, as $B_t$ is Hermitian.

\subsection{Combinatorial open problems}

Our problem of studying state transfer arises in quantum computing, but it has led to several purely  combinatorial problems, which provides a rich source of new challenges for combinatorialists. We will present some of them below. 

\subsection*{Maximizing entries in \texorpdfstring{$ M_G = \sum_{\theta} |E_{\theta}| $}{M}}

In an upcoming work \cite{CouGuoSch2025}, the authors, along with G.~Coutinho, define peak state transfer for continuous-time quantum walks and give a characterization analogous to Theorem \ref{thm:PeakSTchar}. In both the continuous-time setting and the discrete-time arc-reversal walk on regular graphs, we are interested in the entries of matrix 
 \[ M_G = \sum_{\theta} |E_{\theta}| \]
 where $A$ is either the adjacency or the normalized adjacency matrix of $G$, with spectral decomposition $A =  \sum_{\theta} \theta E_{\theta}$. The $u,v$ entry of this matrix governs how much state transfer can occur, as seen in the definition of peak state transfer. If there is perfect or pretty good state transfer between $u$ and $v$, then $M_G(u,v) = 1$. It is interesting to maximize this quantity, even when it is not equal to $1$, amongst graphs on $n$ vertices. More concretely, we can ask
 
 \begin{openprob}
 What is the value of:
 \[
 \max \{ M_G(u,v) \mid M_G(u,v) < 1, G \in \cG_n , u,v \in V(G)\}
 \]
 where $\cG_n$ can be taken to be all graphs on $n$ vertices? We can take the maximum with respect to the adjacency matrix, when bounding peak state transfer for continuous-time walks, or the normalized adjacency matrix, when bounding peak state transfer for the arc reversal walk. 
 \end{openprob}
 
 For example, on $3$ vertices, the entries of $M$ which are not equal to $1$ is maximized by that of $P_3$, for $u,v$ adjacent, which is $\frac{1}{\sqrt{2}} \approx 0.707107$. 

\subsection*{Graphs with special eigenvalues }

Let $G$ be a graph and let $A$ be the normalized adjacency matrix of $G$. The eigenvalue condition in (i) of Theorem \ref{thm:PeakSTchar} will hold if all eigenvalues of $A$ are contained in $\{\pm 1, \pm \frac{1}{2},0\}$. This leads us to the following problem:

\begin{openprob}
Determine all graphs which have normalized adjacency matrices with all eigenvalues contained in the set $\{\pm 1, \pm \frac{1}{2},0\}$. 
\end{openprob}

In particular, if $G$ is $k$-regular, we can rephrase this as a question about the adjacency matrix.

\begin{openprob}
Determine all $k$-regular graphs which have adjacency matrices with all eigenvalues contained in the set $\{\pm k, \pm \frac{k}{2},0\}$. Determine all bipartite, biregular graphs with degrees $r,k$ such that the adjacency matrix has all eigenvalues contained in $\{\pm \sqrt{rk}, \pm \frac{\sqrt{rk}}{2},0\}$. 
\end{openprob}

In the course of determining peak state transfer in the arc-reversal walk, we have already determined this for  strongly regular graphs (Section \ref{subsec:SRGs}) and  bipartite, biregular graphs constructed as point-block incidence graphs of a block design (Section \ref{subsec:blockdesigns}). We also gave several examples of distance-regular graphs in Section \ref{subsec:other_examples}, but the classification of intersection arrays of distance-regular graph admitting peak state transfer is left as an open problem. 

We note that peak state transfer could still occur for some pair of vertices, even if not all eigenvalues of the normalized adjacency matrix are contained in $\{\pm 1, \pm \frac{1}{2},0\}$; this is the topic of the next subsection. 

\subsection*{Mutual eigenvalue support}

In Section \ref{sec:peakst}, we defined the mutual eigenvalue support $\Lambda(u,v)$ of two vertices $u,v$. Given a projected transition matrix $B$, with spectral decomposition $B= \sum_{\theta} \theta E_{\theta}$, and $u,v \in X$, where $X$ indexes the rows and columns of $B$, we note that the eigenvalue condition in (i) of Theorem \ref{thm:PeakSTchar} applies only for $\theta \in \Lambda(u,v)$. This leads to the following question:

\begin{openprob}
   What are combinatorial (or algebraic) conditions on graphs $G$ with pairs of vertices $u,v$, such that some spectral idempotent of its normalized adjacency matrix has $E_{\theta}(u,v) =0$?
\end{openprob}

We note that this occurs if $u,v$ are in different connected components. 
Since $E_{\theta}$ is positive semidefinite, if the $(u,u)$ diagonal entry is equal to $0$, then the $u$-th row and $u$-th column are both $0$. In an imprimitive distance-regular graph, some spectral idempotent(s) are block-diagonal; having a block-diagonal spectral idempotent is equivalent to being imprimitive, see \cite{MarMuzWil2007}. For graphs in association schemes, this is equivalent to asking when the $Q$-matrix has a $0$ entry and determining when this occurs will aid in classifying peak state transfer in distance-regular graphs.

\bibliographystyle{plain}

\addcontentsline{toc}{section}{References}
\begin{thebibliography}{10}

\bibitem{BanCouGod2017}
Leonardo Banchi, Gabriel Coutinho, Chris Godsil, and Simone Severini.
\newblock Pretty good state transfer in qubit chains---the {H}eisenberg {H}amiltonian.
\newblock {\em J. Math. Phys.}, 58(3):032202, 9, 2017.

\bibitem{BroCohNeu1989}
A.~E. Brouwer, A.~M. Cohen, and A.~Neumaier.
\newblock {\em Distance-regular graphs}, volume~18 of {\em Ergebnisse der Mathematik und ihrer Grenzgebiete (3) [Results in Mathematics and Related Areas (3)]}.
\newblock Springer-Verlag, Berlin, 1989.

\bibitem{BroHae2012}
Andries~E. Brouwer and Willem~H. Haemers.
\newblock {\em Spectra of graphs}.
\newblock Universitext. Springer, New York, 2012.

\bibitem{ChaSin2023}
Ada Chan and Peter Sin.
\newblock Pretty good state transfer among large sets of vertices, 2023.

\bibitem{ChaZha2023}
Ada Chan and Hanmeng Zhan.
\newblock Pretty good state transfer in discrete-time quantum walks.
\newblock {\em J. Phys. A}, 56(16):Paper No. 165305, 25, 2023.

\bibitem{CheGod2011}
Wang-Chi Cheung and Chris Godsil.
\newblock Perfect state transfer in cubelike graphs.
\newblock {\em Linear Algebra and its Applications}, 435(10):2468--2474, 2011.
\newblock Special Issue in Honor of Dragos Cvetkovic.

\bibitem{Chu1997}
Fan R.~K. Chung.
\newblock {\em Spectral graph theory}, volume~92 of {\em CBMS Regional Conference Series in Mathematics}.
\newblock Conference Board of the Mathematical Sciences, Washington, DC; by the American Mathematical Society, Providence, RI, 1997.

\bibitem{ColDin2007}
Charles~J. Colbourn and Jeffrey~H. Dinitz, editors.
\newblock {\em Handbook of combinatorial designs}.
\newblock Discrete Mathematics and its Applications (Boca Raton). Chapman \& Hall/CRC, Boca Raton, FL, second edition, 2007.

\bibitem{ConJon1976}
J.~H. Conway and A.~J. Jones.
\newblock Trigonometric {D}iophantine equations ({O}n vanishing sums of roots of unity).
\newblock {\em Acta Arith.}, 30(3):229--240, 1976.

\bibitem{CooDhoGoe2023}
Kris Coolsaet, Sven D’hondt, and Jan Goedgebeur.
\newblock House of graphs 2.0: A database of interesting graphs and more.
\newblock {\em Discrete Applied Mathematics}, 325:97--107, 2023.

\bibitem{CouGodGuo2015}
G.~Coutinho, C.~Godsil, K.~Guo, and F.~Vanhove.
\newblock Perfect state transfer on distance-regular graphs and association schemes.
\newblock {\em Linear Algebra and its Applications}, 478:108--130, 2015.

\bibitem{CouGod2023book}
Gabriel Coutinho and Chris Godsil.
\newblock Graph spectra and continuous quantum walks, 2023.
\newblock Manuscript available at \url{https://www.math.uwaterloo.ca/~cgodsil/quagmire/QuantumWalks/pdfs/GrfSpc3.pdf} (accessed 8 Nov 2024).

\bibitem{CouGuoSch2025}
Gabriel Coutinho, Krystal Guo, and Vincent Schmeits.
\newblock Peak state transfer in continuous-time quantum walks.
\newblock \textit{(in preparation)}.

\bibitem{CouGuoBom2017}
Gabriel Coutinho, Krystal Guo, and Christopher~M. van Bommel.
\newblock Pretty good state transfer between internal nodes of paths.
\newblock {\em Quantum Inf. Comput.}, 17(9-10):825--830, 2017.

\bibitem{vDamKooTan2016}
Edwin~R. Van~Dam, Jack~H. Koolen, and Hajime Tanaka.
\newblock Distance-regular graphs.
\newblock {\em The Electronic Journal of Combinatorics}, {\#DS22}, 2016.

\bibitem{Dav1979}
Philip~J. Davis.
\newblock {\em Circulant matrices}.
\newblock A Wiley-Interscience Publication. John Wiley \& Sons, New York-Chichester-Brisbane, 1979.
\newblock Pure and Applied Mathematics.

\bibitem{DumFoo2004}
David~S. Dummit and Richard~M. Foote.
\newblock {\em Abstract algebra}.
\newblock John Wiley \& Sons, Inc., Hoboken, NJ, third edition, 2004.

\bibitem{GodGuoKem2020}
C.~D. Godsil, K.~Guo, M.~Kempton, G.~Lippner, and F.~M{\"{u}}nch.
\newblock State transfer in strongly regular graphs with an edge perturbation.
\newblock {\em J. Comb. Theory, Ser. {A}}, 172:105181, 2020.

\bibitem{God2011}
Chris Godsil.
\newblock Periodic graphs.
\newblock {\em Electron. J. Combin.}, 18(1):Paper 23, 15, 2011.

\bibitem{GodKirSev2012}
Chris Godsil, Stephen Kirkland, Simone Severini, and Jamie Smith.
\newblock Number-theoretic nature of communication in quantum spin systems.
\newblock {\em Phys. Rev. Lett.}, 109:050502, Aug 2012.

\bibitem{GodRoy2001}
Chris Godsil and Gordon Royle.
\newblock {\em Algebraic graph theory}, volume 207 of {\em Graduate Texts in Mathematics}.
\newblock Springer-Verlag, New York, 2001.

\bibitem{GodSmi2024}
Chris Godsil and Jamie Smith.
\newblock Strongly cospectral vertices.
\newblock {\em Australas. J. Combin.}, 88:1--21, 2024.

\bibitem{GodZha2019}
Chris Godsil and Hanmeng Zhan.
\newblock Discrete-time quantum walks and graph structures.
\newblock {\em Journal of combinatorial theory. Series A}, 167:181--212, 2019.

\bibitem{GodZha2023}
Chris Godsil and Hanmeng Zhan.
\newblock {\em Discrete quantum walks on graphs and digraphs}, volume 484 of {\em London Mathematical Society Lecture Note Series}.
\newblock Cambridge University Press, Cambridge, 2023.

\bibitem{GuoSch2024}
Krystal Guo and Vincent Schmeits.
\newblock Perfect state transfer in quantum walks on orientable maps.
\newblock {\em Algebraic Combinatorics}, 7(3):713--747, 2024.

\bibitem{SteSko2016}
M.~\ifmmode \check{S}\else \v{S}\fi{}tefa\ifmmode~\check{n}\else \v{n}\fi{}\'ak and S.~Skoup\'y.
\newblock Perfect state transfer by means of discrete-time quantum walk search algorithms on highly symmetric graphs.
\newblock {\em Phys. Rev. A}, 94:022301, Aug 2016.

\bibitem{JefZur2023}
Stacey Jeffery and Sebastian Zur.
\newblock Multidimensional quantum walks, with application to $k$-distinctness.
\newblock \arxiv{2208.13492}, 2023.

\bibitem{Ved2021}
Vedran Kr\v{c}adinac.
\newblock A new partial geometry {$pg(5,5,2)$}.
\newblock {\em J. Combin. Theory Ser. A}, 183:Paper No. 105493, 4, 2021.

\bibitem{KubSeg2022}
Sho Kubota and Etsuo Segawa.
\newblock Perfect state transfer in {G}rover walks between states associated to vertices of a graph.
\newblock {\em Linear Algebra Appl.}, 646:238--251, 2022.

\bibitem{KubYos2024}
Sho Kubota and Kiyoto Yoshino.
\newblock Symmetry of graphs and perfect state transfer in grover walks, 2024.

\bibitem{KurWoj2011}
Pawe\l{} Kurzy\ifmmode~\acute{n}\else \'{n}\fi{}ski and Antoni W\'ojcik.
\newblock Discrete-time quantum walk approach to state transfer.
\newblock {\em Phys. Rev. A}, 83:062315, Jun 2011.

\bibitem{LiuZho2022}
Xiaogang Liu and Sanming Zhou.
\newblock Eigenvalues of Cayley graphs.
\newblock {\em The Electronic Journal of Combinatorics}, 29(2):{\#P}2.9, 2022.

\bibitem{MarMuzWil2007}
William~J. Martin, Mikhail Muzychuk, and Jason Williford.
\newblock Imprimitive cometric association schemes: Constructions and analysis.
\newblock {\em Journal of Algebraic Combinatorics}, 25(4):399--415, 2007.

\bibitem{MarPisTom2005}
Dragan Maru\v{s}i\v{c}, Toma\v{z}{} Pisanski, and Steve Wilson.
\newblock The genus of the {GRAY} graph is 7.
\newblock {\em European J. Combin.}, 26(3-4):377--385, 2005.

\bibitem{Riv1974}
Theodore~Joseph Rivlin.
\newblock {\em The Chebyshev polynomials}.
\newblock Pure and applied mathematics ; 40. Wiley-Interscience, New York ;, 1974.

\bibitem{Spi2019}
Daniel~A. Spielman.
\newblock Spectral and algebraic graph theory.
\newblock http://cs-www.cs.yale.edu/homes/spielman/sagt/sagt.pdf, 2019.

\bibitem{Sze2004}
M~Szegedy.
\newblock Quantum speed-up of {M}arkov chain based algorithms.
\newblock In {\em 45th Annual IEEE Symposium on Foundations of Computer Science}, pages 32--41, Los Alamitos CA, 2004. IEEE.

\bibitem{sagemath}
{The Sage Developers}.
\newblock {\em {S}ageMath, the {S}age {M}athematics {S}oftware {S}ystem ({V}ersion 9.0)}, 2020.
\newblock {\tt https://www.sagemath.org}.

\bibitem{Bom2019}
Christopher~M. van Bommel.
\newblock A complete characterization of pretty good state transfer on paths.
\newblock {\em Quantum Inf. Comput.}, 19(7-8):601--608, 2019.

\bibitem{LinSch1981}
J.~H. van Lint and A.~Schrijver.
\newblock Construction of strongly regular graphs, two-weight codes and partial geometries by finite fields.
\newblock {\em Combinatorica}, 1(1):63--73, 1981.

\bibitem{VinZhe2012}
Luc Vinet and Alexei Zhedanov.
\newblock Almost perfect state transfer in quantum spin chains.
\newblock {\em Phys. Rev. A}, 86:052319, Nov 2012.

\bibitem{Zha2019}
Hanmeng Zhan.
\newblock An infinite family of circulant graphs with perfect state transfer in discrete quantum walks.
\newblock {\em Quantum Inf. Process.}, 18(12):Paper No. 369, 26, 2019.

\bibitem{Zha2021}
Hanmeng Zhan.
\newblock Quantum walks on embeddings.
\newblock {\em J. Algebraic Combin.}, 53(4):1187--1213, 2021.

\bibitem{YalGed2015}
İskender Yalçınkaya and Zafer Gedik.
\newblock Qubit state transfer via discrete-time quantum walks.
\newblock {\em Journal of Physics A: Mathematical and Theoretical}, 48(22):225302, {M}ay 2015.

\end{thebibliography}

\appendix 
\section{Omitted proofs}
\label{app:proofs}
In this appendix, we will provide omitted proofs of some results in the text, in order of appearance.

The proof of Lemma \ref{lem:lex_product_periodicity} is provided below. The lemma states that if the arc-reversal walk on some connected graph $G$ is $\tau$-periodic at some vertex $u$, then the blow-up graph $G[\overline{K_m}]$ is $\tau'$-periodic at every vertex $(u,a)$ for $a =1,\ldots,m$, where $\tau' = \lcm\{\tau,4\}$. Moreover, is $(u,b)$ is a vertex that is distinct from $(u,a)$, then there is peak state transfer from $(u,a)$ to $(u,b)$ at time $\tau'/2$ if and only if one of the following is true:
\begin{enumerate}[(a)]
\item $\tau$ is not divisible by $4$;
\item $\tau \equiv 4 \mod 8$ and the integer $\tau \cdot \frac{r(\theta)}{s(\theta)}$ is even for all eigenvalues $\theta \in \Lambda(u,u) \setminus \{0\}$ of the projected transition matrix of $B$, where $r(\theta)$ and $s(\theta)$ are as in Remark \ref{rem:periodicity_char_simplified}.
\end{enumerate}
Conversely, if there is periodicity at $(u,a)$ or peak state transfer from $(u,a)$ to $(u,b)$ in $G[\overline{K_m}]$, there must be periodicity at $u$ in $G$.

\begin{proof}[Proof of Lemma \ref{lem:lex_product_periodicity}]
We will need to consider the eigenvalue supports of several vertices and we will use the following notation for the mutual eigenvalue supports of the projected transition matrices $B$ and $B'$ of $G$ and $G[\overline{K_m}]$ respectively:
\[
\begin{split}
\Lambda &\coloneq \Lambda(u,u) \subset \ev(B) \\
\Omega &\coloneq \Lambda((u,a),(u,a)) \subset \ev(B')\\
\Theta^{\pm 1} &\coloneq \Lambda^{\pm 1}((u,a),(u,b)) \subset \ev(B')
\end{split}
\]
In order to apply Theorems \ref{thm:PeakSTchar} and \ref{thm:periodicity_characterisation} to $B'$, we need to determine $\Omega$, $\Theta^{1}$ and $\Theta^{-1}$. By equation \eqref{eq:nonzero_idempotents_mG} we have that
\[
F_\theta((u,a),(u,b)) = \frac{1}{m}E_\theta (u,u) \begin{cases}
> 0 &\text{if $\theta \in \Lambda$;} \\
= 0 &\text{otherwise,}
\end{cases}
\]
and by equation \eqref{eq:zero_idempotent_mG}
\[
F_0((u,a),(u,b)) = \begin{cases}
- \frac{1}{m}(1 - E_0(u,u)) \quad &\text{if $0 \in \sigma(B)$}; \\
- \frac{1}{m} &\text{otherwise.}
\end{cases}
\]
In both cases, $F_0((u,a),(u,b)) < 0$, since $E_0(u,u) = 1$ would imply that $u$ is an isolated vertex of $G$. Hence
\[
\Theta^1 = \Lambda \setminus \{0\} \quad \text{and} \quad \Theta^{-1} = \{0\}.
\]
Similarly, by \eqref{eq:nonzero_idempotents_mG} and \eqref{eq:zero_idempotent_mG},
\[
F_\theta((u,a),(u,a)) = \frac{1}{m}E_\theta(u,u) > 0
\]
for all $\theta \in \Lambda$ and $F_0((u,a),(u,a)) > 0$, so that
\[
\Omega = \Lambda \cup \{0\}
\]
By Remark \ref{rem:periodicity_char_simplified}, every $\theta \in \Lambda$ can be uniquely written as
\[
\theta = \cos \frac{2r(\theta)\pi}{s(\theta)}
\]
where $r(\theta) \geq 0$ and $s(\theta) > 0$ are integers such that $r(\theta) \leq \frac{1}{2}s(\theta)$ and $\gcd(r(\theta),s(\theta)) = 1$; moreover,
\[
\tau = \lcm\{s(\theta) : \theta \in \Lambda\}.
\]
To see that statement (i) holds, note that since $0 = \cos \frac{2\cdot 1 \cdot \pi}{4} \in \Omega$, we have by Remark \ref{rem:periodicity_char_simplified} that there is periodicity in $G[\overline{K_m}]$ with period
\[
\tau' = \lcm\{\tau,4\}.
\]
To prove statement (ii), in order to check for peak state transfer, we rewrite every $\theta \in \Lambda$ in the form
\[
\theta = \cos \frac{p(\theta)\pi}{q(\theta)}
\]
as in Theorem \ref{thm:PeakSTchar}(i), where in particular $\gcd(p(\theta),q(\theta)) = 1$. Then
\[
(p(\theta),q(\theta)) = \begin{cases}
(r(\theta),s(\theta)/2) &\text{if $s(\theta)$ is even}; \\
(2r(\theta),s(\theta)) &\text{if $s(\theta)$ is odd.}
\end{cases}
\]
Regardless of $0$ being an eigenvalue of $B$ in $\Lambda$, it is an eigenvalue of $B'$ in $\Theta^{-1}$ and can be written as $\cos\frac{1 \cdot \pi}{2}$. Hence the first time of peak state transfer from $(u,a)$ to $(u,b)$, if it occurs, must be
\[
t := \lcm(\{q(\theta) : \theta \in \Lambda\} \cup \{2\}) = \begin{cases}
2\tau &\text{if $\tau \equiv 1 \mod 2$}; \\
\tau &\text{if $\tau \equiv 2 \mod 4$}; \\
\tau/2 &\text{if $\tau \equiv 0 \mod 4$},
\end{cases}
\]
or alternatively, $t = \tau'/2$. Moreover, for peak state transfer to occur, it is necessary that $t\cdot \frac{1}{2}$ is odd in order for the eigenvalue $0 \in \Theta^{-1}$ to satisfy condition (ii) of Theorem \ref{thm:PeakSTchar}. For the `if' part of (ii), we deal with (a) and (b) separately.

Case (a). If $\tau$ is not divisible by $4$, then every $s(\theta)$ is not divisible by $4$, so that every $q(\theta)$ is odd. This implies that $0 \notin \Lambda$. Moreover, as $t$ is even, $t \cdot \frac{p(\theta)}{q(\theta)}$ is even for every $\theta \in \Lambda = \Theta^1$. 
For the remaining eigenvalue $0 \in \Theta^{-1}$, note that $t \equiv 2 \mod 4$ so that $t \cdot \frac{1}{2}$ is odd, as required. We conclude that there is peak state transfer at time $\tau'/2$ by Theorem \ref{thm:PeakSTchar}.

Case (b). Here, $\tau' = \tau$ and $t \equiv 2 \mod 4$. Moreover,
\begin{equation}
\label{eq:peak_st_condition_odd}
t \cdot \frac{p(\theta)}{q(\theta)} = t \cdot \frac{2r(\theta)}{s(\theta)} = \tau \cdot \frac{r(\theta)}{s(\theta)}
\end{equation}
is even for all $\theta \in \Lambda \setminus \{0\} = \Theta^1$. Since $t\cdot \frac{1}{2}$ is again odd, there is peak state transfer at time $\tau'/2$ by Theorem \ref{thm:PeakSTchar}.

For the `only if' part of (ii), assume that (a) and (b) are both false. Then either
\begin{itemize}
\item $\tau \equiv 0 \mod 8$, or
\item $\tau \equiv 4 \mod 8$ and there is some $\theta \in \Lambda$ such that $\tau \cdot \frac{r(\theta)}{s(\theta)}$ is odd.
\end{itemize}
In the former case, $t$ is divisible by $4$ so that $t \cdot \frac{1}{2}$ is even, and there cannot be peak state transfer. In the latter case, we have that there is some other eigenvalue $\theta \in \Lambda \setminus \{0\} = \Theta^1$ for which $t \cdot \frac{p(\theta)}{q(\theta)}$ is odd by \eqref{eq:peak_st_condition_odd}, which then fails to satisfy condition (ii) of Theorem \ref{thm:PeakSTchar}.

Finally, note that
\[
\Lambda \subset \Omega = \Theta^1 \cup \Theta^{-1}
\]
so that if is periodicity at $(u,a)$ or peak state transfer from $(u,a)$ to $(u,b)$ in $G[\overline{K_m}]$, every eigenvalue $\theta \in \Lambda$ can be written as a cosine of a rational multiple of $\pi$, and there must be periodicity by Theorem \ref{thm:periodicity_characterisation}.
\end{proof}

We will give the proof of Lemma \ref{lem:2design_idempotents} below. Recall that, in the lemma statement,  the eigenvalues of the adjacency matrix $A$ of the point-block incidence graph are $0, \pm \sqrt{r-\lambda}$ and $\pm\sqrt{rk}$ and their respective spectral idempotents are
\begin{align*}
E_{0} &= 
\begin{bmatrix}
0 & 0 \\
0 & I - \frac{1}{r-\lambda}\left(N^TN - \frac{\lambda v}{b}J\right)
\end{bmatrix}; \\
E_{\pm \sqrt{r-\lambda}} &= 
\frac{1}{2}
\begin{bmatrix}
I - \frac{1}{v}J & \pm\frac{1}{\sqrt{r-\lambda}} \left(N - \frac{k}{v}J\right)\\
\pm \frac{1}{\sqrt{r-\lambda}} \left(N - \frac{k}{v}J\right) & \frac{1}{r-\lambda}\left(N^TN - \frac{rk}{b}J\right)
\end{bmatrix}
; \quad \text{and}\\
E_{\pm \sqrt{rk}} &= \frac{1}{2}
\begin{bmatrix}
\frac{1}{v}J & \pm \frac{k}{v\sqrt{rk}} J \\
\pm \frac{k}{v\sqrt{rk}} J & \frac{1}{b}J
\end{bmatrix}.
\end{align*}

\begin{proof}[Proof of Lemma \ref{lem:2design_idempotents}]
This can be verified with a standard (but rather tedious) computation: one can show that the $E_\mu$ are pairwise projections that sum to the identity matrix, as well as that $AE_\mu = \mu E_\mu$ for every $\mu$. For this, one needs the identities in \eqref{eq:2design_parameter_relations}, as well as the expressions in \eqref{eq:NNT_expression_JI} and \eqref{eq:N_times_J} to simplify products of the matrices $N$, $N^T$ and $J$.

For an alternative proof, since $X$ is bipartite, one can determine first the spectral decomposition of the squared adjacency matrix
\[
A^2 = \begin{bmatrix}
NN^T & 0 \\
0 & N^TN
\end{bmatrix},
\]
using the spectral decomposition of $NN^T$ given in \eqref{eq:NNT_spectral_decomposition} (note that if $F$ is the spectral idempotent of $NN^T$ for the eigenvalue $\mu \neq 0$, then $\frac{1}{\mu}N^TFN$ is the spectral idempotent for $\mu$ as an eigenvalue of $N^TN$). These spectral idempotents are the following:

\begin{align*}
F_{0} &= 
\begin{bmatrix}
0 & 0 \\
0 & I - \frac{1}{r-\lambda}\left(N^TN - \frac{\lambda v}{b}J\right)
\end{bmatrix}; \\
F_{r-\lambda} &= 
\begin{bmatrix}
I - \frac{1}{v}J & 0 \\
0 & \frac{1}{r-\lambda}\left(N^TN - \frac{rk}{b}J\right)
\end{bmatrix}
; \quad \text{and}\\
F_{rk} &= \begin{bmatrix}
\frac{1}{v}J & 0 \\
0 & \frac{1}{b}J
\end{bmatrix},
\end{align*}
for the eigenvalues $0$, $r - \lambda$ and $rk$ respectively, assuming that $v < b$. Because $X$ is bipartite, the eigenvalues of $A$ are symmetric (with multiplicity) around $0$, so its eigenvalues must be $0$, $\pm \sqrt{r-\lambda}$ and $\sqrt{rk}$. Moreover,
\[
\begin{split}
F_0 &= E_0; \\
F_{r-\lambda} &= E_{\sqrt{r-\lambda}} + E_{-\sqrt{r-\lambda}}; \text{ and} \\
F_{rk} &= E_{\sqrt{rk}} + E_{-\sqrt{rk}},
\end{split}
\]
which allows one to compute the idempotents $E_{\mu}$ as given in the statement of the lemma (note that $\begin{bmatrix}x & y\end{bmatrix}^T$ is an eigenvector for $\mu$ if and only if $\begin{bmatrix}x & -y\end{bmatrix}^T$ is an eigenvector for $-\mu$). The same argument works when $v = b$, except that in that case, the matrix $N^TN$ is non-singular and $0$ is not an eigenvalue.
\end{proof}

Next, we provide the proof of Proposition \ref{prop:4_4grid_periodic}, which states that the vertex-face walk on the $(4,4)$-grid is $12$-periodic and that it does not admit perfect state transfer at any time.

\begin{proof}[Proof of Proposition \ref{prop:4_4grid_periodic}]
Considering the equations in \eqref{eq:vxf_grid_cos_equation}, the walk on the toroidal $(4,4)$-grid is periodic if and only if for every pair of $k,\ell \in \{0,1,2,3\}$ there exists an $s_{k,\ell} \in \rats$ such that
\[
\cos(2\pi s_{k,\ell}) = \frac{1}{2}\left(\cos \frac{ \pi k}{2} + 1\right)\left(\cos \frac{\pi \ell}{2} + 1\right) -1
\]
Evidently, each of these equations has a unique solution $s_{k,\ell}$ in the interval $[0,\frac{1}{2}]$; we only need to make sure that each $s_{k,\ell}$ is rational. Note that by symmetry
\[
s_{k,\ell} = s_{\ell,k} = s_{n-k,\ell}
\]
and also $s_{k,2} = -1$ for every $k$, so it suffices to consider the four pairs $(k,\ell)$ that are given in Table \ref{tab:4_4grid_periodic}. From this table we see that all solutions $s_{k,\ell}$ are indeed rational. Moreover, the smallest integer $\tau$ for which $\tau s_{k,\ell}$ is integer for all $(k,\ell)$ is 12, so the walk is periodic with period 12. (I.e.\ $\tau$ is the least common multiple of the denominators in the $s_{k,\ell}$ column in the table above.) To verify that there is no perfect state transfer, we write every eigenvalue $\cos(2\pi s_{k,\ell})$ in the form
\[
\cos\frac{p_{k,\ell}\pi}{q_{k,\ell}},
\]
where the pairs $(p_{k,\ell},q_{k,\ell})$ represent the $p(\theta)$ and $q(\theta)$ that we need to consider to check the condition (ii) of Theorem \ref{thm:PeakSTchar} are also given in Table \ref{tab:4_4grid_periodic}. We take $\tau = \lcm\{q_{k,\ell}\mid k,\ell \in \{0,1,2,3\}\} = 6$. By Lemma \ref{lem:cycle_eigenvectors}, the spectral idempotents of $B$ are
\[
\begin{split}
E_1 &= F_{4,0} \otimes F_{4,0};\\
E_{0} &= F_{4,0} \otimes F_{4,1} + F_{4,1} \otimes F_{4,0}; \\
E_{-\frac{1}{2}} &= F_{4,1} \otimes F_{4,1};\\
E_{-1} &= I_4 \otimes F_{4,2} + F_{4,2} \otimes I_4 - F_{4,2} \otimes F_{4,2}.
\end{split}
\]
By symmetry of the grid, it suffices to rule out perfect state transfer from the vertex $(0,0)$ to another vertex. Moreover, there is only one candidate vertex to which there could be perfect state transfer when $(0,0)$ is the initial vertex, and that is $(2,2)$. This is because perfect state transfer is a monogamous relation as shown in \cite[Theorem 4.5 (iii)]{GuoSch2024}, and thus, by symmetry, if there is perfect state transfer from $(0,0)$ to $(i,j)$, there must be perfect state transfer to $(\pm i, \pm j)$. Using the expressions for the idempotents above, as well as \eqref{eq:cycle_idempotent_entries}, we find that
\[
\begin{split}
E_1((0,0),(2,2)) &= \frac{1}{16} > 0; \\
E_0((0,0),(2,2)) &= -\frac{2}{4^2} - \frac{2}{4^2} = -\frac{1}{4} < 0; \\
E_{-\frac{1}{2}}((0,0),(2,2)) &= \frac{2}{4}\cdot \frac{2}{4} = \frac{1}{4} > 0; \\
E_{-1}((0,0),(2,2)) &= 0 + 0 - \frac{1}{4}\cdot \frac{1}{4} = -\frac{1}{16} < 0.
\end{split}
\]
We conclude that the positive and negative mutual eigenvalue supports of $(0,0)$ and $(2,2)$ are respectively $\Lambda^{1} = \{-\frac{1}{2},1\}$ and $\Lambda^{-1} = \{-1,0\}$. However, since $\tau \cdot \frac{p_{0,2}}{q_{0,2}}$ is even for the eigenvalue $\lambda_{0,2} = -1$, condition (ii) of Theorem \ref{thm:PeakSTchar} is not satisfied and there is no perfect state transfer from $(0,0)$ to $(2,2)$. (Recall that for the vertex-face walk, the $\gamma$ in Theorem \ref{thm:PeakSTchar} is always equal to $1$.)
\end{proof}

Below, we give the proof of Theorem \ref{thm:4n_grids_peakST}, which classifies peak state transfer in the vertex-face walk on $(4,n)$-grids. The statement specifies three cases, which will be dealt with separately.

\begin{proof}[Proof of Theorem \ref{thm:4n_grids_peakST}]
Using equation \eqref{eq:toroidal_B_eigs}, we determine that the projected transition matrix $B$ for the vertex-face walk on the $(4,n)$-toroidal grid has eigenvalues of the following forms:
\[
\begin{split}
\lambda_{0,\ell} &= \cos\frac{2\pi \ell}{n}; \\
\lambda_{\pm 1, \ell } &= \frac{1}{2}\cos\frac{2\pi \ell}{n} - \frac{1}{2}; \text{ and}\\
\lambda_{2,\ell} &= -1,
\end{split}
\]
where $\ell = 0,\ldots,n-1$. These forms are not mutually disjoint. For instance, it can happen that $\lambda_{0,\ell} = -1$ or $\lambda_{\pm 1,\ell} = -1$ (namely if $n$ is even and $\ell = \frac{n}{2}$). It can also happen that $\lambda_{0,\ell} = \lambda_{\pm 1,r}$ for some $\ell,r$. Indeed, suppose that
\[
\cos\frac{2\pi \ell}{n} = \frac{1}{2}\cos\frac{2\pi r}{n} - \frac{1}{2}.
\]
Then
\[
2\cos\frac{2\pi \ell}{n} - \cos\frac{2\pi r}{n} = -1
\]
By Lemma \ref{lem:conway_jones}(ii), it cannot be that $\cos\frac{2\pi \ell}{n}$ and $\cos\frac{2\pi r}{n}$ are both irrational, so the equation above can only hold if both cosines are rational. Hence the following pairs of $\ell,r$ satisfy the equation:
\[
\begin{array}{c|c|c}
     \ell & r & \text{corresponding eigenvalue} \\ \hline
     \pm \frac{n}{4} & 0 & 0 \\
     \pm \frac{n}{3} & \pm \frac{n}{4} & -\frac{1}{2}\\
     \frac{n}{2} & \frac{n}{2} & -1
\end{array}
\]
Now that we know the overlap in the different forms that the eigenvalues of $B$ can take, we can determine, using Lemma \ref{lem:cycle_eigenvectors}, the spectral idempotents of $B$. If we denote the spectral idempotent corresponding to eigenvalue $\lambda$ by $E_\lambda$, then $B$ always has the following spectral idempotents:
\begin{align}
\nonumber E_1 &= \frac{1}{4n}J_{4n};\\
\label{eq:cos_idempotent1} E_{\cos\frac{2\pi \ell}{n}} &= F_{4,0} \otimes F_{n,\ell} \quad \text{for all } \ell \in \left\{1,\ldots,\left\lfloor\frac{n-1}{2}\right\rfloor \right\} \mathbin{\big\backslash} \left\{\frac{n}{4},\frac{n}{3}\right\}; \text{ and} \\
\label{eq:cos_idempotent2} E_{\frac{1}{2}\cos\frac{2\pi r}{n} - \frac{1}{2}} &= F_{4,1} \otimes F_{n,r} \quad \text{for all } r \in \left\{1,\ldots,\left\lfloor\frac{n-1}{2}\right\rfloor\right\} \mathbin{\big\backslash} \left\{\frac{n}{4}\right\},
\end{align}
where $F_{m,k}$ is defined as in Lemma \ref{lem:cycle_eigenvectors}. Because of the overlap between eigenvalues that we have discussed above, we have left out some values of $\ell$ and $r$. We also take $E_1$ as a special case of $E_{\cos\frac{2\pi \ell}{n}}$ (with $\ell = 0$); note that $J_{4n}$ denotes the $4n \times 4n$ all-ones matrix.

The forms of the remaining spectral idempotents depend on $n$:
\begin{align*}
E_0 &= F_{4,1} \otimes F_{n,0} + \begin{cases}
F_{4,0} \otimes F_{n,\frac{n}{4}} &\text{if $4\mid n$}; \\
0 &\text{otherwise;}
\end{cases} \\
E_{-1} &= F_{4,2} \otimes I_n + \begin{cases}
F_{4,0} \otimes F_{n,\frac{n}{2}} + F_{4,1} \otimes F_{n,\frac{n}{2}} &\text{if $2\mid n$}; \\
0 &\text{otherwise;}
\end{cases}
\end{align*}
and finally, if $-\frac{1}{2}$ is an eigenvalue of $B$, it has the spectral idempotent
\[
E_{-\frac{1}{2}} = \begin{cases}
F_{4,0} \otimes F_{n,\frac{n}{3}} &\text{if $3\mid n$}; \\
0 &\text{otherwise;}
\end{cases} \quad + \quad 
\begin{cases}
F_{4,1} \otimes F_{n,\frac{n}{4}} &\text{if $4\mid n$}; \\
0 &\text{otherwise.}
\end{cases}
\]
We need these spectral idempotents in order to determine the positive and negative mutual eigenvalue supports, so that we can verify the conditions in Theorem \ref{thm:PeakSTchar}. We first make some observations in order to simplify the analysis. By symmetry of the grid, it suffices to show that there is peak state transfer from vertex $(0,0)$ to vertices $(1,x)$, where $x$ takes on the values $0,\frac{n}{2}$ and $\frac{n}{4}$ for the cases (a), (b) and (c) respectively, as specified in the statement of the theorem. We find that the $((0,0),(1,x))$-entry of $F_{4,1} \otimes F_{n,r}$ equals
\begin{equation}
\label{eq:FtimesF_entry_zero}
F_{4,1}(0,1) \cdot F_{n,r}(0,x) = 0\cdot F_{n,r}(0,x) = 0
\end{equation}
for any integer $r$, so that, considering \eqref{eq:cos_idempotent2}, we can see that no eigenvalue of the form $\frac{1}{2}\cos\frac{2\pi r}{n} - \frac{1}{2}$ is in the mutual eigenvalue support of $(0,0)$ and $(1,x)$ and we can disregard these eigenvalues in the analysis. Equation \eqref{eq:FtimesF_entry_zero} also implies  the eigenvalues $0$ and $-\frac{1}{2}$ can be treated together: their $((0,0),(1,x))$-entries depend only on the matrices $F_{4,0} \otimes F_{n,\frac{n}{4}}$ (if $4 \mid n$) and $F_{4,0} \otimes F_{n,\frac{n}{3}}$ (if $3 \mid n$) respectively, so we can treat the relevant entries in the same way as for the other spectral idempotents in \eqref{eq:cos_idempotent1}. We also conclude that every eigenvalue that is potentially in the mutual eigenvalue support of $(0,0)$ and $(1,x)$ can be written as a cosine of a rational multiple of $\pi$. For the eigenvalues $1$ and $-1$, we have that
\[
1 = \cos\frac{0 \cdot \pi}{1} \quad \text{and} \quad -1 = \cos\frac{1 \cdot \pi}{1}.
\]
Since both denominators in the fractions in the cosines are $1$, these eigenvalues both contribute a $1$ to the least common multiple $\tau$ as defined in Theorem \ref{thm:PeakSTchar}(ii). For $\ell \in \{1,\ldots,\left\lfloor\frac{n-1}{2}\right\rfloor\}$,
\[
\cos \frac{2\pi \ell}{n} = \cos\frac{p_\ell\pi}{q_\ell}
\]
where $p_\ell := \frac{2\ell}{\gcd(n,2\ell)}$ and $q_\ell := \frac{n}{\gcd(n,2\ell)}$ are the remaining coprime integers that we need to determine the time of peak state transfer and to verify condition (ii) of Theorem \ref{thm:PeakSTchar}.

To complete the proof, we will discuss the cases (a), (b) and (c) separately. For each case, we need to find the positive and negative mutual eigenvalue supports in order to determine the time of peak state transfer. Note that, as mentioned before, by Remark \ref{rem:gamma=1} we can always take $\gamma = 1$ in condition (ii) of Theorem \ref{thm:PeakSTchar}.
In case (a), $n$ is an odd integer and it suffices to show that there is peak state transfer from vertex $(0,0)$ to vertex $(1,0)$. For all $\ell \in \left\{1,\ldots,\frac{n-1}{2}\right\}$,
\[
E_{\cos\frac{2\pi \ell}{n}}((0,0),(1,0)) = F_{4,0}(0,1)\cdot F_{n,\ell}(0,0) = \frac{1}{4}\cdot \frac{2}{n} \cos(0) = \frac{1}{2n} > 0,
\]
so that $\cos\frac{2\pi \ell}{n}$ is in the positive mutual eigenvalue support of $(0,0)$ and $(1,0)$. For each such $\ell$, we have $\gcd(n,2\ell) = \gcd(n,\ell)$ as $n$ is odd, so that
\[
p_\ell = \frac{2\ell}{\gcd(n,\ell)} \quad \text{and} \quad q_\ell = \frac{n}{\gcd(n,\ell)}.
\]
Further, we find that
\[
E_{-1}((0,0),(1,0)) = F_{4,2}(0,1) \cdot I_n(0,0) = -\frac{1}{4} < 0
\]
so that $-1$ is in the negative mutual eigenvalue support of $(0,0)$ and $(1,0)$. For $n \geq 3$, we take
\[
\tau = \lcm\left\{q_\ell : \ell = 1,\ldots,\frac{n-1}{2}\right\} = n
\]
and for $n = 1$, the value of $\tau$ is trivially also $\tau = 1 = n$. For each $\ell \in \{1,\ldots,\frac{n-1}{2}\}$,
\[
\tau \cdot \frac{p_\ell}{q_\ell} = n\cdot \frac{2\ell}{n} = 2\ell
\]
is even, so that condition (ii)(a) of Theorem \ref{thm:PeakSTchar} holds for these eigenvalues. For the eigenvalue $-1$, we have
\[
\tau \cdot \frac{1}{1} = n
\]
is odd so that (ii)(b) holds instead. We conclude that there is peak state transfer from vertex $(0,0)$ to vertex $(1,0)$ at time $n$.

For case (b), where we let $n = 2m$, we will show that there is peak state transfer from vertex $(0,0)$ to vertex $(1,m)$ at time $m$. As before, $1$ is in the positive mutual eigenvalue support of these vertices and we already know that Theorem 3.5(ii)(a) is satisfied. For $\ell \in \{1,\ldots,m -1\}$:
\begin{align*}
E_{\cos\frac{2\pi \ell}{n}}((0,0),(1,m)) &= F_{4,0}(0,1)\cdot F_{n,\ell}(0,m)\\
&= \frac{1}{2n}\cos\frac{2\pi\ell(0 - m)}{n} \\
&= \frac{1}{2n} \cos\pi\ell \\
&= \begin{cases}
\frac{1}{2n} > 0 &\text{if $\ell$ is even;} \\
-\frac{1}{2n} < 0 &\text{if $\ell$ is odd.}
\end{cases}
\end{align*}
Hence $\cos\frac{2\pi\ell}{n}$ is in the positive (negative) mutual eigenvalue support of $(0,0)$ and $(1,m)$ if $\ell$ is even (odd). This time, since $n = 2m$, we can write
\[
p_\ell = \frac{\ell}{\gcd(m,\ell)} \quad \text{and} \quad q_\ell = \frac{m}{\gcd(m,\ell)}.
\]
Further, we find that
\begin{align*}
E_{-1}((0,0),(1,m)) &= F_{4,2}(0,1) \cdot I_n(0,m) \,+\, F_{4,0}(0,1) \cdot F_{n,m}(0,m) \,+\, F_{4,1}(0,1)\cdot F_{n,m}(0,m) \\
&= 0 \,+\, \frac{1}{2n}\cos\pi m \,+\, 0 \\
&= \begin{cases}
\frac{1}{2n} > 0 &\text{if $m$ is even;} \\
-\frac{1}{2n} < 0 &\text{if $m$ is odd.}
\end{cases}
\end{align*}
We obtain for $n \geq 4$:
\[
\tau = \lcm\left\{q_\ell : \ell = 1,\ldots,m-1\right\} = m
\]
and if $n=2$, we trivially have $\tau = 1 = m$ as well. For each $\ell \in \{1,\ldots,m -1\}$:
\[
\tau \cdot \frac{p_\ell}{q_\ell} = m \cdot \frac{2\ell}{n} = \ell,
\]
which is even if $\cos\frac{2\pi \ell}{n}$ is in the positive mutual eigenvalue support and odd if it is in the negative mutual eigenvalue support, as required. Finally,
\[
\tau \cdot \frac{1}{1} = m,
\]
which is even if $-1$ is in the positive mutual eigenvalue support and odd if it is in the negative eigenvalue support. We conclude that Theorem \ref{thm:PeakSTchar}(ii) is satisfied and there is peak state transfer from vertex $(0,0)$ to vertex $(0,m)$ at time $m$.

Finally, for case (c), assume that $n = 4k$; we want to show that there is peak state transfer from $(0,0)$ to $(1,k)$. Once more, condition (ii)(a) of Theorem \ref{thm:PeakSTchar} is satisfied for the eigenvalue $1$. For $\ell \in \{1,\ldots,2k-1\}$:
\[
E_{\cos\frac{2\pi \ell}{n}}((0,0),(1,k)) = \frac{1}{2n} \cos\frac{\pi\ell}{2} = \begin{cases}
0 &\text{if $\ell$ is odd;} \\
\frac{1}{2n} > 0 &\text{if $\ell \equiv 0 \mod 4$;} \\
-\frac{1}{2n} < 0 &\text{if $\ell \equiv 2 \mod 4$,}
\end{cases}
\]
so that $\cos\frac{2\pi\ell}{n}$ is in the mutual eigenvalue support of $(0,0)$ and $(1,m)$ only if $\ell$ is even and it is in the positive/negative part depending on if $\ell$ is divisible by $4$. Since $n = 4k$, we can write for all $\ell \in \{2,4,\ldots,2k-2\}$:
\[
p_\ell = \frac{\ell/2}{\gcd(k,\ell/2)} \quad \text{and} \quad q_\ell = \frac{k}{\gcd(k,\ell/2)}.
\]
Furthermore, we have
\begin{align*}
E_{-1}((0,0),(1,k)) &= F_{4,2}(0,1) \cdot I_n(0,k) \,+\, F_{4,0}(0,1) \cdot F_{n,2k}(0,k) \,+\, F_{4,1}(0,1)\cdot F_{n,2k}(0,k) \\
&= \frac{1}{2n}\cos\pi k \\
&= \begin{cases}
\frac{1}{2n} > 0 &\text{if $k$ is even;} \\
-\frac{1}{2n} < 0 &\text{if $k$ is odd.}
\end{cases}
\end{align*}
We obtain for all $n\geq 4$:
\[
\tau = \lcm\left\{q_\ell : \ell = 2,\ldots,2k-2\right\} = k.
\]
For each $\ell \in \{2,4,\ldots,2k-2\}$:
\[
\tau \cdot \frac{p_\ell}{q_\ell} = k\cdot \frac{2\ell}{n} = \frac{\ell}{2},
\]
which is even (odd) if $\cos\frac{2\pi\ell}{n}$ is in the positive (negative) eigenvalue support of $(0,0)$ and $(1,k)$, so that Theorem 3.5(ii) is satisfied for these eigenvalues. Finally,
\[
\tau \cdot \frac{1}{1} = k
\]
so that Theorem 3.5(ii) is also satisfied for the eigenvalue $-1$. We conclude that there is peak state transfer from vertex $(0,0)$ to vertex $(1,k)$ at time $k$.
\end{proof}

\end{document}